\documentclass[sn-mathphys,Numbered]{sn-jnl}

\usepackage{graphicx}%
\usepackage{multirow}%
\usepackage{amsmath,amssymb,amsfonts}%
\usepackage{amsthm}%
\usepackage{mathrsfs}%
\usepackage[title]{appendix}%
\usepackage{xcolor}%
\usepackage{textcomp}%
\usepackage{manyfoot}%
\usepackage{booktabs}%
\usepackage{algorithm}%
\usepackage{algorithmicx}%
\usepackage{algpseudocode}%
\usepackage{listings}%

\usepackage{tikz}
\usetikzlibrary{patterns}
\usepackage{pgfkeys}

\pgfdeclarelayer{nodelayer}
\pgfdeclarelayer{edgelayer}
\pgfsetlayers{edgelayer,nodelayer,main}

\tikzstyle{none}=[inner sep=0pt]
\tikzstyle{style0}=[shape=circle, draw=none, fill=black,inner sep=0pt,minimum size=5pt]
\tikzstyle{white}=[shape=circle, draw=none, fill=white,inner sep=0pt,minimum size=20pt]

\tikzstyle{line}=[-]
\tikzstyle{arrow}=[<-]
\tikzstyle{dashline}=[-,dashed]
\tikzstyle{dasharrow}=[<-,dashed]

\tikzstyle{colordomain}=[-,fill={rgb,255: red,191; green,191; blue,191}]

\def\bff{{\bf f}}
\def\bfg{{\bf g}}

\def\bfp{{\bf p}}

\def\bfF{{\bf F}}
\def\bfL{{\bf L}}

\def\bfv{{\bf v}}

\def\bfx{{\bf x}}
\def\bfy{{\bf y}}

\newtheorem{theorem}{Theorem}[section]
\newtheorem{lemma}[theorem]{Lemma}
\newtheorem{claim}[theorem]{Claim}
\newtheorem{proposition}[theorem]{Proposition}
\newtheorem{corollary}{Corollary}[theorem]
\newtheorem{remark}{Remark}[section]
\newtheorem{example}[theorem]{Example}
\newtheorem{definition}{Definition}[section]

\hypersetup{ colorlinks=true, linkcolor=black, filecolor=black, urlcolor=black }

\begin{document}

\title[Article Title]{A Construction of Interpolating Space Curves with Any Degree of Geometric Continuity}


\author*[1]{\fnm{Tsung-Wei} \sur{Hu}}\email{znkt492357816@gmail.com}

\author[2]{\fnm{Ming-Jun} \sur{Lai}}\email{mjlai@uga.edu}

\affil*[1]{\orgdiv{Department of Mathematics}, \orgname{University of Georgia}, \orgaddress{\street{Herty Drive}, \city{Athens}, \postcode{30602}, \state{Georgia}, \country{U.S.A}}}

\affil[2]{\orgdiv{Department of Mathematics}, \orgname{University of Georgia}, \orgaddress{\street{Herty Drive}, \city{Athens}, \postcode{30602}, \state{Georgia}, \country{U.S.A}}}


\abstract{This paper outlines a methodology for constructing a geometrically smooth interpolatory curve in $\mathbb{R}^d$ applicable to oriented and flattenable points with $d\ge 2$. The construction involves four essential components: local functions, blending functions, redistributing functions, and gluing functions. The resulting curve possesses favorable attributes, including $G^2$ geometric smoothness, locality, the absence of cusps, and no self-intersection. Moreover, the algorithm is adaptable to various scenarios, such as preserving convexity, interpolating sharp corners, and ensuring sphere preservation.
The paper substantiates the efficacy of the proposed method through the presentation of numerous numerical examples, offering a practical demonstration of its capabilities.}




\maketitle

\section{Introduction}\label{sec:intro}

Curves are fundamental geometric objects with applications in various fields that require different types of curve construction and design. The study of constructing smooth interpolatory space curves has been ongoing for many decades. In this paper, we are interested in constructing space curves in $\mathbb{R}^d$ with $d\ge 3$, aiming to satisfy the following requirements:
\begin{itemize}
\item interpolation, that is, the curve $\Gamma$ passes the given designed points;
\item smoothness, that the curve must be geometrically smooth $G^r$ with $r\ge 1$, We shall present our construction of $G^2$ curves in this paper.  
\item locality, i.e., any local change of the given points only affects the curve locally. 
\item no cusp points and no self-intersection unless there is such a need; 
\item convexity control capability;
\item conic preservation, i.e. if a part of the given data points is from a sphere, the curve passing these points is also on the sphere.
\item invariant under the translation and rotation of the data points;
\item corner points are also possible. In particular, the sharpness of the corner points is controllable. 
\end{itemize}

Many approaches are available in the literature and practice to construct smooth space curves due to the active study in the last several decades. However, constructive methods satisfying all the above requirements are
few, The requirement of the high order 
smoothness, say $C^2$ is one of 
the main hurdles for interpolating curve construction. 
The high-order smoothness needs a lot of
more parameters than the construction of $C^1$ curves and hence the parameters are hard to control
while easy to produce the cusps and/or overshoots in the curve. 
It is well-known that the smoothness of the parametric 
functions for a curve does not ensure the smoothness of the graph of the curve or a geometrically continuous 
curve. 
An example will be presented in Example ~\ref{ex: r1notg1}. Typically, the distinction between the "smoothness of graph" and "smoothness of function" is made using the concept of geometric continuity (cf. \cite{D85} \cite{gregory_geometric_1989} \cite{garrity_geometric_1991}). In this paper, we define the geometric continuity of a curve by considering the regularity of its parametric functions. Please refer to Section ~\ref{sec:GeoCont} for the formal definition.

Let us first present two examples of curves
based on our construction in Figure~\ref{fig:3dC2} and Figure~\ref{fig:plane}.  We give sets data set $\{ \bfv_1, \bfv_2, \cdots, \bfv_N \}$ in $\mathbb{R}^3$ which are oriented data sets, and the data sets are flattenable to be explained in Definition~\ref{flat2D}. These figures show that our construction works nicely for various sizes of data points. In Figure~\ref{fig:3dC2}, we have 20 points while in Figure~\ref{fig:plane}, we have 2776 space points.  
More examples will be given in a later section to demonstrate that our construction can preserve the corners, preserve circle or sphere when the data points are on a circle or on a sphere.

\begin{figure}
    \centering
    \includegraphics[width=1\textwidth]{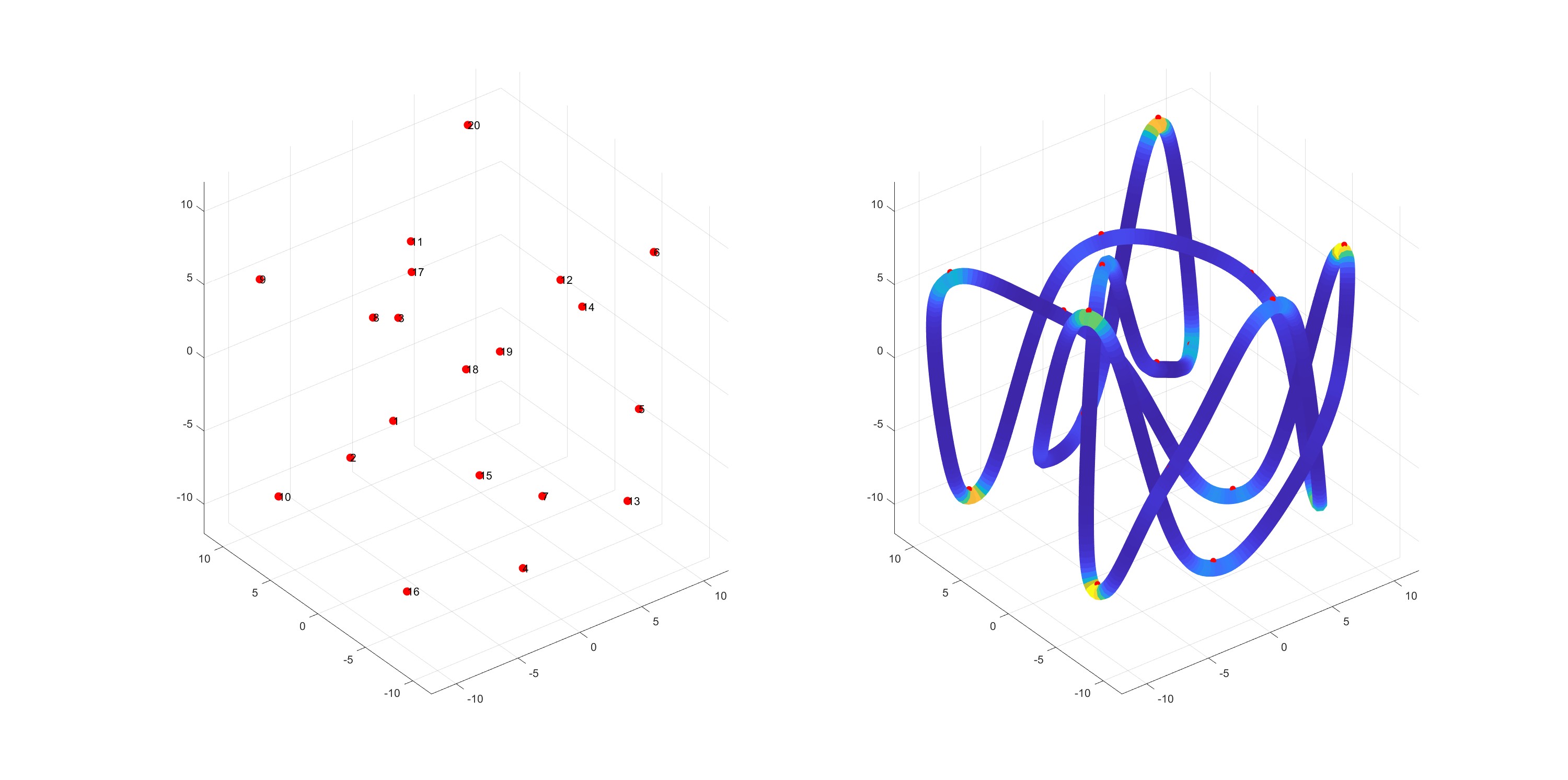}
    \caption{A $C^2$ smooth interpolating curve(right) based on a Lissajous knot (left). 
    The curvature of the curve is shown by the color on the curve. The curve change shows that the curvatures are all continuous.}
    \label{fig:3dC2}
\end{figure}

\begin{figure}
    \centering
    \includegraphics[width=1\textwidth]{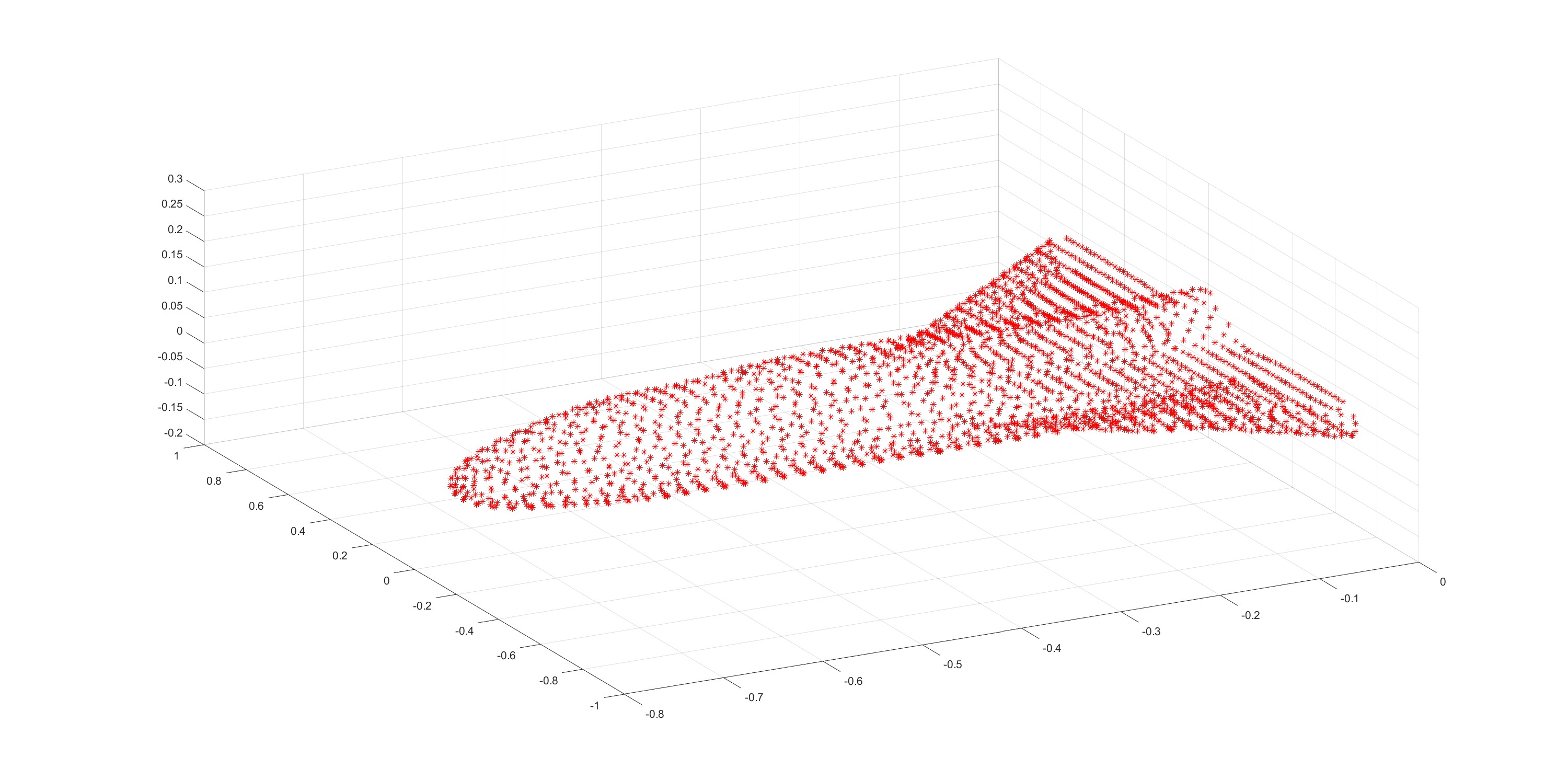}
     \includegraphics[width=1\textwidth]{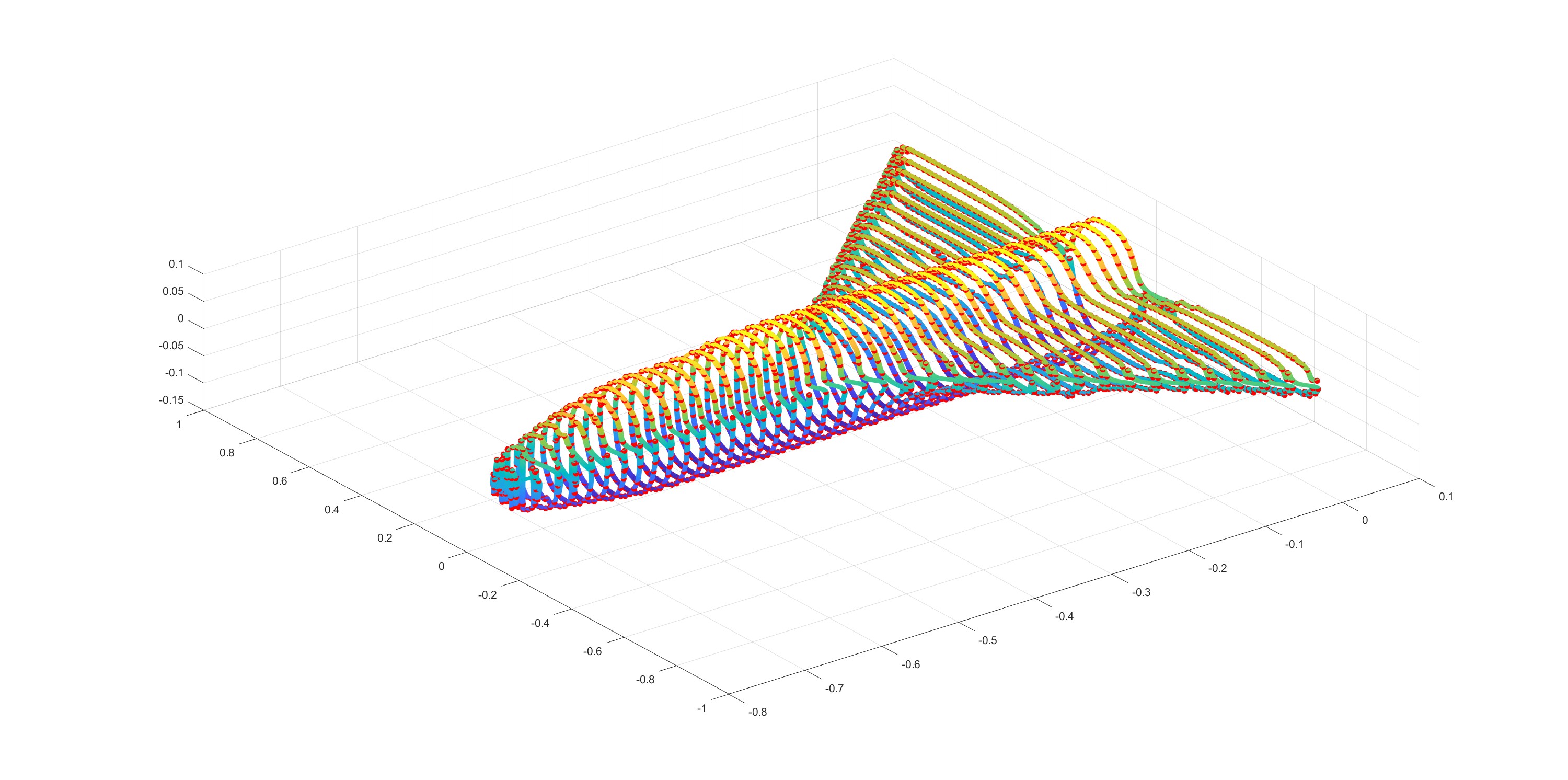}
    \caption{A $C^2$ smooth interpolating curve(bottom) based on a plane data (top). }
    \label{fig:plane}
\end{figure}

Our method of construction 
shares similar steps with previous methods in the literature. 
More precisely, our algorithm can be decomposed into four parts: local curve construction, 
redistribution, blending function construction, and gluing two pieces together. Let us quickly outline 
each of the four components as follows.  
\begin{itemize}

\item {\bf Local Curves}
 Consider the series of points $\{ \bfv_1, \bfv_2, \cdots, \bfv_N \} \subset \mathbb{R}^n$, we define a local curve $\bff_i(t)$ be a the curve pass $\bfv_{i-1}, \bfv_i, \bfv_{i+1}$, where $\bff_i(t_j)=\bfv_j$ for $j=i-1, i, i+1$. There are few canonical options for local curves: parabola, circle, or B\'ezier curves. We shall explain them in much more detail in Section ~\ref{Local_QRP}. 
\item {\bf Redistribution}
Once we have a local curve (function) $\bff_i(t)$, the reparametrization of the function will not change the graph. Therefore, we can reparametrize the curve in a conventional way. One of the convenience ways is let $t(s)$ where $t(j)=t_j$ for $j=i-1, i, i+1$, and $t(s)$ is linear on $[i-1, i]$ and $[i, i+1]$. We define ${\bf F}_i(s)= \bff(t(s)) : [i-1,i+1] \to \mathbb{R}^n$.  Note that $t(s)$ is piecewise linear, so ${\bf F}_i$ is only $C^0$, but since it parametrizes the same curve as $\bff(t)$, so ${\bf F}_i(s)$ is still has a smooth graph.
\item {\bf Blending Function}
For short, a blending function is a pair of functions $B_1, B_2: [0,1] \rightarrow [0,1]$ with $B_1(0)=B_2(1)=0$, $B_1(1)=B_2(0)=1$, and $B_1(s)+B_2(s)=1$ for all $s \in [0,1]$. One can understand it by weighted functions of endpoints $0$ and $1$. We shall impose an 
additional 
assumptions on the derivative in need for the smoothness, see details in Section~\ref{sec:blending}.
\item {\bf Gluing Two Pieces Together}
After redistribution, $\bfF_i$ and $\bfF_{i+1}$ have over lapping domain $[i, i+1]$. We can define a new function
\[ 
\Gamma(s) =
\begin{cases} 
{\bf F}_i(s) & s \in [i, i-1) \\
{\bf F}_i(s)B_1(s-i)+{\bf F}_{i+1}(t)B_2(s-i) & s \in [ i, i+1] \\
{\bf F}_{i+1}(s) & s \in (i+1,i+2] 
\end{cases}
\]
More generally, if we have a continuous function ${\bf \phi}(s,u): [i,i+1] \times [0,1] \to \mathbb{R}^n$ with ${\bf \Phi}(s,1)=\bfF_i$ and ${\bf \Phi}(s,0)=\bfF_{i+1}$, we can gluing the function as  
\[ 
\Gamma(s) =
\begin{cases} 
{\bf F}_i(s) & s \in [i, i-1) \\
{\bf \Phi}(s,B_1(s-i)) & s \in [ i, i+1] \\
{\bf F}_{i+1}(s) & s \in (i+1,i+2] 
\end{cases}
\]
This is particularly useful if we know some extra features of the original data. For instance, in the section on numerical 
experimental results, we will have two examples to show when the point data are on the sphere, we can make sure our interpolating curve 
is on the same sphere as well by choosing a suitable $\bf \Phi$.\\
\end{itemize}

\subsection{Related works}

The realm of curve construction is vast, and a complete overview is not within the scope of this paper. Here, we present a selection of essential works that are more related to our work. Farin ~\cite{FARIN1993} provides a thorough organization of curve interpolation methods using $B$-splines in his book. He also discusses B\'ezier curves, a group of smooth curves based on polynomials. Two primary challenges associated with B\'ezier curves are the need for a local shape control algorithm adaptable to various cases and the difficulty in achieving a higher order of continuity due to the flexibility in choosing control points. 
To achieve a local shape control, various methods are employed to interpolate curves. One example is the Non-uniform Rational B-spline (NURB) ~\cite{PT97}, which utilizes a rational polynomial as a substitute for the conventional polynomial. This substitution provides a broader range of shape options with a reduced number of control points. Addressing the specific need for circle interpolation, a blending technique of the circular functions was introduced ~\cite{WENZ1996}. Notably, Yuksel~\cite{Yuksel2020} made significant advancements by replacing the local circular function with an ellipse function when the given data points deviate from a circle. Another noteworthy approach involves the use of clothoids to interpolate curves, offering a method for achieving curves with uniform curvature changes~\cite{BERTOLAZZI2018}. While these algorithms enhance curve fitting in diverse scenarios, it's important to note that they may introduce increased curve complexity, and the highest achievable smoothness of these algorithms are limited to $G^2$.\\
 Some approaches center around polynomials and aims to enhance smoothness. One approach is to interpolate the curve independently along different axes. Let $\bfv_i=(x_i, y_i, z_i)$, one can interpolate $x_i$, $y_i$, and $z_i$ by $X(t)$, $Y(t)$, and $Z(t)$, where $( X(t_i), Y(t_i), Z(t_i))= (x_i, y_i, z_i)$ for $t_1 < t_2 < \cdots < t_i < \cdots <t_N$. Further details are expounded in the book by Knott ~\cite{Knott2000}. Subsequent to this, Lee refined the selection of $t_i$ \cite{LEE1989363} and made the curve more appealing. Key challenges associated with interpolating along different axes include the absence of locality and a deficiency in rotational stability.\\
 
Our algorithm is grounded in the notion of blending functions. By introducing the $r$-blending function, one can craft an interpolation function with a desired shape and achieve any desired level of smoothness. We redefine geometric continuity in an alternative manner and provide a robust methodology for generating curves without cusps or self-intersections.

\subsection{Organization of the paper}
The paper is organized as follows. We will 
explain the concept of the oriented data point
set and points that are flattenable. Then, we explain the concept of geometric 
continuity and the geometrical smoothness of 
the graph of a curve. These will be done 
in Section~\ref{sec:Prelim}. we explain our constructive
steps with emphasis on the four components that we explained in the introduction. We will give a very general definition of these components and provide a detailed justification of how smooth these curves are. We also use the ideas to create a few new categories of curves that achieve any given smoothness while consisting of piecewise polynomials or preserving the shape of a circle and/or sphere. 

\section{Preliminaries}\label{sec:Prelim}
In this section, we first explain the concepts of oriented point sets 
and flattenable point sets.  Then, we explain the regularity of a graph and geometric continuity/smoothness.

\subsection{Oriented Set of Points and Flattenable Points}
Consider the set of space points ${\bf v}_i, i=1, \cdots, N$ in $\mathbb{R}^n$ which is oriented in the sense that ${\bf v}_i$ is connected to ${\bf v}_{i+1}$
for all $i=1, \dots, N-1$ in the setting of space curves. We call it sometimes a 1D-oriented point set, sometimes an oriented curve data set, or 
sometimes a curve data set, just for simplicity. Next, we introduce the concept of flattenable points. 
\begin{definition}\label{flat2D}
	An 1D-oriented point set $\{\bfv_i, i=1, \cdots, N\} \subset \mathbb{R}^n$ is said to be 
	flattenable if for any  three 
 consecutive points ${\bf v}_{i-1}, {\bf v}_i, {\bf v}_{i+1}$ 
	there exist a line $L_i$ in $\mathbb{R}^n$ such 
 that the projections of these three points 
	$\bfv_{i-1}, \bfv_i, \bfv_{i+1}$ to $L_i$ preserve the order of these three points, 
 i.e. the projection
	of ${\bf v}_i$ is in the middle of the projections of ${\bf v}_{i-1}$ and ${\bf v}_{i+1}$.
\end{definition}

\subsection{Regularity of the Graph of Curves}
Consider a curve consisting of points in $\mathbb{R}^n$ which are defined by one variable parametric functions. 
For instance, a 
non-self intersects curve in $\mathbb{R}^n$ can be defined by 
$\Gamma(t)=(x_1(t),\ldots, x_n(t))$, where $t \in [a,b]$ for an interval $[a,b]$ and $x_i(t), i=1, \cdots, n$ are 
coordinate functions of $\mathbb{R}^n$, where $n\ge 2$. We call $\Gamma(t)$ a {\bf parametrization function} of the arc, and the arc which it parameterized is called {\bf the graph} of $\Gamma$. Obviously, if $\Gamma(t)$ 
is continuous, then its graph is continuous. However, the parametrization functions' smoothness is insufficient for the graph to be smooth. Let us look at the following example.
\begin{example} \label{ex: r1notg1}
	Let $\Gamma(t)= (t^2, t^3)$ for $t\in [-1, 1]$.  The parametric functions are differentiable, but the graph of $\Gamma(t)$ is not. As shown in Figure~\ref{ex1}, the
	curve has a cusp at $t=0$. That is,  the smoothness of the parametric functions did not guarantee that the graph of the parametric curve looks smooth.
	\begin{figure}
		\begin{picture}(250,150)
		\unitlength=0.008in
		\thicklines 
		\put(130,100.5){\line(1,0){20}}
		\put(130,99.5){\line(1,0){20}}
		\qbezier(150,100)(200,110)(300,200)
		\qbezier(150,100)(200,90)(300,0)
		\end{picture}
		\caption{The differentiability of the parametric functions of $\Gamma(t)=(t^2,t^3), t\in
			[-1, 1]$ 
			does not imply the graph is smooth. \label{ex1}}
	\end{figure}
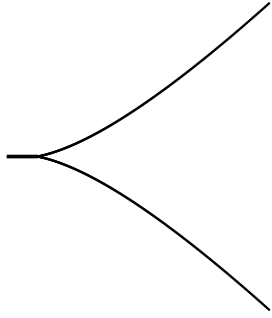
\end{example}
In the Example \ref{ex: r1notg1}, the cusp happens at $t=0$ where the tangent vector $\Gamma'(0)$ vanishes. Indeed, if the graph looks smooth, it should at least have a continuous tangent. However, the $C^1$ of its parametric function does not imply this property as we saw in the Example ~\ref{ex: r1notg1}. To distinguish the ideas, we call a graph of a parametric function has a {\bf geometric continuity} of degree 1 (denoted as $G^1$) if it has tangent lines on each point on it. If we think about the sufficient condition for a parametric function to have a $G^1$ graph, we use  the following definition:
\begin{definition}
	A $C^1$ parametrization function $\Gamma(t)=(x_1(t),\cdots , x_n(t))$ is said to be regular at $t \in (a,b)$ if $\|\Gamma'(t) \|:=\|(x_1'(t),\cdots , x_n'(t))\| > 0$.
\end{definition}
The vanish of the first derivative of $\Gamma$ gives us the graph without continuous tangent. We can easily see the following property is true.
\begin{proposition}
    If a pasteurization function $\Gamma(t): [a,b] \to \mathbb{R}^n$ is regular, then its graph is $G^1$.
\end{proposition}
We may now extend the ideas to the high order of geometric smoothness.  In addition, if $\Gamma(t)$ is  
$C^2$ continuously differentiable in parameter $t$, its curvature which can be formulated as \cite{carmo_differential_2016}:
\begin{equation}
\label{curvature}
\kappa(t)=\dfrac{\sqrt{\|\Gamma'(t)\|^2 \|\Gamma''(t)\|^2 -\|\Gamma'(t)\cdot \Gamma''(t) \|^2}}{\|\Gamma'(t)\|^3}
\end{equation}
Since $\Gamma'(t)\not=0$, the curvature is well defined. If the curvature is continuous, we can say that the graph of $C^2$ parametric curve has continuous tangent and curvature. In general, the non-vanish of the first derivative is sufficient for a $C^r$ curve that has geometric continuity up to order $r$.\\

To summarize, the regularity of the parameterization makes the graph have geometric continuity up to the smoothness of the parameterization. We remark that:
\begin{remark}
	A $C^1$ regular curve has continuous tangents on its graph. A $C^2$ regular curve has continuous curvatures of its graph.
\end{remark}

\begin{figure}
\centering
\begin{tikzpicture}[scale=0.40]
	\begin{pgfonlayer}{nodelayer}
		\node [style=none] (0) at (-9, 1) {};
		\node [style=none] (1) at (-4.5, 1) {};
		\node [style=none] (2) at (-4.5, 1) {};
		\node [style=none] (3) at (-2, 2.75) {};
		\node [style=none] (4) at (1, 2.75) {};
		\node [style=none] (5) at (3.25, 2.25) {};
		\node [style=none] (6) at (5, 4) {};
		\node [style=none] (7) at (8.5, 2.25) {};
		\node [style=none] (8) at (10.75, 2.75) {};
		\node [style=none] (9) at (-2, -1.5) {};
		\node [style=none] (10) at (1, -1.5) {};
		\node [style=none] (11) at (3.5, -3) {};
		\node [style=none] (12) at (5.75, -0.5) {};
		\node [style=none] (13) at (8.25, -3) {};
		\node [style=none] (15) at (-7, 5.5) {\bf Parametric Space};
		\node [style=none] (16) at (7, 5.5) {\bf Graph};
		\node [style=none] (17) at (7, 3.5) {\tiny $G^r$};
		\node [style=none] (18) at (5.75, 0) {\tiny Not nessasary be $G^r$};
		\node [style=none] (19) at (-0.5, 3.25) {\tiny Regular $C^r$};
		\node [style=none] (20) at (-0.5, -0.75) {\tiny Only $C^r$};
	\end{pgfonlayer}
	\begin{pgfonlayer}{edgelayer}
		\draw (0.center) to (2.center);
		\draw [style=arrow] (4.center) to (3.center);
		\draw [style=line, in=165, out=75] (5.center) to (6.center);
		\draw [style=line, in=-165, out=-15] (6.center) to (7.center);
		\draw [style=line, in=150, out=15] (7.center) to (8.center);
		\draw [style=arrow] (10.center) to (9.center);
		\draw [style=line, in=-90, out=0, looseness=1.25] (11.center) to (12.center);
		\draw [style=line, in=-180, out=-90, looseness=1.25] (12.center) to (13.center);
	\end{pgfonlayer}
\end{tikzpicture}
\caption{The graph of a curve is not necessary to be $G^r$ without the regularity of the parametric functions.}\label{fig:rmkGr}
\end{figure}
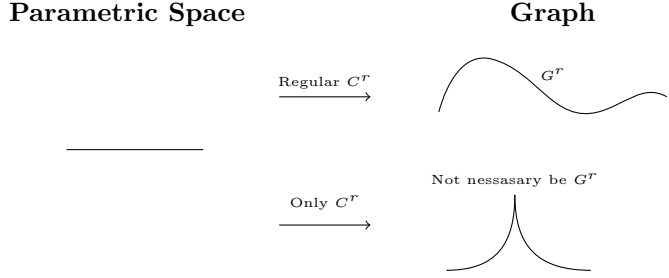

In the end of this subsection, let us present an example that a linear combination of two regular curve
pieces may not be regular any more. 
\begin{example}\label{Weitedsum_of_regular_curve_maynot_regular}
	Let ${\bf F}(t)=( 0.5+4(t-0.5)^3, t(1-t)^2)$, ${\bf G}(t)=( 0.5+4(t-0.5)^3, t^2(1-t))$ and 1-blending function $B_1=1-B_2=1+2t^3-3t^2$ with $B_2=-2t^3+3t^2$. ${\bf F}$ and ${\bf G}$ are $C^r$ regular curve with $r\ge 1$ 
 which connect two points $(0,0)$ and $(1,0)$ for $t \in [0,1]$. However, the weighted sum ${\bf F}(t)B_1(t)+{\bf G}(t)B_2(t)$ is not regular and not $G^1$ at $t=0.5$. See the graph on the right of Figure~\ref{ex2}. \\

\begin{figure}[thpb]
	\includegraphics[width=0.3\textwidth]{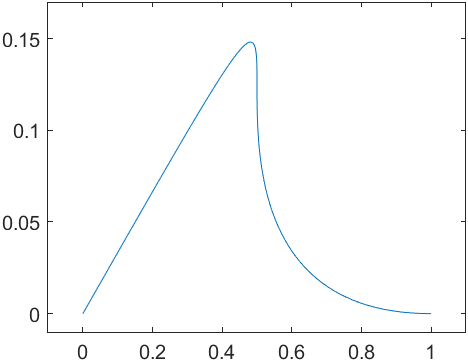}
	\includegraphics[width=0.3\textwidth]{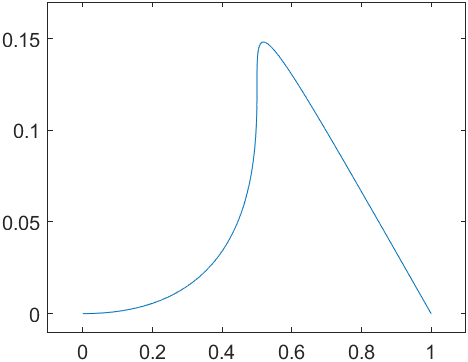}
	\includegraphics[width=0.3\textwidth]{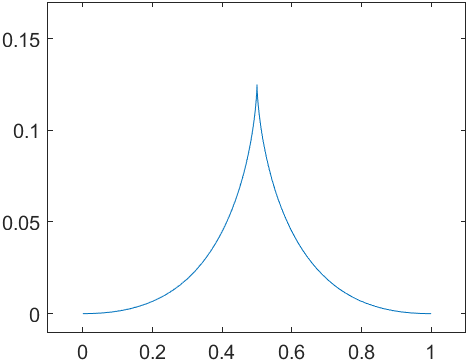}
 \caption{A linear combination of two regular curves may not be regular }\label{ex2}
 \end{figure}
\end{example}

\subsection{Connection of Geometric Continuity Curves}\label{sec:GeoCont}
For curves, the regularity of its parametric function can guarantee nice-looking graphs. However, in most practical
applications, joining multiple curves with different parametric functions is required to express a long curve. 
That is, we may not have a global regular parametric function for the entire curve. It is possible that 
we may still have a nice-looking entire curve even if the parametric functions of each piece are not joined smoothly. Let us see the following example.
\begin{example}\label{ex:liseg}
	Consider a parametric curve defined below. 
	\[ \Gamma(t) =\begin{cases} 
	(t , t) & t \in [-1,0] \\
	(3t, 3t) & t \in (0,1].  
	\end{cases}
	\]
	The graph of $\Gamma(t)$ is a line segment from $(-1,-1)$ to $(3,3)$ in $\mathbb{R}^2$ which joins smoothly at $t=0$. However, the parametric function itself is not differentiable at $t=0$.  
 Nevertheless, we can define a bijection $t=S(s)$, where
	\[ S(s) =\begin{cases} 
	s & s \in [-1,0] \\
	\frac{s}{3} & t \in (0,3] 
	\end{cases}
	\]
	such that $\Gamma(S(s))=s$ for $s \in [-1,3]$. We can see that $\Gamma \circ S$ is now a regular parametric function even though $\Gamma$ and $S$ are not smooth at $t=S(0)=0$. 
\end{example}

We see the existence of a regular parametric function is crucial to the geometrical smooth feature. We are now ready to formally define $G^r$ smoothness for $r\ge 1$ for a curve consisting of multiple pieces.   

\begin{definition}\label{GrCUrve} [cf. \cite{HUTSUNG-WEI2023CMfS}]
	Fix $r\ge 1$. 	A graph $\Gamma$ of curve in $\mathbb{R}^n$ is said to be $G^r$ continuity at a point ${\bf p} \in \Gamma$ if there exist a regular $C^r$ parametrization of functions 
	$x_i(t),i=1, \cdots, n$ and a constant	$\epsilon>0$ such that  $\Gamma(t)=(x_1(t), \cdots, x_n(t)): t\in (- \epsilon, \epsilon) \rightarrow \Gamma$ with $\Gamma(0)={\bf p}$ and 
	$\nabla \Gamma(0):=\Gamma'(0)$ is a nonzero vector, i.e. $\|\nabla \Gamma(0)\|> 0$. 
	An entire curve $\Gamma$ in $\mathbb{R}^n$ is said to be $G^r$ continuous if the curve is of $G^r$ continuity at each point on the curve.
\end{definition}
\begin{figure}
\centering
\begin{tikzpicture}[scale=0.40]
	\begin{pgfonlayer}{nodelayer}
		\node [style=none] (0) at (5, 3.75) {};
		\node [style=none] (1) at (9.5, 3.75) {};
		\node [style=none] (2) at (9.5, 3.75) {};
		\node [style=none] (3) at (-0.75, 3.75) {};
		\node [style=none] (4) at (2.25, 3.75) {};
		\node [style=none] (5) at (-9.5, 1.75) {};
		\node [style=none] (6) at (-7.75, 3.5) {};
		\node [style=none] (7) at (-4.25, 1.75) {};
		\node [style=none] (8) at (-2, 2.25) {};
		\node [style=none] (11) at (-9.25, -3.5) {};
		\node [style=none] (12) at (-7, -1) {};
		\node [style=none] (13) at (-4.5, -3.5) {};
		\node [style=none] (15) at (-7, 5.5) {\bf Graph};
		\node [style=none] (16) at (7, 5.5) {\bf Parametric Space};
		\node [style=none] (17) at (-6, 3.5) {\tiny $G^r$};
		\node [style=none] (18) at (-7, -0.5) {\tiny Not $G^r$};
		\node [style=none] (19) at (0.75, 4.25) {\tiny Existence of regular $C^r$};
		\node [style=none] (21) at (5, 2) {};
		\node [style=none] (22) at (9.5, 2) {};
		\node [style=none] (23) at (9.5, 2) {};
		\node [style=none] (24) at (-0.75, 2) {};
		\node [style=none] (25) at (2.25, 2) {};
		\node [style=none] (26) at (0.75, 2.5) {\tiny $C^0$};
		\node [style=none] (27) at (5, -2) {};
		\node [style=none] (28) at (9.5, -2) {};
		\node [style=none] (29) at (9.5, -2) {};
		\node [style=none] (30) at (-0.75, -2) {};
		\node [style=none] (31) at (2.25, -2) {};
		\node [style=none] (32) at (0.75, -1.0) {$\substack{\text{\tiny Can be } C^r , \\ \text{\tiny but not regular}}$};
	\end{pgfonlayer}
	\begin{pgfonlayer}{edgelayer}
		\draw (0.center) to (2.center);
		\draw [style=line, in=165, out=75] (5.center) to (6.center);
		\draw [style=line, in=-165, out=-15] (6.center) to (7.center);
		\draw [style=line, in=150, out=15] (7.center) to (8.center);
		\draw [style=line, in=-90, out=0, looseness=1.25] (11.center) to (12.center);
		\draw [style=line, in=-180, out=-90, looseness=1.25] (12.center) to (13.center);
		\draw [style=arrow] (3.center) to (4.center);
		\draw (21.center) to (23.center);
		\draw [style=arrow] (24.center) to (25.center);
		\draw (27.center) to (29.center);
		\draw [style=arrow] (30.center) to (31.center);
	\end{pgfonlayer}
\end{tikzpicture}
\caption{Only the Graph with $G^r$ Property can have a regular Parametric Function, While the non $G^r$ Graph can}\label{fig:rmkGr2}
\end{figure}
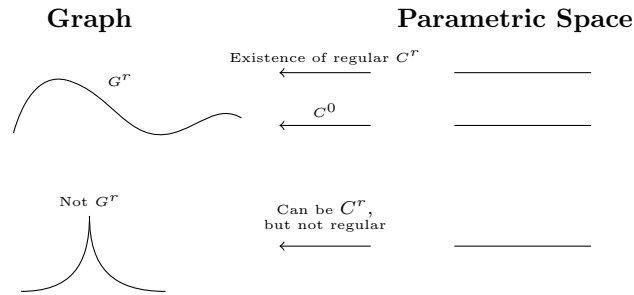

In short, the $G^r$ continuity is a property of a curve-like subset $\Gamma$ in $\mathbb{R}^n$. Regularity and $C^r$ are properties of the underlying parametric functions
of the curve. Figure ~\ref{fig:rmkGr2} shows the idea of the Definition ~\ref{GrCUrve}. Let us summarize the above discussion in the following remark.
\begin{remark}
\label{rem1}
\hfill
	\begin{itemize}
		\item A curve with $C^r$ parametric function does not imply the graph of the curve is $G^r$. See Example \ref{ex: r1notg1}.
		\item If a curve with $C^r$ parametric function is regular, the graph of the curve is $G^r$. 
		\item The graph $\Gamma$ is of $G^r$ does not imply its parametric function is regular or $C^r$. See Example \ref{ex:liseg}. However, if $\Gamma(t)$ is a continuous function, for any $t=t_0$, there exist a bijection function $t=S(s)$, which defined on small interval $(- \epsilon, \epsilon)$ with $S(0)=t_0$, such that $\Gamma(S(s))$ is regular and $C^r$ on $(- \epsilon, \epsilon)$.
	\end{itemize}
\end{remark}
In other words, a bad graph (not $G^r$) may have a $C^r$ smooth parametric function, but it can not have a regular $C^r$ smooth function. Vice versa, a $G^r$ graph could adapt different parametric functions $\bff_1$ and $\bff_2$, see Example ~\ref{ex:liseg}. Figure ~\ref{fig: fgrelation} shows the relation between the function and graph. It's unnecessary to describe a $G^r$ smooth curve with a smooth parametric function. However, there must exist one regular $C^r$ function which parametrizes the curve.

\begin{figure}
    \centering

\begin{tikzpicture}[x=0.75pt,y=0.75pt,yscale=-1,xscale=0.9]

\draw   (190,176) .. controls (190,148.39) and (240.82,126) .. (303.5,126) .. controls (366.18,126) and (417,148.39) .. (417,176) .. controls (417,203.61) and (366.18,226) .. (303.5,226) .. controls (240.82,226) and (190,203.61) .. (190,176) -- cycle ;
\draw    (54,219) -- (149,219) ;
\draw    (100,185) -- (100,126) ;
\draw [shift={(100,124)}, rotate = 90] [color={rgb, 255:red, 0; green, 0; blue, 0 }  ][line width=0.75]    (10.93,-3.29) .. controls (6.95,-1.4) and (3.31,-0.3) .. (0,0) .. controls (3.31,0.3) and (6.95,1.4) .. (10.93,3.29)   ;
\draw    (62,97) .. controls (102,67) and (122,127) .. (162,97) ;
\draw   (254,134) .. controls (254,74.35) and (339.29,26) .. (444.5,26) .. controls (549.71,26) and (635,74.35) .. (635,134) .. controls (635,193.65) and (549.71,242) .. (444.5,242) .. controls (339.29,242) and (254,193.65) .. (254,134) -- cycle ;
\draw   (278,174) .. controls (278,159.64) and (305.53,148) .. (339.5,148) .. controls (373.47,148) and (401,159.64) .. (401,174) .. controls (401,188.36) and (373.47,200) .. (339.5,200) .. controls (305.53,200) and (278,188.36) .. (278,174) -- cycle ;

\draw (80,146.4) node [anchor=north west][inner sep=0.75pt]    {$f$};
\draw (92,64.4) node [anchor=north west][inner sep=0.75pt]    {$\Gamma$};
\draw (206,170.4) node [anchor=north west][inner sep=0.75pt]    {$f\ \ is\ C^{r}$};
\draw (284,166.0) node [anchor=north west][inner sep=0.75pt]    {$f\ \ is\ regular\ C^{r}$};
\draw (269,96.4) node [anchor=north west][inner sep=0.75pt]    {$\Gamma\ is\ G^{r} \ \Leftrightarrow \ f\ can\ be\ reparamertrize\ as\ regular\ C^{r}$};
\end{tikzpicture}
    \caption{The Relation Between Parameterize Function $f$ and its Gragh $G$.}
    \label{fig: fgrelation}
\end{figure}
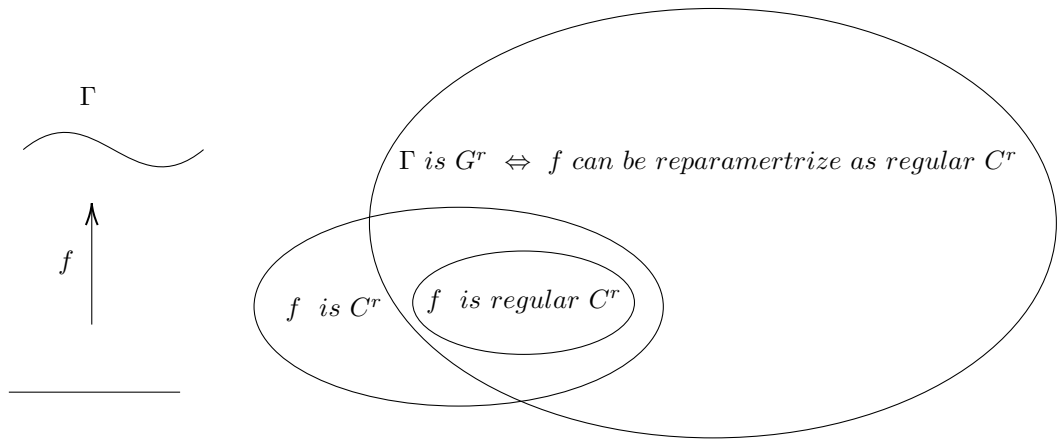

\section{Construction of Geometric Smooth Interpolatory Curves}
This section is divided into five subsections. In the first four subsection, we explain our construction. 
We shall suppose that a point cloud ${\bf P}=\{\bfv_i, i=0,1, \cdots, N\}$ in $\mathbb{R}^3$ is oriented and is flattenable. 
We start the construction with the local curves, reparametrizing them (i.e. redistribution), constructing a blending function, and gluing them together.  
In addition to these four subsections, we 
shall show that the regularity of the curves based on our construction in the 
fifth subsection. 

\subsection{ Construction of Local  Curves}\label{Local_QRP}
Given 3 ordered point $ \bfv_{i-1}, \bfv_i, \bfv_{i+1}$ in ${\bf P}$, we want to construct a $C^r$ regular curve interpolating these three points. 
\begin{definition}\label{def:localcurve}
 A function $\bff_i: [p_i,r_i] \to \mathbb{R}^3$ is said to be a local $G^r$ interpolating function of ${\bf P}$ (at point $\bfv_i$) if $\bff_i$ is regular $C^r$ function and there exist a point $q_i \in (p_i, r_i)$ such that $\bff_i(q_i)=\bfv_i$. \\
\end{definition}
Note that we do not require $\bff_i(p_i)=\bfv_{i-1}$ or $\bff_i(r_i)=\bfv_{i+1}$ during our construction and we will still get an interpolating curve after gluing them together. However,  having local curves pass $\bfv_{i-1}$ and $\bfv_{i+1}$ gives us more pleasing curves
in most cases. 
Let us construct a few different types of local interpolating curves in the following examples. 
\begin{example}[Spatial Parabola]\label{ex:para}
Suppose $\{ \bfv_{i-1}, \bfv_i, \bfv_{i+1} \}$ are three successive points in a flattenable data set.  We define $\tilde{x}$ and $\Tilde{y}$ are two vectors which $\bfp_0+ span\{ \tilde{x},\tilde{y} \}$ is an plane that contains $ \bfv_{i-1}, \bfv_i, \bfv_{i+1}$. Since  $\bfv_{i-1}, \bfv_i, \bfv_{i+1}$ is flatenable, one may choose $\tilde{x}$ such that the projections of $ \bfv_{i-1}, 
\bfv_i, \bfv_{i+1}$ to the $\Tilde{x}$-axis preserve the order, says $p_{i},q_{i}, r_{i}$. Moreover, we can interpolate 
$ \bfv_{i-1}, \bfv_i, \bfv_{i+1}$ by $\{\bfp_0+ s \tilde{\bfx} + Q(s) \tilde{\bfy} \, | \, s \in 
[p_{i},r_{i}] \}$, where $Q(s)$ is a quadratic function satisfying $Q'(q_{i})=0$ by choosing $\tilde{\bf x}$ and $\tilde{\bf y}$ carefully, see Figure ~\ref{fig:Qudratic}. The proof of the existence of such construction is left to the reader.

\begin{figure}
    \centering

\tikzset{every picture/.style={line width=0.75pt}} 

\begin{tikzpicture}[x=0.75pt,y=0.75pt,yscale=-1,xscale=1]

\draw  [draw opacity=0][fill={rgb, 255:red, 0; green, 0; blue, 0 }  ,fill opacity=1 ] (114.67,178.91) .. controls (114.67,176.75) and (116.42,175) .. (118.58,175) .. controls (120.73,175) and (122.48,176.75) .. (122.48,178.91) .. controls (122.48,181.07) and (120.73,182.82) .. (118.58,182.82) .. controls (116.42,182.82) and (114.67,181.07) .. (114.67,178.91) -- cycle ;
\draw  [draw opacity=0][fill={rgb, 255:red, 0; green, 0; blue, 0 }  ,fill opacity=1 ] (110.67,136.24) .. controls (110.67,134.08) and (112.42,132.33) .. (114.58,132.33) .. controls (116.73,132.33) and (118.48,134.08) .. (118.48,136.24) .. controls (118.48,138.4) and (116.73,140.15) .. (114.58,140.15) .. controls (112.42,140.15) and (110.67,138.4) .. (110.67,136.24) -- cycle ;
\draw  [draw opacity=0][fill={rgb, 255:red, 0; green, 0; blue, 0 }  ,fill opacity=1 ] (193.33,121.58) .. controls (193.33,119.42) and (195.08,117.67) .. (197.24,117.67) .. controls (199.4,117.67) and (201.15,119.42) .. (201.15,121.58) .. controls (201.15,123.73) and (199.4,125.48) .. (197.24,125.48) .. controls (195.08,125.48) and (193.33,123.73) .. (193.33,121.58) -- cycle ;
\draw    (158.48,245.09) -- (243.37,119.41) ;
\draw [shift={(244.48,117.76)}, rotate = 124.03] [color={rgb, 255:red, 0; green, 0; blue, 0 }  ][line width=0.75]    (10.93,-3.29) .. controls (6.95,-1.4) and (3.31,-0.3) .. (0,0) .. controls (3.31,0.3) and (6.95,1.4) .. (10.93,3.29)   ;
\draw    (118.58,178.91) .. controls (111.82,163.09) and (107.82,147.76) .. (114.58,136.24) ;
\draw    (114.58,136.24) .. controls (129.15,116.42) and (149.15,121.76) .. (197.24,121.58) ;
\draw  [dash pattern={on 0.84pt off 2.51pt}]  (140.58,97.91) -- (88.58,174.58) ;
\draw    (158.48,245.09) -- (58.13,175.56) ;
\draw [shift={(56.48,174.42)}, rotate = 34.71] [color={rgb, 255:red, 0; green, 0; blue, 0 }  ][line width=0.75]    (10.93,-3.29) .. controls (6.95,-1.4) and (3.31,-0.3) .. (0,0) .. controls (3.31,0.3) and (6.95,1.4) .. (10.93,3.29)   ;
\draw  [dash pattern={on 0.84pt off 2.51pt}]  (73.82,149.09) -- (175.82,219.76) ;
\draw  [dash pattern={on 0.84pt off 2.51pt}]  (114.58,136.24) -- (195.15,191.09) ;
\draw  [dash pattern={on 0.84pt off 2.51pt}]  (197.24,121.58) -- (228.48,141.76) ;
\draw  [draw opacity=0][fill={rgb, 255:red, 0; green, 0; blue, 0 }  ,fill opacity=1 ] (156.53,245.09) .. controls (156.53,244.01) and (157.41,243.14) .. (158.48,243.14) .. controls (159.56,243.14) and (160.44,244.01) .. (160.44,245.09) .. controls (160.44,246.17) and (159.56,247.05) .. (158.48,247.05) .. controls (157.41,247.05) and (156.53,246.17) .. (156.53,245.09) -- cycle ;

\draw (105.91,182.31) node [anchor=north west][inner sep=0.75pt]    {$v_{i-1}$};
\draw (93.24,118.98) node [anchor=north west][inner sep=0.75pt]    {$v_{i}$};
\draw (193.24,97.64) node [anchor=north west][inner sep=0.75pt]    {$v_{i+1}$};
\draw (157.24,244.98) node [anchor=north west][inner sep=0.75pt]    {$p_{0}$};
\draw (246.58,97.4) node [anchor=north west][inner sep=0.75pt]    {$\tilde{x}$};
\draw (39.91,158.73) node [anchor=north west][inner sep=0.75pt]    {$\tilde{y}$};
\end{tikzpicture}
    \caption{Choosing $\tilde{\bfx}$ and $\tilde{\bfy}$ such that the parabola pass $\bfv_{i-1}$ $\bfv_{i}$, and $\bfv_{i+1}$ has a extreme at $\bfv_{i}$.}
    \label{fig:Qudratic}
\end{figure}
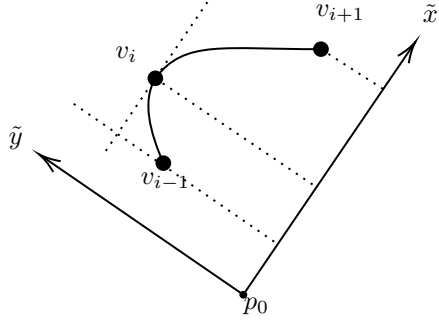

To obtain a parametric function from the above construction, we need a rotation.  Let us use $(x, y, z)$ for as the coordinate in $\mathbb{R}^3$, as the example. We rewrite $\tilde{\bf x}= R_{11}x+R_{12}y+R_{13}z$ , $\tilde{\bf y}=R_{21}x+R_{22}y+R_{23}z$. We can then define $\bff_i(s)=\bfp_0 + s \tilde{\bfx} +Q(s) \tilde{\bfy}$, i.e.
$$
\bff_i(s)=
\begin{bmatrix}
x_i(s) \\
y_i(s) \\
z_i(s)
\end{bmatrix}
= \bfp_0 +
\begin{bmatrix} 
R_{11} & R_{12}  & R_{13}\\
R_{21} & R_{22}  & R_{23}
\end{bmatrix}^T
\begin{bmatrix}
s\\
Q(s)
\end{bmatrix}.
$$

We can see that $\bff_i$ is regular by
	$$
	\bff'_i(s)=
	\begin{bmatrix} 
	R_{11} & R_{12}  & R_{13}\\
	R_{21} & R_{22}  & R_{23}
	\end{bmatrix}^T
	\begin{bmatrix}
	1\\
	Q'(s)
	\end{bmatrix},
	$$
 which is never zero. Hence, the parabola is a local $G^\infty$ interpolating function.\\
The various $\bff_i$ can be visualized in Figure~\ref{mainfig}.

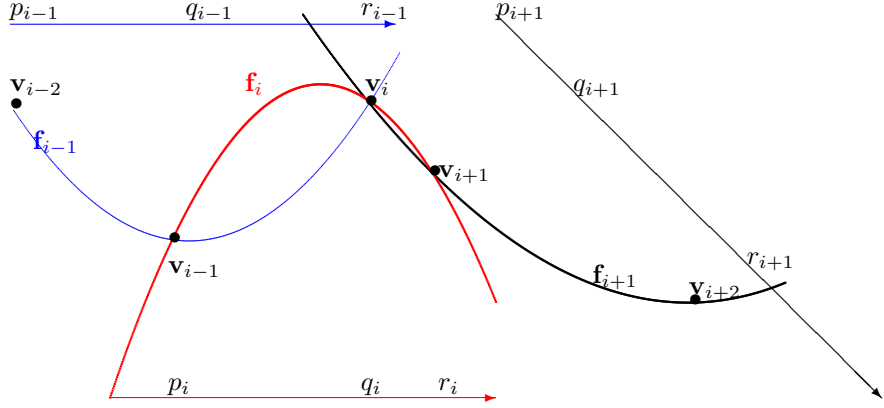
\begin{figure}[htpb]
	\begin{picture}(250,150)(50,30)
	\unitlength=0.01in
	\textcolor{blue}{
		\qbezier(100,200)(200,50)(300,230)
		\put(110,180){$\bff_{i-1}$}
		\put(98,245){\vector(1,0){200}}
	}
	\put(98,250){$p_{i-1}$}
	\put(189,250){$q_{i-1}$}
	\put(280,250){$r_{i-1}$}
	\put(98,200){$\bullet$}
	\put(98,210){$\bfv_{i-2}$}
	\put(400,110){$\bff_{i+1}$}
	\put(450,98){$\bullet$}
	\put(450,104){$\bfv_{i+2}$}
	\put(350,250){\vector(1,-1){200}}
	\put(350,250){$p_{i+1}$}
	\put(390,210){$q_{i+1}$}
	\put(480,120){$r_{i+1}$}
	\textcolor{red}{\put(150,50){\vector(1,0){200}}}
	\put(180,54){$p_{i}$}
	\put(280,54){$q_i$}
	\put(320,54){$r_{i}$}
	\thicklines
	\qbezier(250,250)(380,60)(500,110)
	\textcolor{red}{
		\qbezier(150,50)(250,350)(350,100)
		\put(220,210){$\bff_i$}
	}
	\put(180,130){$\bullet$}
	\put(180,115){$\bfv_{i-1}$}
	\put(282,202){$\bullet$}
	\put(282,210){$\bfv_i$}
	\put(315,165){$\bullet$}
	\put(320,165){$\bfv_{i+1}$}
	\end{picture}
	\caption{An illustration of point cloud and different local curves \label{mainfig}}
\end{figure}

\end{example}

\begin{example}[Arc of Circle or Ellipse]\label{sec: exp2 Arc}
It it known that 3 points $ \bfv_{i-1}, \bfv_i, \bfv_{i+1}$ in $\mathbb{R}^2$ can uniquely determine a circle (a line can be seen as a degenerated circle) so are the points in $\mathbb{R}^n$ when they are flattenable for $n\ge 3$. 
If  $ \bfv_{i-1}, \bfv_i, \bfv_{i+1}$ are flattenable, we can choose an arc of the circle in which $\bfv_i$ is located within the arc while $ \bfv_{i-1}$ and $\bfv_{i+1}$ are the endpoints of the arc, see the left of Figure ~\ref{fig:LocalCircle}. However, the circle arc could be very off when the angle $ \bfv_{i-1}, \bfv_i, \bfv_{i+1}$ is relatively small (Middle of Figure ~\ref{fig:LocalCircle}). 
In \cite{Yuksel2020}, Yuksel proposes to use the ellipse arc instead if one of the arc $ \bfv_{i-1}- \bfv_i$, $\bfv_i-\bfv_{i+1}$ has corresponding to an angle larger than 90 degrees.
The construction can be described as follows: first, place $\bfv_i$ on one axis of the ellipse. Second, place either $ \bfv_{i-1}$ or $\bfv_{i_1}$ on another axis of the ellipse, depending on which further from $\bfv_i$. Then, there is a unique ellipse that passes through three points with the axis with the specified property. See the right of Figure ~\ref{fig:LocalCircle}.

\begin{figure}
    \centering

\tikzset{every picture/.style={line width=0.75pt}} 

\begin{tikzpicture}[x=0.75pt,y=0.75pt,yscale=-1,xscale=1]

\draw  [draw opacity=0] (283.33,208.81) .. controls (261.78,198.67) and (246.4,177.24) .. (245.18,151.78) .. controls (243.42,115.1) and (271.74,83.93) .. (308.42,82.18) .. controls (345.11,80.42) and (376.27,108.74) .. (378.03,145.42) .. controls (378.79,161.41) and (373.85,176.34) .. (364.99,188.25) -- (311.6,148.6) -- cycle ; \draw  [color={rgb, 255:red, 208; green, 2; blue, 27 }  ,draw opacity=1 ] (283.33,208.81) .. controls (261.78,198.67) and (246.4,177.24) .. (245.18,151.78) .. controls (243.42,115.1) and (271.74,83.93) .. (308.42,82.18) .. controls (345.11,80.42) and (376.27,108.74) .. (378.03,145.42) .. controls (378.79,161.41) and (373.85,176.34) .. (364.99,188.25) ;  
\draw  [draw opacity=0][fill={rgb, 255:red, 0; green, 0; blue, 0 }  ,fill opacity=1 ] (87.17,196.83) .. controls (87.17,194.67) and (88.92,192.92) .. (91.08,192.92) .. controls (93.23,192.92) and (94.98,194.67) .. (94.98,196.83) .. controls (94.98,198.99) and (93.23,200.74) .. (91.08,200.74) .. controls (88.92,200.74) and (87.17,198.99) .. (87.17,196.83) -- cycle ;
\draw  [draw opacity=0][fill={rgb, 255:red, 0; green, 0; blue, 0 }  ,fill opacity=1 ] (83.17,154.17) .. controls (83.17,152.01) and (84.92,150.26) .. (87.08,150.26) .. controls (89.23,150.26) and (90.98,152.01) .. (90.98,154.17) .. controls (90.98,156.33) and (89.23,158.08) .. (87.08,158.08) .. controls (84.92,158.08) and (83.17,156.33) .. (83.17,154.17) -- cycle ;
\draw  [draw opacity=0][fill={rgb, 255:red, 0; green, 0; blue, 0 }  ,fill opacity=1 ] (165.83,139.5) .. controls (165.83,137.34) and (167.58,135.59) .. (169.74,135.59) .. controls (171.9,135.59) and (173.65,137.34) .. (173.65,139.5) .. controls (173.65,141.66) and (171.9,143.41) .. (169.74,143.41) .. controls (167.58,143.41) and (165.83,141.66) .. (165.83,139.5) -- cycle ;
\draw  [draw opacity=0] (89.58,195.45) .. controls (80.53,179.61) and (80.75,159.32) .. (91.82,143.31) .. controls (107.2,121.11) and (137.66,115.56) .. (159.87,130.94) .. controls (164.7,134.28) and (168.75,138.34) .. (171.96,142.87) -- (132.04,171.15) -- cycle ; \draw   (89.58,195.45) .. controls (80.53,179.61) and (80.75,159.32) .. (91.82,143.31) .. controls (107.2,121.11) and (137.66,115.56) .. (159.87,130.94) .. controls (164.7,134.28) and (168.75,138.34) .. (171.96,142.87) ;  
\draw  [draw opacity=0][fill={rgb, 255:red, 0; green, 0; blue, 0 }  ,fill opacity=1 ] (279.87,208.81) .. controls (279.87,206.89) and (281.42,205.34) .. (283.33,205.34) .. controls (285.25,205.34) and (286.8,206.89) .. (286.8,208.81) .. controls (286.8,210.73) and (285.25,212.28) .. (283.33,212.28) .. controls (281.42,212.28) and (279.87,210.73) .. (279.87,208.81) -- cycle ;
\draw  [draw opacity=0][fill={rgb, 255:red, 0; green, 0; blue, 0 }  ,fill opacity=1 ] (252.35,183.96) .. controls (252.35,182.04) and (253.91,180.49) .. (255.82,180.49) .. controls (257.74,180.49) and (259.29,182.04) .. (259.29,183.96) .. controls (259.29,185.88) and (257.74,187.43) .. (255.82,187.43) .. controls (253.91,187.43) and (252.35,185.88) .. (252.35,183.96) -- cycle ;
\draw  [draw opacity=0][fill={rgb, 255:red, 0; green, 0; blue, 0 }  ,fill opacity=1 ] (361.52,188.25) .. controls (361.52,186.34) and (363.07,184.78) .. (364.99,184.78) .. controls (366.91,184.78) and (368.46,186.34) .. (368.46,188.25) .. controls (368.46,190.17) and (366.91,191.72) .. (364.99,191.72) .. controls (363.07,191.72) and (361.52,190.17) .. (361.52,188.25) -- cycle ;
\draw [color={rgb, 255:red, 74; green, 144; blue, 226 }  ,draw opacity=1 ][fill={rgb, 255:red, 74; green, 144; blue, 226 }  ,fill opacity=1 ] [dash pattern={on 4.5pt off 4.5pt}]  (395.53,158.82) -- (354.71,252) ;
\draw  [dash pattern={on 4.5pt off 4.5pt}]  (283.33,208.81) -- (360,241.2) ;
\draw  [dash pattern={on 4.5pt off 4.5pt}]  (255.82,183.96) -- (364.4,230) ;
\draw  [dash pattern={on 4.5pt off 4.5pt}]  (364.99,188.25) -- (379.6,194.8) ;
\draw  [draw opacity=0][dash pattern={on 4.5pt off 4.5pt}] (566.14,142.18) .. controls (596.54,183.9) and (611.04,224.98) .. (598.52,234.16) .. controls (585.92,243.4) and (550.76,216.85) .. (519.97,174.86) .. controls (505.1,154.58) and (493.97,134.41) .. (487.86,117.93) -- (542.78,158.14) -- cycle ; \draw  [dash pattern={on 4.5pt off 4.5pt}] (566.14,142.18) .. controls (596.54,183.9) and (611.04,224.98) .. (598.52,234.16) .. controls (585.92,243.4) and (550.76,216.85) .. (519.97,174.86) .. controls (505.1,154.58) and (493.97,134.41) .. (487.86,117.93) ;  
\draw  [draw opacity=0] (488.07,118.5) .. controls (481.34,100.58) and (480.46,86.95) .. (487.04,82.12) .. controls (499.62,72.89) and (534.74,99.39) .. (565.51,141.32) -- (542.78,158.14) -- cycle ; \draw  [color={rgb, 255:red, 208; green, 2; blue, 27 }  ,draw opacity=1 ] (488.07,118.5) .. controls (481.34,100.58) and (480.46,86.95) .. (487.04,82.12) .. controls (499.62,72.89) and (534.74,99.39) .. (565.51,141.32) ;  
\draw  [draw opacity=0][fill={rgb, 255:red, 0; green, 0; blue, 0 }  ,fill opacity=1 ] (562.24,141.32) .. controls (562.24,139.51) and (563.7,138.05) .. (565.51,138.05) .. controls (567.32,138.05) and (568.78,139.51) .. (568.78,141.32) .. controls (568.78,143.13) and (567.32,144.59) .. (565.51,144.59) .. controls (563.7,144.59) and (562.24,143.13) .. (562.24,141.32) -- cycle ;
\draw  [draw opacity=0][fill={rgb, 255:red, 0; green, 0; blue, 0 }  ,fill opacity=1 ] (483.48,82.34) .. controls (483.48,80.53) and (484.94,79.07) .. (486.75,79.07) .. controls (488.56,79.07) and (490.02,80.53) .. (490.02,82.34) .. controls (490.02,84.15) and (488.56,85.61) .. (486.75,85.61) .. controls (484.94,85.61) and (483.48,84.15) .. (483.48,82.34) -- cycle ;
\draw  [dash pattern={on 0.84pt off 2.51pt}]  (565.51,141.32) -- (519.46,174.17) ;
\draw  [dash pattern={on 0.84pt off 2.51pt}]  (486.75,82.34) -- (598.57,234.12) ;
\draw  [draw opacity=0][fill={rgb, 255:red, 0; green, 0; blue, 0 }  ,fill opacity=1 ] (484.8,118.5) .. controls (484.8,116.69) and (486.27,115.23) .. (488.07,115.23) .. controls (489.88,115.23) and (491.34,116.69) .. (491.34,118.5) .. controls (491.34,120.3) and (489.88,121.77) .. (488.07,121.77) .. controls (486.27,121.77) and (484.8,120.3) .. (484.8,118.5) -- cycle ;

\draw (78.41,200.23) node [anchor=north west][inner sep=0.75pt]    {$v_{i-1}$};
\draw (65.74,136.9) node [anchor=north west][inner sep=0.75pt]    {$v_{i}$};
\draw (165.74,115.57) node [anchor=north west][inner sep=0.75pt]    {$v_{i+1}$};
\draw (272.91,211.53) node [anchor=north west][inner sep=0.75pt]    {$v_{i-1}$};
\draw (235.04,173.8) node [anchor=north west][inner sep=0.75pt]    {$v_{i}$};
\draw (341.04,162.87) node [anchor=north west][inner sep=0.75pt]    {$v_{i+1}$};
\draw (571.61,125.95) node [anchor=north west][inner sep=0.75pt]    {$v_{i+1}$};
\draw (449.62,110.86) node [anchor=north west][inner sep=0.75pt]    {$v_{i-1}$};
\draw (476.72,58.07) node [anchor=north west][inner sep=0.75pt]    {$v_{i}$};

\end{tikzpicture}

    \caption{Left: an Arc Passing Through $\bfv_{i-1}, \bfv_i$, and  $\bfv_{i+1}$. Mid: $\bfv_{i-1}, \bfv_i$, and  $\bfv_{i+1}$ are Flatenale (As they can Project to the Blue Dash Line and Preserve the Order). However, the Local Curve is Off. Right: The Construction of the Ellipse Local Curve.}
    \label{fig:LocalCircle}
\end{figure}
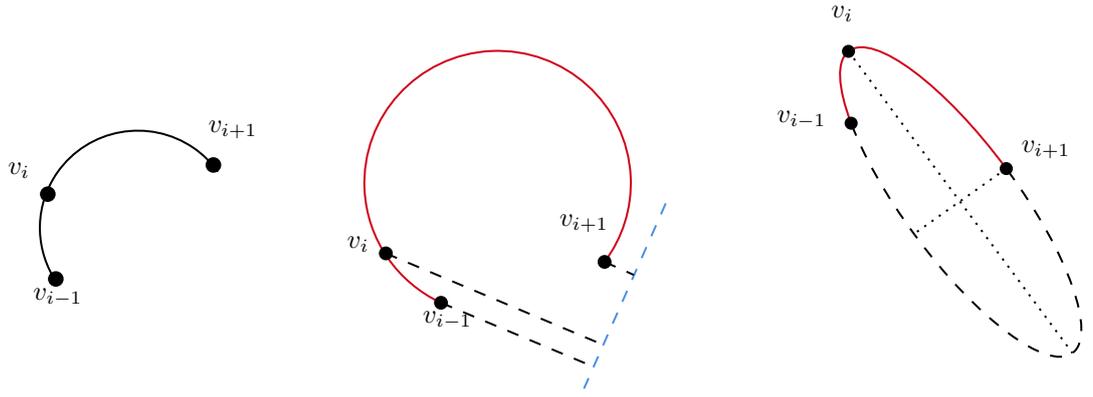

The parametric function of the arc can be defined by the polar coordinate. As shown in Figure ~\ref{fig:ArcPara}, We can parameterize the arc by $\bff_i(t)=\bfp_0+\cos(-t) \tilde{\bfx}+\sin(-t) \tilde{\bfy}$, where $\tilde{\bfx}$ and $\tilde{\bfy}$ axis of circle or ellipse. One can easily see that the parametric function is a regular $G^r$ local function.

\begin{figure}
    \centering

\tikzset{every picture/.style={line width=0.75pt}} 

\begin{tikzpicture}[x=0.75pt,y=0.75pt,yscale=-1,xscale=1]

\draw    (142.17,147.84) -- (173.33,92.23) ;
\draw [shift={(174.3,90.48)}, rotate = 119.26] [color={rgb, 255:red, 0; green, 0; blue, 0 }  ][line width=0.75]    (10.93,-3.29) .. controls (6.95,-1.4) and (3.31,-0.3) .. (0,0) .. controls (3.31,0.3) and (6.95,1.4) .. (10.93,3.29)   ;
\draw  [draw opacity=0] (81.55,171.22) .. controls (78.82,164.33) and (77.23,156.87) .. (77.01,149.03) .. controls (75.99,113.04) and (104.34,83.03) .. (140.34,82.02) .. controls (176.33,81) and (206.34,109.35) .. (207.35,145.35) .. controls (207.69,157.36) and (204.76,168.7) .. (199.37,178.53) -- (142.18,147.19) -- cycle ; \draw  [color={rgb, 255:red, 208; green, 2; blue, 27 }  ,draw opacity=1 ] (81.55,171.22) .. controls (78.82,164.33) and (77.23,156.87) .. (77.01,149.03) .. controls (75.99,113.04) and (104.34,83.03) .. (140.34,82.02) .. controls (176.33,81) and (206.34,109.35) .. (207.35,145.35) .. controls (207.69,157.36) and (204.76,168.7) .. (199.37,178.53) ;  
\draw  [draw opacity=0][dash pattern={on 4.5pt off 4.5pt}] (400.62,127.37) .. controls (429.46,166.95) and (443.21,205.91) .. (431.34,214.62) .. controls (419.39,223.38) and (386.03,198.2) .. (356.83,158.37) .. controls (342.72,139.14) and (332.17,120) .. (326.37,104.37) -- (378.46,142.51) -- cycle ; \draw  [dash pattern={on 4.5pt off 4.5pt}] (400.62,127.37) .. controls (429.46,166.95) and (443.21,205.91) .. (431.34,214.62) .. controls (419.39,223.38) and (386.03,198.2) .. (356.83,158.37) .. controls (342.72,139.14) and (332.17,120) .. (326.37,104.37) ;  
\draw  [draw opacity=0] (326.57,104.91) .. controls (320.18,87.91) and (319.35,74.98) .. (325.59,70.4) .. controls (337.53,61.65) and (370.84,86.78) .. (400.03,126.55) -- (378.46,142.51) -- cycle ; \draw  [color={rgb, 255:red, 208; green, 2; blue, 27 }  ,draw opacity=1 ] (326.57,104.91) .. controls (320.18,87.91) and (319.35,74.98) .. (325.59,70.4) .. controls (337.53,61.65) and (370.84,86.78) .. (400.03,126.55) ;  
\draw  [draw opacity=0][fill={rgb, 255:red, 0; green, 0; blue, 0 }  ,fill opacity=1 ] (396.92,126.55) .. controls (396.92,124.84) and (398.31,123.45) .. (400.03,123.45) .. controls (401.74,123.45) and (403.13,124.84) .. (403.13,126.55) .. controls (403.13,128.27) and (401.74,129.66) .. (400.03,129.66) .. controls (398.31,129.66) and (396.92,128.27) .. (396.92,126.55) -- cycle ;
\draw  [draw opacity=0][fill={rgb, 255:red, 0; green, 0; blue, 0 }  ,fill opacity=1 ] (322.21,70.61) .. controls (322.21,68.9) and (323.6,67.51) .. (325.32,67.51) .. controls (327.03,67.51) and (328.42,68.9) .. (328.42,70.61) .. controls (328.42,72.32) and (327.03,73.71) .. (325.32,73.71) .. controls (323.6,73.71) and (322.21,72.32) .. (322.21,70.61) -- cycle ;
\draw  [dash pattern={on 0.84pt off 2.51pt}]  (378.19,142.14) -- (356.35,157.72) ;
\draw  [dash pattern={on 0.84pt off 2.51pt}]  (378.19,142.14) -- (431.39,214.59) ;
\draw  [draw opacity=0][fill={rgb, 255:red, 0; green, 0; blue, 0 }  ,fill opacity=1 ] (323.47,104.91) .. controls (323.47,103.19) and (324.86,101.8) .. (326.57,101.8) .. controls (328.28,101.8) and (329.67,103.19) .. (329.67,104.91) .. controls (329.67,106.62) and (328.28,108.01) .. (326.57,108.01) .. controls (324.86,108.01) and (323.47,106.62) .. (323.47,104.91) -- cycle ;
\draw  [draw opacity=0][fill={rgb, 255:red, 0; green, 0; blue, 0 }  ,fill opacity=1 ] (195.9,178.53) .. controls (195.9,176.61) and (197.45,175.06) .. (199.37,175.06) .. controls (201.28,175.06) and (202.84,176.61) .. (202.84,178.53) .. controls (202.84,180.44) and (201.28,181.99) .. (199.37,181.99) .. controls (197.45,181.99) and (195.9,180.44) .. (195.9,178.53) -- cycle ;
\draw  [draw opacity=0][fill={rgb, 255:red, 0; green, 0; blue, 0 }  ,fill opacity=1 ] (118.43,85.59) .. controls (118.43,83.67) and (119.98,82.12) .. (121.9,82.12) .. controls (123.81,82.12) and (125.37,83.67) .. (125.37,85.59) .. controls (125.37,87.51) and (123.81,89.06) .. (121.9,89.06) .. controls (119.98,89.06) and (118.43,87.51) .. (118.43,85.59) -- cycle ;
\draw  [dash pattern={on 0.84pt off 2.51pt}]  (142.17,147.84) -- (84.97,117.15) ;
\draw  [draw opacity=0][fill={rgb, 255:red, 0; green, 0; blue, 0 }  ,fill opacity=1 ] (78.08,167.75) .. controls (78.08,165.83) and (79.64,164.28) .. (81.55,164.28) .. controls (83.47,164.28) and (85.02,165.83) .. (85.02,167.75) .. controls (85.02,169.67) and (83.47,171.22) .. (81.55,171.22) .. controls (79.64,171.22) and (78.08,169.67) .. (78.08,167.75) -- cycle ;
\draw    (142.17,147.84) -- (197.61,177.58) ;
\draw [shift={(199.37,178.53)}, rotate = 208.21] [color={rgb, 255:red, 0; green, 0; blue, 0 }  ][line width=0.75]    (10.93,-3.29) .. controls (6.95,-1.4) and (3.31,-0.3) .. (0,0) .. controls (3.31,0.3) and (6.95,1.4) .. (10.93,3.29)   ;
\draw    (378.46,142.51) -- (398.97,128.5) ;
\draw [shift={(400.62,127.37)}, rotate = 145.65] [color={rgb, 255:red, 0; green, 0; blue, 0 }  ][line width=0.75]    (10.93,-3.29) .. controls (6.95,-1.4) and (3.31,-0.3) .. (0,0) .. controls (3.31,0.3) and (6.95,1.4) .. (10.93,3.29)   ;
\draw    (378.46,142.51) -- (326.5,72.22) ;
\draw [shift={(325.32,70.61)}, rotate = 53.53] [color={rgb, 255:red, 0; green, 0; blue, 0 }  ][line width=0.75]    (10.93,-3.29) .. controls (6.95,-1.4) and (3.31,-0.3) .. (0,0) .. controls (3.31,0.3) and (6.95,1.4) .. (10.93,3.29)   ;
\draw  [draw opacity=0][fill={rgb, 255:red, 0; green, 0; blue, 0 }  ,fill opacity=1 ] (138.71,147.19) .. controls (138.71,145.27) and (140.26,143.72) .. (142.18,143.72) .. controls (144.1,143.72) and (145.65,145.27) .. (145.65,147.19) .. controls (145.65,149.11) and (144.1,150.66) .. (142.18,150.66) .. controls (140.26,150.66) and (138.71,149.11) .. (138.71,147.19) -- cycle ;
\draw  [draw opacity=0][fill={rgb, 255:red, 0; green, 0; blue, 0 }  ,fill opacity=1 ] (374.72,142.14) .. controls (374.72,140.22) and (376.27,138.67) .. (378.19,138.67) .. controls (380.1,138.67) and (381.66,140.22) .. (381.66,142.14) .. controls (381.66,144.05) and (380.1,145.6) .. (378.19,145.6) .. controls (376.27,145.6) and (374.72,144.05) .. (374.72,142.14) -- cycle ;

\draw (405.15,111.61) node [anchor=north west][inner sep=0.75pt]    {$v_{i+1}$};
\draw (293.43,101.33) node [anchor=north west][inner sep=0.75pt]    {$v_{i-1}$};
\draw (316.92,47.23) node [anchor=north west][inner sep=0.75pt]    {$v_{i}$};
\draw (206.17,173.73) node [anchor=north west][inner sep=0.75pt]    {$v_{i+1}$};
\draw (54.5,164.23) node [anchor=north west][inner sep=0.75pt]    {$v_{i-1}$};
\draw (112,62.4) node [anchor=north west][inner sep=0.75pt]    {$v_{i}$};
\draw (396.1,137.51) node [anchor=north west][inner sep=0.75pt]    {$\tilde{x}$};
\draw (145.5,99.73) node [anchor=north west][inner sep=0.75pt]    {$\tilde{y}$};
\draw (358.6,90.75) node [anchor=north west][inner sep=0.75pt]    {$\tilde{y}$};
\draw (126.67,147.73) node [anchor=north west][inner sep=0.75pt]    {$p_{0}$};
\draw (367.33,147.07) node [anchor=north west][inner sep=0.75pt]    {$p_{0}$};
\end{tikzpicture}
    \caption{The Coordinate of Circle and Ellipse Local Curve.}
    \label{fig:ArcPara}
\end{figure}
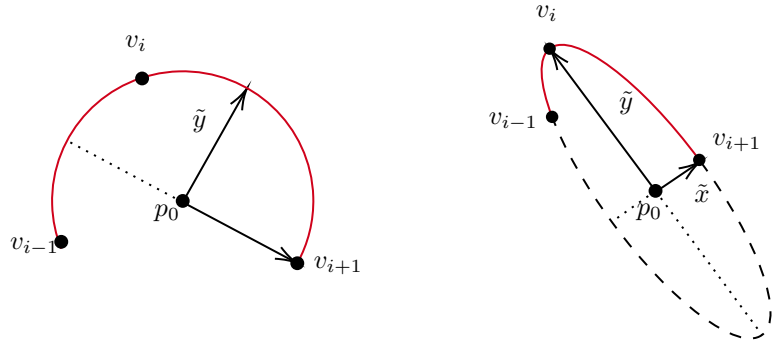

\end{example}

\begin{remark}\label{sec:multiplePTS}
 If the curve has endpoints, then we need a local curve that passes through $\bfv_0$ (and $\bfv_N$). One natural choice is the line segment connecting $\bfv_0$ and $\bfv_1$ ($\bfv_{N-1}$ and $\bfv_N$). Or, if our local curve of $\bfv_1$ also passes $\bfv_0$, we can use the same local curve as $\bfv_1$. 
 If the curve is a closed circle, say $\bfv_0=\bfv_N$, we have to get another local curve at $\bfv_0=\bfv_N$, which may be a parabola pass $\bfv_{N-1}$, $\bfv_0=\bfv_N$, and $\bfv_1$ or other local curves depend on needs. 
 \end{remark}

 \subsection{Redistribution and Quasi-regular Local Curves}
Suppose $\bfv_i$ is a local $G^r$ curve at $\bfv_i$, where $\bff_i(p_i)=\bfv_{i-1}$, $\bff_i(q_i)=\bfv_{i}$, and $\bff_i(r_i)=\bfv_{i+1}$. Redistribution of $\bff_i$ is a re-parameterization of the function by $\bff_i(t)=\bff_i(S_{i}^{-1}(s))$ using  $t=S_{i}^{-1}(s)$, 
where the bijection $S:[p_i,r_i] \to [i-1,i+1]$ with $S_i(p_i)=i-1$, $S_i(q_i)=i$, and $S(r_i)=i$. We call $\bfF_i(s)=\bff_i(S_{i}^{-1}(s))$ quasi-regular function (QR function).
The formal definition of the quasi-regular local parametric function can be described below:
\begin{definition}\label{def:QR localcurve}
 A function $\bfF_i: [i-1,i+1] \to \mathbb{R}^3$ is said to be a local $G^r$ quasi-regular interpolating function of ${\bf P}$ (at point $\bfv_i$) if:
 	\begin{itemize}
		\item ${\bf F}|_{(i-1,i)}$ and ${\bf F}|_{(i,i+1)}$ are regular and $C^r$.
		\item $\lim\limits_{t \to i^+} \nabla {\bf F}(t) \neq \bf 0$ for $t=i-1,i$, and 
		$\lim\limits_{t \to i^-} \nabla {\bf F}(t) \neq \bf 0$ for $t=i,i+1$.
		\item There exists a local re-parameterization $S_i: [q_i-\epsilon,q_i+\epsilon] \to [i-1,i]$ with $S_i(q_i)=i$ such 
		that ${\bf F}(S_i(s))$ is a regular $C^r$ function on $[q_i-\epsilon,q_i+\epsilon] \subset [p_i,r_i]$ for a $\epsilon>0$.
	\end{itemize}
\end{definition}
By condition 1) and 2) in Definition ~\ref{def:QR localcurve}, we can see that $\bfF_i$ is $G^r$ on $[i-1,i+1] \setminus \{ i \}$. By condition 3), we can see that $\bfF_i$ is $G^r$ at $i$. Hence, $\bfF$ is $G^r$ on $[i-1,i+1]$. This is not a surprise if we see that $\bfF_i$ is just the re-parametrize of $\bff_i$. The main problem is to choose the distribution function $S_i$. We see that using a piece-wise linear function will work.
\begin{lemma}\label{constructionFi}
	For a local curve $\bff_i: [p_i,r_i] \to \mathbb{R}^n$ which interpolate $\bfv_{i-1},\bfv_{i}$, and $\bfv_{i+1}$. Suppose $S_i:[p_i,r_i] \rightarrow [i-1,i+1]$ is a piecewise linear function consisting of two line segments such that $S_i(p_i)=i-1, S_i(q_i)=i$ and $S_i(r_i)=i+1$.  Then ${\bf F}_i(t)=\bff_i(S_i^{-1}(t))$ is a 
	$G^{\infty}$  quasi-regular parametric function such that ${\bf F}_i(j)=\bfv_j$ for $j=i-1, i, i+1$.
\end{lemma}
\begin{proof}
	First, $S_i(s)$ is a linear function on $s \in [p_i,q_i]$. We have 
	${\bf F}'_i(t)=\bff_i'(S_i^{-1}(t))(S_i^{-1})'(t)$.  
	Hence, ${\bf F}$ is a  regular curve on $(i-1, i)$ because 
	$(S_i^{-1})'(t)$ is a constant and $\bff_i$
	is regular (so that $\bff_i'(S_i^{-1}(t))$ is never zero). Similarly, ${\bf F}$ is a  regular curve on $(i,i+1)$.
	Secondly, one can easily see that ${\bf F}'_i(t)$ exists when $t \to (i-1)_+, i_-, i_+, or (i+1)_-$.
	Finally, it easy to see that ${\bf F}_i(S_i(s))=\bff_i(s)$ and $S_i$ is a regular $C^{\infty}$ function. Thus, ${\bf F}_i(t)$ is a $G^{\infty}$  quasi-regular parametric function.\\
    The interpolation property could be obtained by ${\bf F}_i(i-1)=\bff_i(S_i^{-1}(i-1))=\bff_i(p_i)=\bfv_{i-1}$, ${\bf F}_i(i)=\bff_i(S_i^{-1}(i))=\bff_i(q_i)=\bfv_{i}$, and ${\bf F}_i(i+1)=\bff_i(S_i^{-1}(i+1))=\bff_i(r_i)=\bfv_{i+1}$. 
\end{proof}

Since ${\bf F}$ and $\bff$ only differ in the parameterization, they have the same graph, which is $G^r$. Although ${\bf F}$ loses its regularity at the nodes, it stores the point cloud at integer nodes, which is a more desirable property in computer programming.\\

Before we move on to constrict some local curves, we shall define another important property of these $ G^r$ quasi-regular functions.
\begin{definition}\label{contractness}
 We say a local $G^r$ interpolating QR function $\bfF_i: [i-1,i+1] \to \mathbb{R}^3$, as we defined in Definition ~\ref{def:QR localcurve}, is contracted if the inner products 
 \[ \langle \bfF_i(s)-\bfL_i(s),\bfv_i-\bfv_{i-1} \rangle > 0, \text{ for } s \in (i-1,i) \]
 and 
\[ \langle \bfF_i(s)-\bfL_i(s),\bfv_i-\bfv_{i+1} \rangle > 0, \text{ for } s \in (i,i+1), \]
where $\bfL_i$ is the piecewise linear function with $\bfL_i(i-1)=\bfv_{i-1}$, $\bfL_i(i)=\bfv_{i}$, and $\bfL_i(i+1)=\bfv_{i+1}$
\end{definition}

We illustrated the contracted local curve in Figure ~\ref{fig:contract}. Generally speaking, a parametric function is contracted if it is always closer to $\bfv_i=f(q_i)$ than the linear interpolating on both sides. \\

\begin{figure}
    \centering

\tikzset{every picture/.style={line width=0.75pt}} 

\begin{tikzpicture}[x=0.75pt,y=0.75pt,yscale=-1,xscale=1]

\draw  [dash pattern={on 0.84pt off 2.51pt}]  (155,56) -- (277,135) ;
\draw    (71,111.5) .. controls (103,79.5) and (109,66) .. (155,56) ;
\draw    (155,56) .. controls (207,51) and (224.5,63) .. (256.5,99) ;
\draw  [fill={rgb, 255:red, 0; green, 0; blue, 0 }  ,fill opacity=1 ] (150.5,56) .. controls (150.5,53.51) and (152.51,51.5) .. (155,51.5) .. controls (157.49,51.5) and (159.5,53.51) .. (159.5,56) .. controls (159.5,58.49) and (157.49,60.5) .. (155,60.5) .. controls (152.51,60.5) and (150.5,58.49) .. (150.5,56) -- cycle ;
\draw  [fill={rgb, 255:red, 0; green, 0; blue, 0 }  ,fill opacity=1 ] (54.5,146) .. controls (54.5,143.51) and (56.51,141.5) .. (59,141.5) .. controls (61.49,141.5) and (63.5,143.51) .. (63.5,146) .. controls (63.5,148.49) and (61.49,150.5) .. (59,150.5) .. controls (56.51,150.5) and (54.5,148.49) .. (54.5,146) -- cycle ;
\draw  [fill={rgb, 255:red, 0; green, 0; blue, 0 }  ,fill opacity=1 ] (272.5,135) .. controls (272.5,132.51) and (274.51,130.5) .. (277,130.5) .. controls (279.49,130.5) and (281.5,132.51) .. (281.5,135) .. controls (281.5,137.49) and (279.49,139.5) .. (277,139.5) .. controls (274.51,139.5) and (272.5,137.49) .. (272.5,135) -- cycle ;
\draw  [dash pattern={on 0.84pt off 2.51pt}]  (59,146) -- (155,56) ;
\draw  [fill={rgb, 255:red, 0; green, 0; blue, 0 }  ,fill opacity=1 ] (252,99) .. controls (252,96.51) and (254.01,94.5) .. (256.5,94.5) .. controls (258.99,94.5) and (261,96.51) .. (261,99) .. controls (261,101.49) and (258.99,103.5) .. (256.5,103.5) .. controls (254.01,103.5) and (252,101.49) .. (252,99) -- cycle ;
\draw  [fill={rgb, 255:red, 0; green, 0; blue, 0 }  ,fill opacity=1 ] (66.5,111.5) .. controls (66.5,109.01) and (68.51,107) .. (71,107) .. controls (73.49,107) and (75.5,109.01) .. (75.5,111.5) .. controls (75.5,113.99) and (73.49,116) .. (71,116) .. controls (68.51,116) and (66.5,113.99) .. (66.5,111.5) -- cycle ;
\draw  [fill={rgb, 255:red, 208; green, 2; blue, 27 }  ,fill opacity=1 ] (91.5,112) .. controls (91.5,109.51) and (93.51,107.5) .. (96,107.5) .. controls (98.49,107.5) and (100.5,109.51) .. (100.5,112) .. controls (100.5,114.49) and (98.49,116.5) .. (96,116.5) .. controls (93.51,116.5) and (91.5,114.49) .. (91.5,112) -- cycle ;
\draw  [fill={rgb, 255:red, 74; green, 144; blue, 226 }  ,fill opacity=1 ] (95.5,83) .. controls (95.5,80.51) and (97.51,78.5) .. (100,78.5) .. controls (102.49,78.5) and (104.5,80.51) .. (104.5,83) .. controls (104.5,85.49) and (102.49,87.5) .. (100,87.5) .. controls (97.51,87.5) and (95.5,85.49) .. (95.5,83) -- cycle ;
\draw  [dash pattern={on 0.84pt off 2.51pt}]  (100,83) -- (113,96) ;
\draw    (96,112) -- (99.73,84.98) ;
\draw [shift={(100,83)}, rotate = 97.85] [color={rgb, 255:red, 0; green, 0; blue, 0 }  ][line width=0.75]    (10.93,-3.29) .. controls (6.95,-1.4) and (3.31,-0.3) .. (0,0) .. controls (3.31,0.3) and (6.95,1.4) .. (10.93,3.29)   ;

\draw (61,149.4) node [anchor=north west][inner sep=0.75pt]    {$v_{i-1}$};
\draw (121.5,28.4) node [anchor=north west][inner sep=0.75pt]    {$v_{i} =F( i)$};
\draw (279,142.9) node [anchor=north west][inner sep=0.75pt]    {$v_{i+1}$};
\draw (266,86.4) node [anchor=north west][inner sep=0.75pt]    {$F_{i}( i+1)$};
\draw (14,95.4) node [anchor=north west][inner sep=0.75pt]    {$F_{i}( i-1)$};
\draw (101.5,118.9) node [anchor=north west][inner sep=0.75pt]    {$L_{i}( s)$};
\draw (72.5,54.4) node [anchor=north west][inner sep=0.75pt]    {$F_{i}( s)$};

\end{tikzpicture}

    \caption{The Contracted Local Curve. The red point is $\bfL_i(s)$ and the blue points is $\bfF_i(s)$. Generally speaking, the contracted QR function is always closer to $\bfv_i$ than linear interpolation after projection to the line connecting the vertex.}
    \label{fig:contract}
\end{figure}
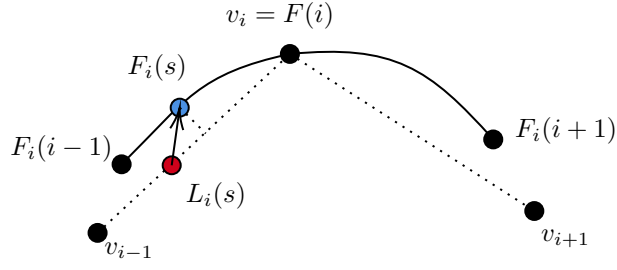

Another important property is the positive definiteness:
\begin{definition}
 We say a local regular (which also imply $C^1$) function $F:[a,b] \to \mathbb{R}^n$ on an interval $(\alpha, \beta) \subset [a,b]$ if $\langle F'(x), F(\beta)-F(\alpha) \rangle >0$ for all $x \in (\alpha, \beta)$.  
\end{definition}
Generally speaking, if a curve is positive definite, then it's target is always toward the endpoint.\\
The contractness and the positive definite are sufficient conditions of the final gluing curve to be $G^r$. Hence, we shall prove our construction of local curves are contracting and positive definite.

\begin{example}[Contractness and Positive definite of Spatial Parabola]\label{exp:pottoreg}
	Let $\bfF_i(s)=\bff_i(S^{-1}_i(s))$, where $\bff_i$ is defined as in Example ~\ref{ex:para}, and $S:[p_i,r_i] \to [i-1,i+1]$ is join of two linear functions with $S_i(p_i)=i-1$, $S_i(q_i)=i$, and $S(r_i)=i$. We can prove that ${\bf F}_i:[i-1,i] \rightarrow \mathbb{R}^3$ is a $C^\infty$ positive definite curve and contracted at $i$, ${\bf F}_i:[i,i+1] \rightarrow \mathbb{R}^3$ is a $C^r$ positive definite curve and contracted at $i$.

	For positive definiteness, we check the graph on $\Tilde{\bf x}, \Tilde{\bf y}$ as in Example ~\ref{ex:para}, where the parabola curve can be described as $(s, Q(s))$ on the coordinate systems $\Tilde{\bf x}, \Tilde{\bf y}$ with origin at $\bfp_0$. Since $Q(q_i)$ is the vertex of the parabola, all the derivative is toward the vertex, i.e. $Q'(s)(Q(q_i)-Q(s))\geq 0$. For $t\in (i-1, i)$, we have

    \begin{align*}
    \langle {\bf F}_i'(t), {\bf F}_i(i)-{\bf F}_i(i-1) \rangle &=\langle \bff_i'(s)(S_i^{-1})'(t), \bff_i(q_i)-\bff_i(p_i) \rangle \\
    &= (S_i^{-1})'(t) \langle Q'(s)\Tilde{\bfy}+\Tilde{\bfx}, \left( Q(q_i)- Q(p_{i})\right)\tilde{\bfy}+ (q_i-p_{i}) \Tilde{\bfx} \rangle  \\
    &= (S_i^{-1})'(t)(Q'(s)\left( Q(q_i)- Q(p_{i})\right)+(q_i-p_{i})) >0.
    \end{align*}

	Note that $S_i^{-1}$ is a linear incresing function on $[i-1,i]$ so $(S_i^{-1})'(t)$ is a positive constant. Similarly $\langle {\bf F}_i'(t), {\bf F}_i(i+1)-{\bf F}_i(i) \rangle = (S_i^{-1})'(t)(Q'(s)\left( Q(r_i)- Q(q_{i})\right)+(r_i-q_{i})) >0$.
 
	For contractness (cf. Definition ~\ref{contractness}), let $t\in (i-1, i)$, since ${L}(t)$ is a linear function from $\bfv_{i-1}$ to $\bfv_{i-1}$ and $S_i$ is a linear function, $ $we have $L(S_i(s))=l(s) \tilde{\bfy}+ s\tilde{\bf x}$, where $l(s)$ is a linear function with $l(p_i)=Q(p_i)$ and $l(q_i)=Q(q_i)$. Now let $t=S_i(s)$ we can calculate
	$\langle {\bf F}_i(t)-{L}(t), {\bf F}_i(i)-{\bf F}_i(i-1) \rangle= \langle f_i(s)-L(S_i(s)), f_i(q_i)-f_i(p_i) \rangle=\langle Q(s)-l(s)\Tilde{\bf y}, \left( Q(q_i)- Q(p_i)\right)\tilde{\bf y}+ (q_i-p_i) \Tilde{\bf x} \rangle=(Q(s)-l(s))\left( Q(q_i)- Q(p_i)\right) \geq 0$. Because if $Q$ is convex, $Q(s) \leq l(s)$ and $Q(q_i) \leq Q(p_i)$ and if $Q$ is concave, $Q(s) \geq l(s)$ and $Q(q_i) \geq Q(p_i)$. Similarly, for $t \in [i, i+1]$.	

\end{example}

\subsection{Blending Functions}\label{sec:blending}
We begin with a construction of blending functions. 
\begin{definition}\label{Def_of_rblending}
	A pair of r-blending polynomials $B_1 ,B_2: [0, 1] \rightarrow \mathbf{R}$ is a polynomial satisfying  the following:
	\begin{itemize}
		\item $1> B_1$, $B_2 >0 $ on $(0,1)$.
		\item $B_1(1)=B_2(0)=0$ and $B_1(0)=B_2(1)=1$.
		\item $B_1^{(\alpha)}(0)=B_1^{(\alpha)}(1)=B_2^{(\alpha)}(0)=B_2^{(\alpha)}(1)=0$ for all $0<\alpha \leq r$.
		\item $B_1+B_2=1$ on $[0,1]$.
		\item $-B_1$ and $B_2$ are non decreasing.
	\end{itemize}    
\end{definition}

An example of a 0-blending function is \\
\begin{equation}
\begin{cases}
B_1(t) &=1-t, \forall t\in [0,1]\\
B_2(t) &=t, \forall t\in [0,1].   
\end{cases}
\end{equation}

An example of a 1-blending function is\\
\begin{equation}
\begin{cases}
B_1(t) &=2t^3-3t^2+1, \forall t\in [0,1]\\
B_2(t) &=-2t^3+3t^2, \forall t\in [0,1]
\end{cases}
\end{equation}

Another example of a 1-blending function is \\
\begin{equation}
\begin{cases}
B_1(t) &=\cos^2(t), \forall t\in [0,1]\\
B_2(t) &=\sin^2(t), \forall t\in [0,1].   
\end{cases}
\end{equation}

We can also generate $\infty$-blending function. let $f(x)=e^{-\frac{1}{x}}$, we define 
\begin{equation}
\begin{cases}
B_1(t) &=\dfrac{f(1-x)}{f(x)+f(1-x)}, \forall t\in [0,1]\\
B_2(t) &=\dfrac{f(x)}{f(x)+f(1-x)}, \forall t\in [0,1].   
\end{cases}
\end{equation}
is not hard to find that $B_1$ $B_2$ satisfying Definition ~\ref{Def_of_rblending}. Principally, this will give you a $G^\infty$ interpolation function if you use these in Algorithm ~\ref{alg: smoothcurve}. However, the derivatives of those blending functions are complex, which may not be suitable for computational manner.

In general, we can generate $r$-blending functions by polynomials $B_1(x)= \displaystyle \int_x^1 ct^r(1-t)^r dt$, 
$B_2(x)= 1-B_1(x)$, where $c= (\int_0^1 t^r(1-t)^r dt)^{-1}$. This is equivalent to $B_1(x)= \displaystyle \sum_{0\leq i \leq r} b_{i,2r+1}(x)$ and $B_2(x)= \displaystyle \sum_{r+1\leq i \leq 2r+1} b_{i,2r+1}(x)$, where $b_{i,2r+1}$ are $i$-th Bernstein basis polynomials of degree $2r+1$. We call this {\bf polynomial r-blending function}. There are still many ways to define r-lending functions.  We leave them to the interested reader. 

\subsection{Linear Gluing of QR Curves and Main Theorem} \label{Glu_2_QPR}
Consider ${\bf F}_i:[i-1,i+1] \to \mathbb{R}^n$ are $G^r$ QR interpolating functions of the data points 
$\{\bfv_i, i=0,1, \cdots, N\}$. Our final glued function $\Gamma:[0,N] \to \mathbb{R}^3$ to be:
\[ 
\Gamma(t)|_{[i,i+1]} ={\bf F}_i(t)B_1(t-i)+{\bf F}_{i+1}(t)B_2(t-i) 
\]
for $i=1,2,\cdots, N-1$. Note that $\bfF_0$ and $\bfF_N$ are boundaries QR function. See Remark ~\ref{sec:multiplePTS}. We will show that $\Gamma$ can be $G^r$ if we carefully choose our ${\bf f}_i$, $S_i$, $B_1$, and $B_2$. Indeed, we have:
\begin{theorem}\label{thm:Main}
	Suppose a set of oriented data point $P=\{ \bfv_0, \bfv_1, \cdots , \bfv_N \}$. We define a curve 
 \[ 
\Gamma(t)|_{[i,i+1]} ={\bf F}_i(t)B_1(t-i)+{\bf F}_{i+1}(t)B_2(t-i) 
\] where:
	\begin{description}
		\item{1)} $\bff_i: [p_i,r_i] \to \mathbb{R}^n$ are $C^r$ local curve which $f_i(q_i)=\bfv_i$ for some $q_i \in (p+i, r_i)$ for $i=1,2, \cdots, N-1$. $\bff_0:[q_0,r_0] \to \mathbb{R}^n$ is a $C^r$ function with $\bff_0(q_0)=\bfv_0$, and $\bff_N:[p_N,q_N] \to \mathbb{R}^n$ is a $C^r$ function with $\bff_N(q_N)=\bfv_N$.
		\item{2)} $\bfF_i=\bff_i\circ S_i^{-1}:[i-1, i+1]\to \mathbb{R}^n$ is positive definite and contracted to $\bfv_i$, where $S^{i}:[p_i,r_i]$ piecewise linear function with $S_i(p_i)=i-1$, $S_i(q_i)=i$, and $S_i(r_i)=i+1$. When $i=0, N$, $S_i$ are linear function where $S_0(q_i)=0$, $S_0(r_i)=1$, $S_N(p_i)=N-1$, and $S_N(q_i)=N$.
        \item{3)} $B_1$, $B_2$ are $r$-blending function.
   \end{description} 
   Then the graph of $\Gamma(t)$ is $G^r$ which $\Gamma(0)=\bfv_0, \Gamma(1)=\bfv_1, \cdots, \Gamma(N)=\bfv_N$.
\end{theorem}
We divided the proof into three claims:
\begin{claim}\label{clm1}
    The $\Gamma(t)$ defines in Theorem ~\ref{thm:Main} interpolate the data points $P$.
\end{claim}
\begin{proof}
    The simple substitution $\Gamma(i)={\bf F}_i(i)B_1(0)+{\bf F}_{i+1}(i)B_2(0)= {\bf F}_i(i)=\bff_i(S_i^{-1}(i))=\bff_i(q_i)=\bfv_i$.
\end{proof}
\begin{claim}\label{clm2}
    The graph of $\Gamma(t)$ is $G^r$ at integer points.
\end{claim}
\begin{proof}
    Consider $\Gamma_i:=\Gamma|_{(i-1,i+1)}$ for $i=1,2,\cdots, N-1$. We want to show that $\Gamma$ is $G^r$ at $i$. We can consider a re-parameterization  $\bfg_i=\Gamma_i \circ S_i : [p_i,r_i] \to \mathbb{R}^n$. We can show $\bfg_i^{(\alpha)}(q_i)=\bff_i^{(\alpha)}(q_i)$, for $1<\alpha \leq r$, by calculate:
    \begin{eqnarray*}
		\lim\limits_{s \to q_i^+} \frac{d^\alpha}{(ds)^\alpha} \bfg_i(s)
		&=& \lim\limits_{s \to q_i^+} \dfrac{d^\alpha}{(ds)^\alpha} \Bigl( \bfF_i(S_i(s))B_1(S_i(s)-i)+\bfF_{i+1}(S_i(s))B_2(S_i(s)-i\Bigr)\\
        &=& \lim\limits_{s \to q_i^+} \dfrac{d^\alpha}{(ds)^\alpha}\Bigl( \bff_i(S_i^{-1} \circ S_i(s))B_1(S_i(s)-i )\\
        && \qquad +\bff_{i+1}(S_{i+1}^{-1} \circ S_i(s))B_2(S_i(s)-i )\Bigr)\\
        &=& \lim\limits_{s \to q_i^+} \dfrac{d^\alpha}{(ds)^\alpha}\Bigl( \bff_i(s)B_1(S_i(s)-i)+\bff_{i+1}(S_{i+1}^{-1} \circ S_i(s))B_2(S_i(s)-i)\Bigr)\\
        &=& \bff_i^{(\alpha)}(q_i) +\lim\limits_{s \to q_i^+} \displaystyle\sum_{\substack{m+n=\alpha\\n>1}} \bff_i^{(m)}(s)\dfrac{d^n}{(ds)^n}B_1(S_i(s)-i)\\
        &&\qquad \quad+\lim\limits_{s \to q_i^+} \displaystyle\sum_{m+n=\alpha} (\dfrac{d^m}{(ds)^m} \bff_{i+1}(S_{i+1}^{-1} \circ S_i(s)))(\dfrac{d^n}{(ds)^n}B_2(S_i(s)-i)).
	\end{eqnarray*}
    For the second and third term, Note that $S_i|_{(q_i,r_i)}$ are a linear function to $(i, i+1)$, so $(S_i(s)-i)$ has range on $(0,1)$ when $s \in (q_i,r_i)$. We have  $\frac{d^n}{(ds)^n}B_1(S_i(s)-i)=0$ for $1\leq n \leq \alpha$ and $\frac{d^n}{(ds)^n}B_2(S_i(s)-i)=0$ for $0 \leq \alpha$, since $B_1$ and $B_2$ are $r$-blending function and $\alpha \leq r$. This leads to the conclusion that the second term is vanish. For the third term, we know the derivative of $B_2$ is zeros but we need to check if $\dfrac{d^m}{(ds)^m} \bff_{i+1}(S_{i+1}^{-1} \circ S_i(s))$ is exist and finite. Note that $S_{i+1}^{-1} \circ S_i|_{(q_i,r_i)}$ is a linear function from $(q_i,r_i)$ to $(p_{i+1},q_{i+1})$. We can conclude that $\dfrac{d^m}{(ds)^m} \bff_{i+1}(S_{i+1}^{-1} \circ S_i(s))$ is exist and finite. Hence, the second and third terms all vanish and 
    $$\lim\limits_{s \to q_i^+} \frac{d^\alpha}{(ds)^\alpha} \bfg_i(s) = \bff_i^{(\alpha)}(q_i).$$
    Similar to the left-hand limit, we have $\bfg_i^{(\alpha)}(q_i)=\bff_i^{(\alpha)}(q_i)$ for $0 \leq \alpha \leq r$. Therefore, $\bfg$ is regular $C^r$ function (so $G^r$) near $\bfv_i$. $\Gamma$ is $G^r$ near $\bfv_i$ because it and $\bfg_i$ have same graph. 
    
\end{proof}

\begin{claim}\label{clm3}
    The graph of $\Gamma(t)$ is $G^r$ on intervals $(i,i+1)$ for $i=0,1,\cdots,N-1$.
\end{claim}
\begin{proof}
    By definition of $\Gamma$, is not hard to see that $\Gamma$ is a $C^r$ function on $(i,i+1)$. For $G^r$, we only need to show that $\Gamma$ is regular, that is, $\Gamma'(t) \neq 0$.
    We can calculate:\\
	\begin{equation}\label{Sipt}
	\begin{aligned}
	{\bf \Gamma}'(t)|_{(i, i+1)} &=  {\bf F}'_i(t)B_1(t-i)+{\bf F}'_{i+1}(t)B_2(t-i)\\
	&+  {\bf F}_i(t)B'_1(t-i)+{\bf F}_{i+1}(t) B'_2(t-i)\\
    &={\bf F}'_i(t)B_1(t-i)+{\bf F}'_{i+1}(t)B_2(t-i)+\bigl({\bf F}_{i+1}(t)-{\bf F}_i(t)\bigr) B'_2(t-i)
	\end{aligned}
	\end{equation}
    Note that in the last line, we use $B'_1=(1-B_2)'=-B'_2$. Now consider the inner product $\langle {\bf \Gamma}'(t), {\bf \Gamma}(i+1)-{\bf \Gamma}(i) \rangle$. We have $\langle {\bf F}'_i(t), {\bf \Gamma}(i+1)-{\bf \Gamma}(i) \rangle$ and $\langle {\bf F}'_{i+1}(t), {\bf \Gamma}(i+1)-{\bf \Gamma}(i) \rangle >0$ because ${\bf F}_i$ and ${\bf F}_{i+1}$ are positive definite function. For the last component in \ref{Sipt}, we consider $\bfL: [i, i+1] \rightarrow \mathbb{R}^3$ be a linear function with $\bfL(i)={\bf \Gamma}(i)=\bfv_i$ and $\bfL(i+1)={\bf \Gamma}(i+1)=\bfv_{i+1}$. We can calculate
	\begin{align*}
	&\langle \bigl({\bf F}_{i+1}(t)-{\bf F}_i(t)\bigr) B'_2(t-i), {\bf \Gamma}(i+1)-{\bf \Gamma}(i) \rangle \\
	=&\langle ({\bf F}_{i+1}(t)-\bfL(t))-({\bf F}_{i}(t)-\bfL(t)), \bfv_{i+1} -\bfv_i \rangle \cdot B'_2(t-i)\\
	=&\Bigl( \bigl\langle {\bf F}_{i+1}(t)-\bfL(t), \bfv_{i+1}-\bfv_i \bigr\rangle + 
       \bigl\langle {\bf F}_{i}(t)-\bfL(t), \bfv_i-\bfv_{i+1} \bigr\rangle \Bigr) \cdot B'_2(t-i) \geq 0.
	\end{align*}
	Where we apply the contractness of $\bfF_i$ at $\bfv_i$ and $\bfF_{i+1}$ at $\bfv_{i+1}$ in the last line. We have ${\bf \Gamma}(t)$ is a positive definite on $(i, i+1)$, so it is regular because the first derivative never vanishes.
\end{proof}

By the claim \ref{clm1}, \ref{clm2}, and \ref{clm1}, the Theorem ~\ref{thm:Main} is proved.

\begin{corollary}[Bonus Smoothness]\label{Coro:BonusS}
If the local function $\bff_i$ in Theorem ~\ref{thm:Main} is $C^{r+1}$ and pass three point $\bff_i(p_i)=\bfv_{i-1}$, $\bff_i(q_i)=\bfv_{i}$, and $\bff_i(r_i)=\bfv_{i+1}$, and the $r$-blending function are differentiable up to $r+1$ (The derivative still vanish up to $r$). Then the curve $\Gamma$ is $G^{r+1}$.
\end{corollary}
\begin{proof}
    We already show the curve $\Gamma$ is regular. We only need to show that the $r+1$-th derivative is equal to the Local QR function on its boundary. From Equation ~\ref{Sipt}, we have
    \begin{equation}
	\begin{aligned}
	 {\bf \Gamma}^{(r+1)}(t)|_{(i, i+1)} &=\Bigl( \sum_{\substack{m+n=r+1\\n<r+1}} {\bf F}^{(m)}_i(t)B_1^{(n)}(t-i)+{\bf F}^{(m)}_{i+1}(t)B_2^{(n)}(t-i) \Bigr)\\
     &+\bigl({\bf F}_{i+1}(t)-{\bf F}_i(t)\bigr) B^{(r+1)}_2(t-i).
	\end{aligned}
	\end{equation}
When $t \to i$, the first term equal to $F^{r+1}_i$ because $B_1(0)=1$ and the rest terms of $B^{(n)}_1$, $B^{(n)}_2$ goes to zeros. For the second term, although $B^{(r+1)}_2(t-i)$ is not zeros, ${\bf F}_{i+1}(i)-{\bf F}_i(i)=\bfv_i-\bfv_i=0$. Hence ${\bf \Gamma}^{(r+1)}(t)=\bfF^{(r+1)}_i(t)$ when $t \to i^+$. Similarly, ${\bf \Gamma}^{(r+1)}(t)=\bfF^{(r+1)}_{i=1}(t)$ for $t \to i+1^-$.
 
\end{proof}

\section{Numerical Examples of 3D Smooth Curves}\label{Sec:Numerical}
In this section, we present several examples to demonstrate that the method discussed in the previous section can be used to construct
smooth 3D curves.  We begin with our computational algorithm below, which has been implemented in MATLAB.


\begin{algorithm}
\caption{The Curve Interpolation Algorithm. (See Figure ~\ref{Fig:alg}.)}\label{alg: smoothcurve}
\textbf{Step 1} For each point $\bfv_i$, use $\bfv_{i-1}$ and $\bfv_{i+1}$ to generate ${\bff_i}$ as in Definition ~\ref{def:localcurve} for $i=1,2,\cdots,N-1$.\\
\textbf{Step 1.5} For boundary point $\bfv_0$ and $\bfv_N$, constructing the local curve ${\bff_0}$ and ${\bff_N}$. Some canonical choices are:\\
\indent \textbf{Linear} Using linear function ${\bff_0}$ connect $\bfv_0$ and $\bfv_1$, and ${\bff_N}$ connect $\bfv_{N-1}$ and $\bfv_{N}$.\\
\indent \textbf{Natural} If local function $\bff_1$ pass $\bfv_0$ and $\bff_{N-1}$, simply use them as the ${\bff_0}$ and ${\bff_N}$.\\
\indent \textbf{Close} If $\bfv_N=\bfv_0$, means the data points is a close curve, we can define $\bff_N=\bff_0$ be the local function using $\bfv_{N-1}$,$\bfv_{N}=\bfv_0$, and $\bfv_1$.\\ 
\textbf{Step 2} Redistribution the local function $F_i=f_i(S_i^{-1})$.\\
\textbf{Step 3} Given smoothness $r\ge 1$, choose a suitable blending function $B_1$, $B_2$.\\
\textbf{Step 4} For each interval $[i-1, i]$, evaluate the function  ${\bff_{i-1}}(S^{-1}(x)) B_1(x) + {\bff_i}(S^{-1}(\bfx)) B_2(x)$.\\
\end{algorithm}

In addition, we present a computational flowchart to illustrate the steps of Algorithm~\ref{alg: smoothcurve} in Figure~\ref{Fig:alg}.  

\newpage
\begin{figure}
\centering       
\includegraphics[width=\textwidth]{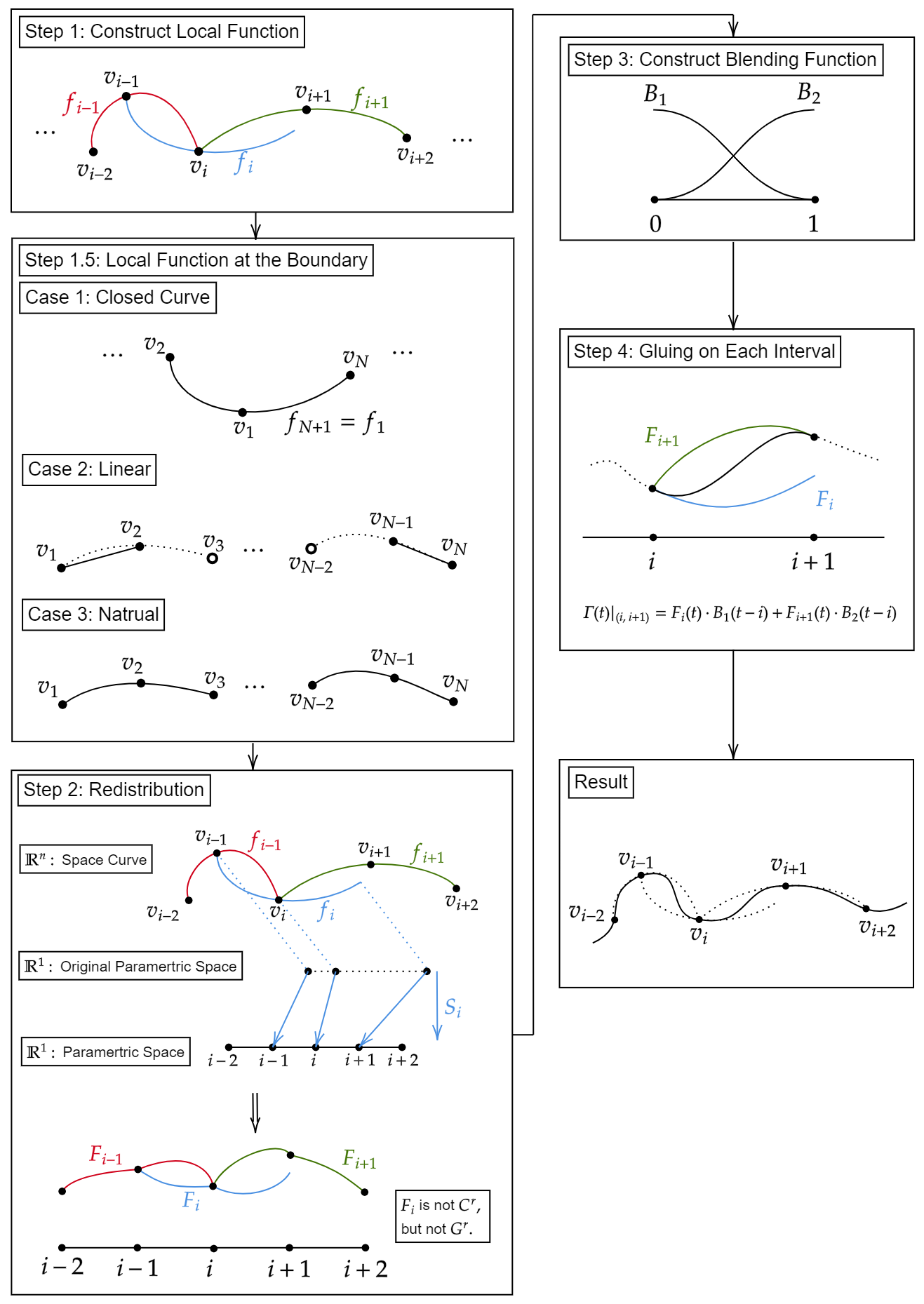}
\caption{Flowchart of our Algorithm \ref{alg: smoothcurve}.}\label{Fig:alg}
\end{figure}


\subsection{Computational Experiments}
Let us show some numerical results based on our Algorithm~\ref{alg: smoothcurve} above.
\begin{example}[$G^r$ Spline Interpolating Curves]\label{exp1:polyfinal}
Fix $r= 2$. 
We start with the $G^r$ interpolation with piecewise polynomial construction. 
We use the parabolic local function from Example ~\ref{ex:para}. We already proved that the local QR function is positive definite and contracted in Example ~\ref{exp:pottoreg}. Therefore, by theorem ~\ref{thm:Main} and Corollary ~\ref{Coro:BonusS}, the interpolation is $G^{r+1}$, where $r$ is only depending on the choice of $r$-blending function. We use the polynomial $r$-blending function defined in Section \ref{sec:blending}. The degree of $r$-blending function is $2r+1$. In this case, we can construct $G^r$ interpolation spline with piecewise polynomials of degree $2r+1$. In the case $r=1$, we use degree $3$ polynomial to interpolate $G^1$ data in any dimension. This result is quite comparable to $G^1$ Bernstein interpolation. Our algorithm automating the local curve so the user does not have to specify the local control points. In addition, our algorithm can increase the smoothness at a very low extra cost. Because the local curve is the same we only have to compute the blending function which the explicit formula of the function was given.\\
For a set of random points in 3d space, we use $2$-blending function to have $G^3$ smooth curve interpolation. As shown in Figure~\ref{fig:3dcurves}, every point in the original data set (red) is passed by the curve. On the other hand, We can calculate the curvature of the curve to show that the curve can achieve at least $G^2$, see Figure ~\ref{fig:3dC2} for the color pattern for the curvature. To show the smoothness of the spline interpolatory curve,  another example is given in Figure~\ref{fig:3dcurves} in tube view.\\
As we can see in Figure ~\ref{fig:3dC2}, the curvature is higher near data points, which is a desired feature. Some other methods like Circular spline will have curvature peaks in the middle between data points. 

\begin{figure}
    \centering
    \includegraphics[width=1 \textwidth, height=0.25\textwidth]{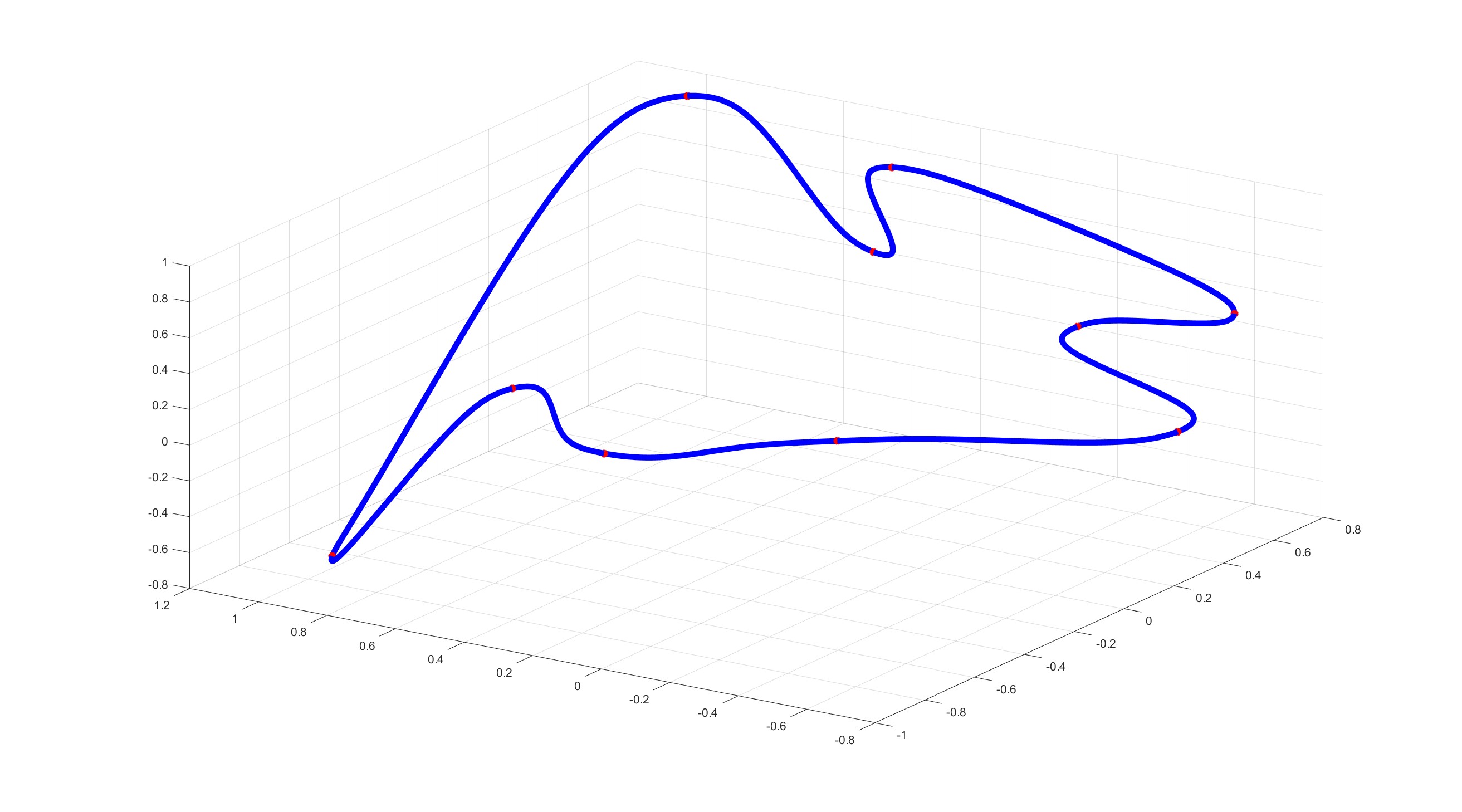}
    \includegraphics[width=1 \textwidth,height=0.25\textwidth]{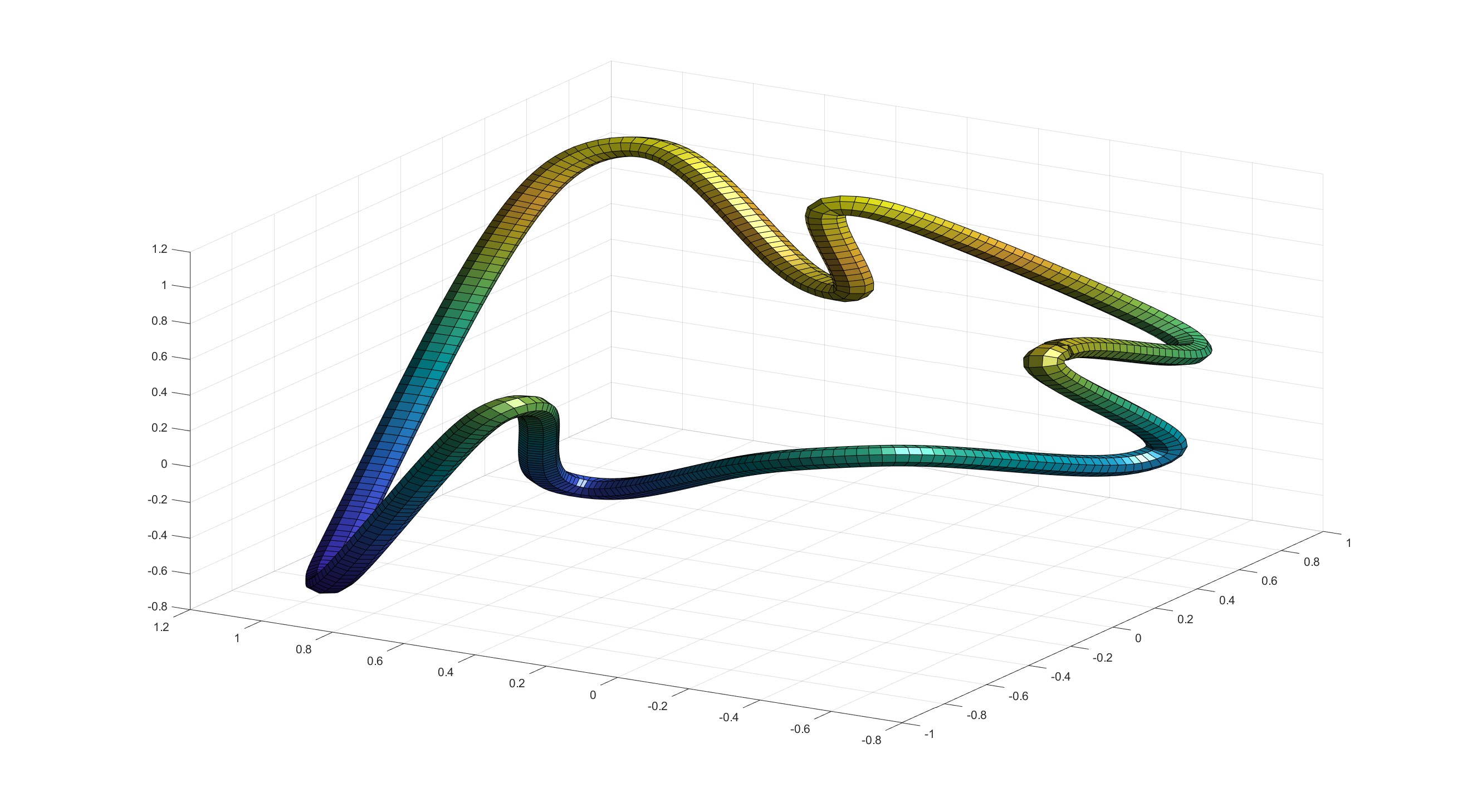}
    \caption{3D interpolatory curve (top) with the given points in red and it's in tube view (bottom)}
    \label{fig:3dcurves}
\end{figure}

Let us give two more examples in Figure ~\ref{fig:thankyou} and ~\ref{fig:curveofhead}.
\begin{figure}
    \centering
    \begin{tabular}{cc}
    \includegraphics[width=0.5\textwidth, height=0.5\textwidth]{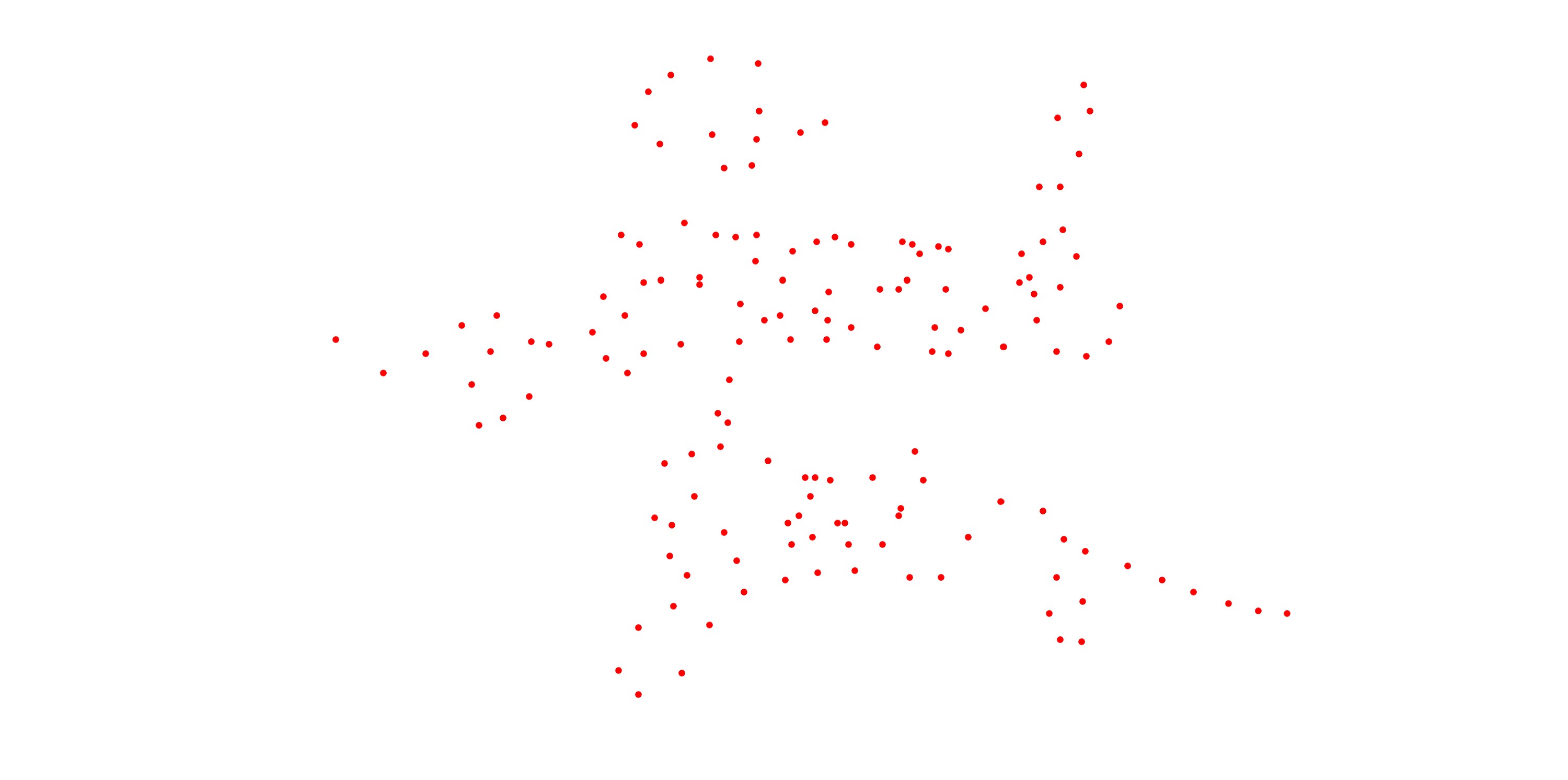}&
    \includegraphics[width=0.5\textwidth, height=0.5\textwidth]{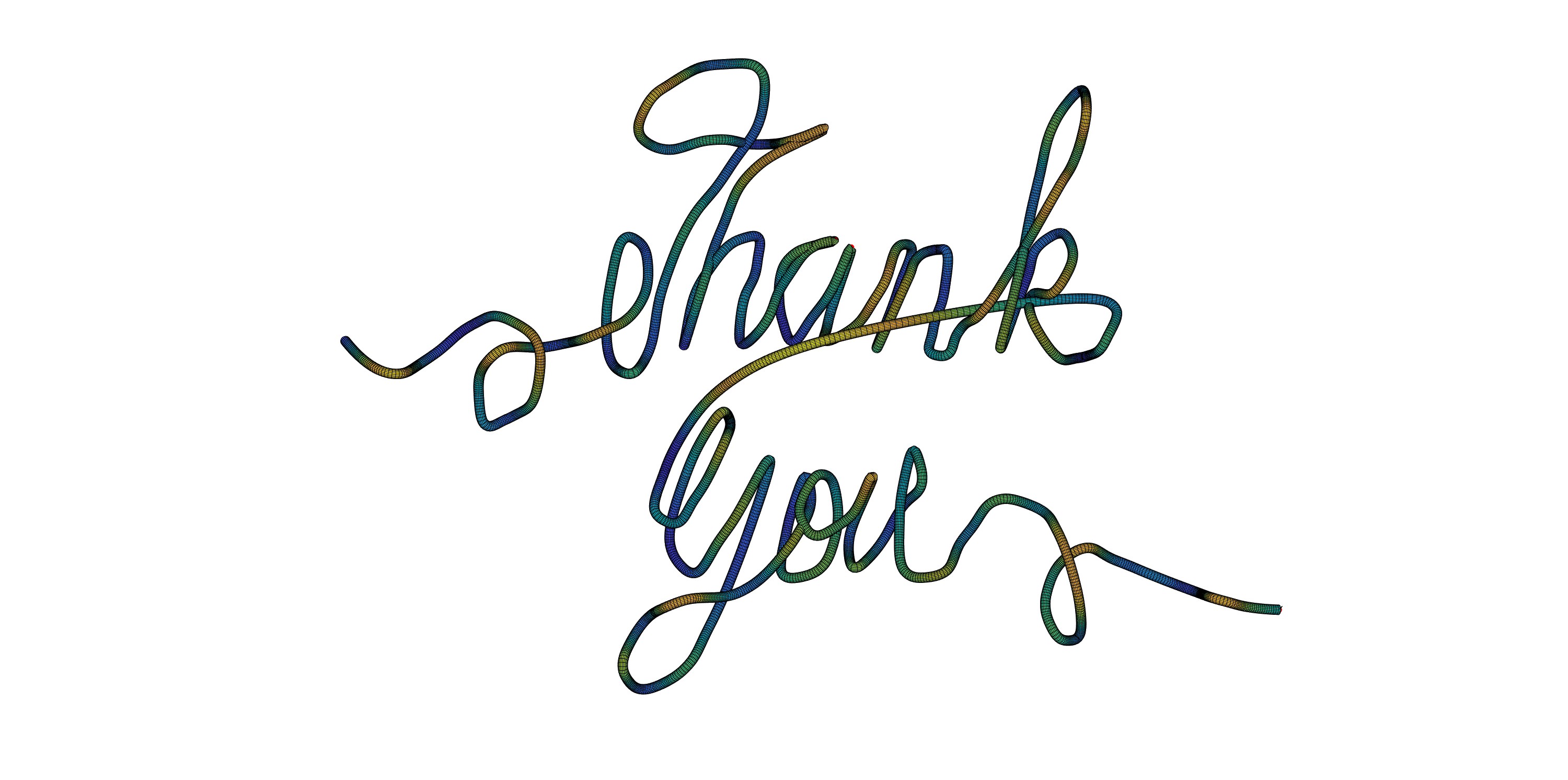}\\
    \end{tabular}
    \caption{A $C^2$ smooth interpolating curve (right) based on a given data set (left). }
    \label{fig:thankyou}
\end{figure}

\begin{figure}
     \begin{tabular}{ll}
    \includegraphics[width=0.5\textwidth, height=0.5\textwidth]{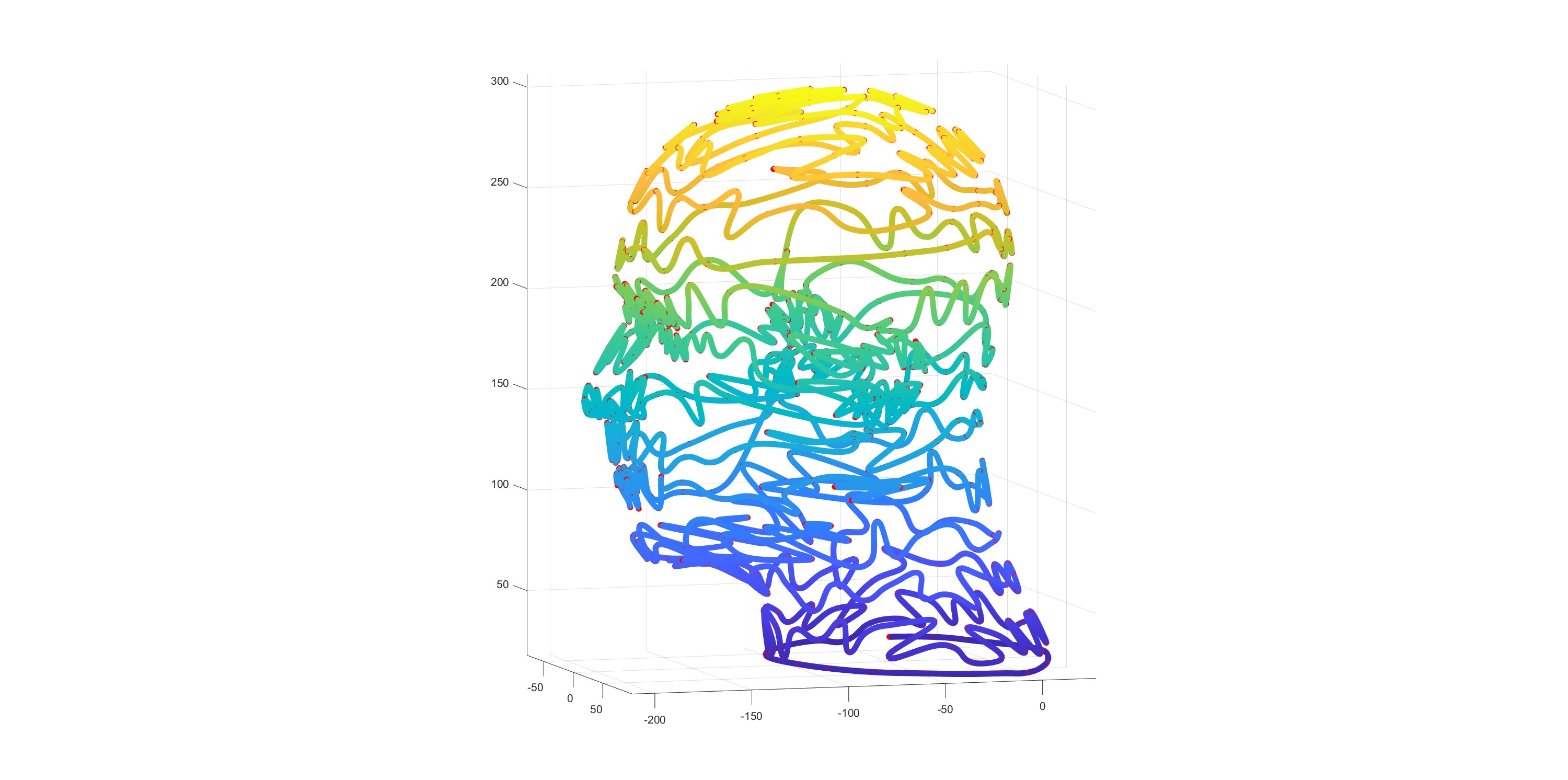}&
     \includegraphics[width=0.5\textwidth, height=0.5\textwidth]{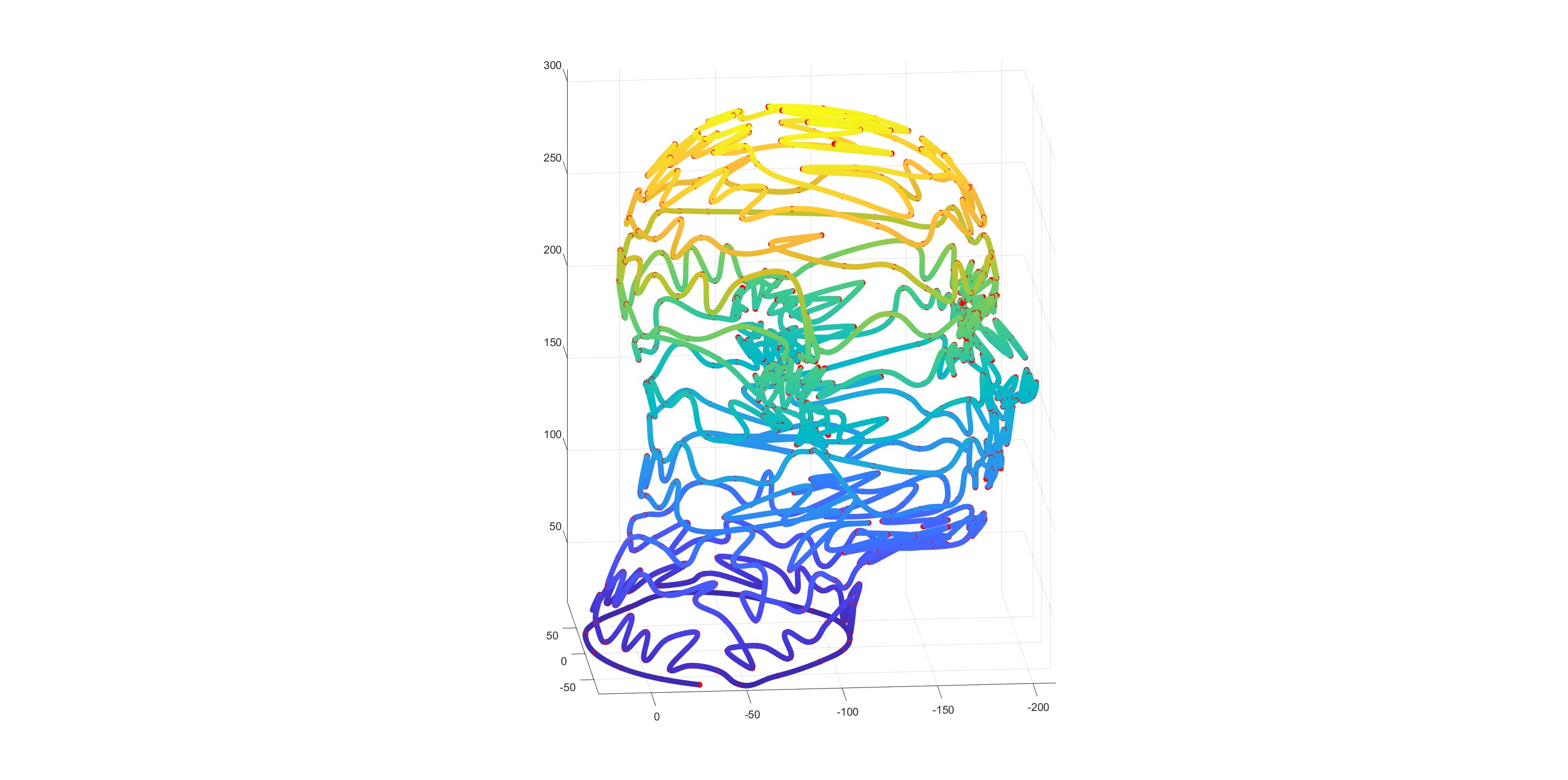}
      \end{tabular}
    \caption{Two views of smooth interpolating curve based on a data set of human head}
    \label{fig:curveofhead}
\end{figure}
\end{example}

\begin{example} [Corner Preservation]\label{exp2:Cornner}
On planar data points, the data points themselves are not always suitable for smooth interpolation. See Figure ~\ref{fig:corner}. In the Figure, we see that in order to interpolate the point at the corner, we must have a curve that is not a straight line. To address this feature, we have the following definition:
\begin{definition}[Lacol Convex and Degenerate Points]\label{def:covexanddeg}
    Consider a series of $5$ data points $\bfv_1, \bfv_2, \cdots, \bfv_5$ on $\mathbb{R}^2$, we say point $\bfv_3$ is a local convex point if $det(\bfv_2-\bfv_1, \bfv_3-\bfv_2)$, $det(\bfv_3-\bfv_2, \bfv_4-\bfv_3)$, and $det(\bfv_4-\bfv_3, \bfv_5-\bfv_4)$ are non-zeros and have same sign. We say $\bf_3$ is degenerated if at least one of these three determination are zero.
\end{definition}
  \begin{figure}
      \centering
\tikzset{every picture/.style={line width=0.75pt}} 

\begin{tikzpicture}[x=0.75pt,y=0.75pt,yscale=-1,xscale=1]

\draw    (338.93,111.25) .. controls (379,111.24) and (368,88.24) .. (388.78,111.58) ;
\draw    (270.92,154.71) .. controls (281,135.24) and (275,128.24) .. (289.08,110.92) ;
\draw  [dash pattern={on 0.84pt off 2.51pt}]  (289.08,110.92) -- (270.92,154.71) ;
\draw    (289.26,111.33) .. controls (306,87.24) and (310,111.24) .. (339.42,111) ;
\draw    (107.75,121.92) .. controls (118.44,109.33) and (105.78,88) .. (134.42,88) ;
\draw    (161.11,122.58) .. controls (151.11,110) and (164.44,88.67) .. (134.42,88) ;
\draw    (425.83,198.5) -- (407.67,154.71) ;
\draw  [dash pattern={on 0.84pt off 2.51pt}]  (389.5,110.92) -- (407.67,154.71) ;
\draw  [draw opacity=0][fill={rgb, 255:red, 0; green, 0; blue, 0 }  ,fill opacity=1 ] (430.08,198.5) .. controls (430.08,196.15) and (428.18,194.25) .. (425.83,194.25) .. controls (423.49,194.25) and (421.58,196.15) .. (421.58,198.5) .. controls (421.58,200.85) and (423.49,202.75) .. (425.83,202.75) .. controls (428.18,202.75) and (430.08,200.85) .. (430.08,198.5) -- cycle ;
\draw    (407.67,154.71) .. controls (397.58,135.24) and (403.58,128.24) .. (389.5,110.92) ;
\draw  [draw opacity=0][fill={rgb, 255:red, 0; green, 0; blue, 0 }  ,fill opacity=1 ] (411.92,154.71) .. controls (411.92,152.36) and (410.01,150.46) .. (407.67,150.46) .. controls (405.32,150.46) and (403.42,152.36) .. (403.42,154.71) .. controls (403.42,157.06) and (405.32,158.96) .. (407.67,158.96) .. controls (410.01,158.96) and (411.92,157.06) .. (411.92,154.71) -- cycle ;
\draw  [dash pattern={on 0.84pt off 2.51pt}]  (289.08,110.92) -- (388.78,111.58) ;
\draw  [dash pattern={on 4.5pt off 4.5pt}]  (161.11,122.58) -- (134.42,88) ;
\draw  [dash pattern={on 4.5pt off 4.5pt}]  (107.75,121.92) -- (134.44,87.33) ;
\draw  [draw opacity=0][fill={rgb, 255:red, 0; green, 0; blue, 0 }  ,fill opacity=1 ] (43.5,197.25) .. controls (43.5,194.9) and (45.4,193) .. (47.75,193) .. controls (50.1,193) and (52,194.9) .. (52,197.25) .. controls (52,199.6) and (50.1,201.5) .. (47.75,201.5) .. controls (45.4,201.5) and (43.5,199.6) .. (43.5,197.25) -- cycle ;
\draw  [draw opacity=0][fill={rgb, 255:red, 0; green, 0; blue, 0 }  ,fill opacity=1 ] (76.83,155.92) .. controls (76.83,153.57) and (78.74,151.67) .. (81.08,151.67) .. controls (83.43,151.67) and (85.33,153.57) .. (85.33,155.92) .. controls (85.33,158.26) and (83.43,160.17) .. (81.08,160.17) .. controls (78.74,160.17) and (76.83,158.26) .. (76.83,155.92) -- cycle ;
\draw  [draw opacity=0][fill={rgb, 255:red, 0; green, 0; blue, 0 }  ,fill opacity=1 ] (103.5,121.92) .. controls (103.5,119.57) and (105.4,117.67) .. (107.75,117.67) .. controls (110.1,117.67) and (112,119.57) .. (112,121.92) .. controls (112,124.26) and (110.1,126.17) .. (107.75,126.17) .. controls (105.4,126.17) and (103.5,124.26) .. (103.5,121.92) -- cycle ;
\draw  [draw opacity=0][fill={rgb, 255:red, 208; green, 2; blue, 27 }  ,fill opacity=1 ] (130.17,88) .. controls (130.17,85.65) and (132.07,83.75) .. (134.42,83.75) .. controls (136.76,83.75) and (138.67,85.65) .. (138.67,88) .. controls (138.67,90.35) and (136.76,92.25) .. (134.42,92.25) .. controls (132.07,92.25) and (130.17,90.35) .. (130.17,88) -- cycle ;
\draw  [draw opacity=0][fill={rgb, 255:red, 0; green, 0; blue, 0 }  ,fill opacity=1 ] (225.36,197.92) .. controls (225.36,195.57) and (223.46,193.67) .. (221.11,193.67) .. controls (218.76,193.67) and (216.86,195.57) .. (216.86,197.92) .. controls (216.86,200.26) and (218.76,202.17) .. (221.11,202.17) .. controls (223.46,202.17) and (225.36,200.26) .. (225.36,197.92) -- cycle ;
\draw  [draw opacity=0][fill={rgb, 255:red, 0; green, 0; blue, 0 }  ,fill opacity=1 ] (192.03,156.58) .. controls (192.03,154.24) and (190.12,152.33) .. (187.78,152.33) .. controls (185.43,152.33) and (183.53,154.24) .. (183.53,156.58) .. controls (183.53,158.93) and (185.43,160.83) .. (187.78,160.83) .. controls (190.12,160.83) and (192.03,158.93) .. (192.03,156.58) -- cycle ;
\draw  [draw opacity=0][fill={rgb, 255:red, 0; green, 0; blue, 0 }  ,fill opacity=1 ] (165.36,122.58) .. controls (165.36,120.24) and (163.46,118.33) .. (161.11,118.33) .. controls (158.76,118.33) and (156.86,120.24) .. (156.86,122.58) .. controls (156.86,124.93) and (158.76,126.83) .. (161.11,126.83) .. controls (163.46,126.83) and (165.36,124.93) .. (165.36,122.58) -- cycle ;
\draw    (107.75,121.92) -- (47.75,197.25) ;
\draw    (161.11,122.58) -- (221.11,197.92) ;
\draw  [draw opacity=0][fill={rgb, 255:red, 208; green, 2; blue, 27 }  ,fill opacity=1 ] (284.83,110.92) .. controls (284.83,108.57) and (286.74,106.67) .. (289.08,106.67) .. controls (291.43,106.67) and (293.33,108.57) .. (293.33,110.92) .. controls (293.33,113.26) and (291.43,115.17) .. (289.08,115.17) .. controls (286.74,115.17) and (284.83,113.26) .. (284.83,110.92) -- cycle ;
\draw  [draw opacity=0][fill={rgb, 255:red, 0; green, 0; blue, 0 }  ,fill opacity=1 ] (335.17,111) .. controls (335.17,108.65) and (337.07,106.75) .. (339.42,106.75) .. controls (341.76,106.75) and (343.67,108.65) .. (343.67,111) .. controls (343.67,113.35) and (341.76,115.25) .. (339.42,115.25) .. controls (337.07,115.25) and (335.17,113.35) .. (335.17,111) -- cycle ;
\draw  [draw opacity=0][fill={rgb, 255:red, 208; green, 2; blue, 27 }  ,fill opacity=1 ] (393.03,111.58) .. controls (393.03,109.24) and (391.12,107.33) .. (388.78,107.33) .. controls (386.43,107.33) and (384.53,109.24) .. (384.53,111.58) .. controls (384.53,113.93) and (386.43,115.83) .. (388.78,115.83) .. controls (391.12,115.83) and (393.03,113.93) .. (393.03,111.58) -- cycle ;
\draw  [draw opacity=0][fill={rgb, 255:red, 0; green, 0; blue, 0 }  ,fill opacity=1 ] (248.5,198.5) .. controls (248.5,196.15) and (250.4,194.25) .. (252.75,194.25) .. controls (255.1,194.25) and (257,196.15) .. (257,198.5) .. controls (257,200.85) and (255.1,202.75) .. (252.75,202.75) .. controls (250.4,202.75) and (248.5,200.85) .. (248.5,198.5) -- cycle ;
\draw  [draw opacity=0][fill={rgb, 255:red, 0; green, 0; blue, 0 }  ,fill opacity=1 ] (266.67,154.71) .. controls (266.67,152.36) and (268.57,150.46) .. (270.92,150.46) .. controls (273.26,150.46) and (275.17,152.36) .. (275.17,154.71) .. controls (275.17,157.06) and (273.26,158.96) .. (270.92,158.96) .. controls (268.57,158.96) and (266.67,157.06) .. (266.67,154.71) -- cycle ;
\draw    (252.75,198.5) -- (270.92,154.71) ;

\end{tikzpicture}

      \caption{The Points at the Corner of Two Straight Lines Can Not be Smooth without Adding Extra Structure. }
     \label{fig:corner}
  \end{figure}
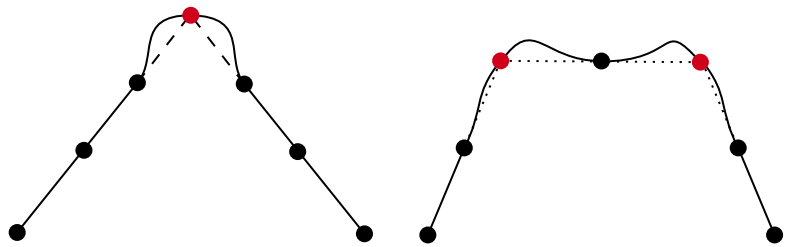

We can illustrate the local convex points in the Figure ~\ref{fig:CornorEvest}. In the figure, let ray $R_1$ start from $\bfv_2$ with direction $\bfv_1 \to \bfv_2$, and ray $R_2$ start from $\bfv_4$ with direction $\bfv_5 \to \bfv_4$. We can get a region by these two rays and line $\bfv_2-\bfv_4$ (See shading region in Figure ~\ref{fig:CornorEvest}). One can easily check that if $\bfv_3$ is located in the slash region, then $\bfv_3$ is a local convex point. If $\bfv_3$ is on the boundary of the slash region, then it's degenerated. One important case of degenerate points is when $\bfv_3$ is at the intersection $A$ of $R_1$ and $R_2$. The smooth interpolation curve will not be pleasing. We say $\bfv_3$ is a corner in this case. In Figure ~\ref{fig:corner}, the red points are corner points. We see that the interpolation curves lose local convexity when the original polygon is local convex. In practice, we can easily detect corner points and replace the local function to be piecewise linear. 
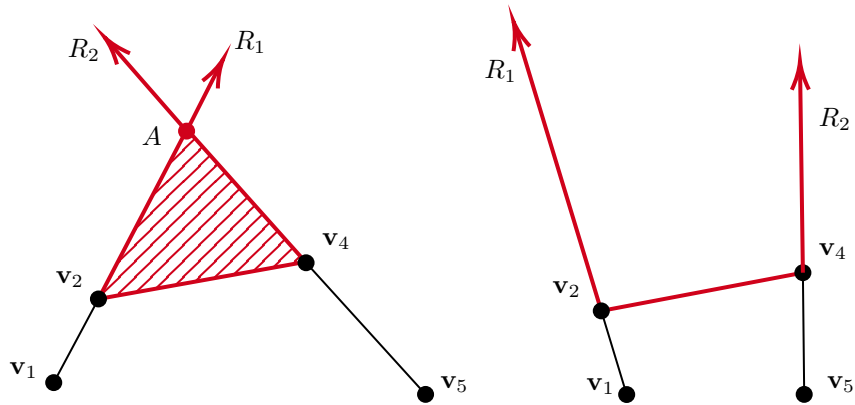
\begin{figure}
    \centering

 
\tikzset{
pattern size/.store in=\mcSize, 
pattern size = 5pt,
pattern thickness/.store in=\mcThickness, 
pattern thickness = 0.3pt,
pattern radius/.store in=\mcRadius, 
pattern radius = 1pt}
\makeatletter
\pgfutil@ifundefined{pgf@pattern@name@_yyd6e5nh6}{
\pgfdeclarepatternformonly[\mcThickness,\mcSize]{_yyd6e5nh6}
{\pgfqpoint{0pt}{0pt}}
{\pgfpoint{\mcSize+\mcThickness}{\mcSize+\mcThickness}}
{\pgfpoint{\mcSize}{\mcSize}}
{
\pgfsetcolor{\tikz@pattern@color}
\pgfsetlinewidth{\mcThickness}
\pgfpathmoveto{\pgfqpoint{0pt}{0pt}}
\pgfpathlineto{\pgfpoint{\mcSize+\mcThickness}{\mcSize+\mcThickness}}
\pgfusepath{stroke}
}}
\makeatother

 
\tikzset{
pattern size/.store in=\mcSize, 
pattern size = 5pt,
pattern thickness/.store in=\mcThickness, 
pattern thickness = 0.3pt,
pattern radius/.store in=\mcRadius, 
pattern radius = 1pt}
\makeatletter
\pgfutil@ifundefined{pgf@pattern@name@_b60i0ngfc}{
\pgfdeclarepatternformonly[\mcThickness,\mcSize]{_b60i0ngfc}
{\pgfqpoint{0pt}{0pt}}
{\pgfpoint{\mcSize+\mcThickness}{\mcSize+\mcThickness}}
{\pgfpoint{\mcSize}{\mcSize}}
{
\pgfsetcolor{\tikz@pattern@color}
\pgfsetlinewidth{\mcThickness}
\pgfpathmoveto{\pgfqpoint{0pt}{0pt}}
\pgfpathlineto{\pgfpoint{\mcSize+\mcThickness}{\mcSize+\mcThickness}}
\pgfusepath{stroke}
}}
\makeatother
\tikzset{every picture/.style={line width=0.75pt}} 

\begin{tikzpicture}[x=0.75pt,y=0.75pt,yscale=-1,xscale=1]

\draw  [draw opacity=0][pattern=_yyd6e5nh6,pattern size=6pt,pattern thickness=0.75pt,pattern radius=0pt, pattern color={rgb, 255:red, 208; green, 2; blue, 27}] (434.82,81.5) -- (435.46,156.71) -- (334.83,175.92) -- (306,80.42) -- cycle ;
\draw [color={rgb, 255:red, 208; green, 2; blue, 27 }  ,draw opacity=1 ][line width=1.5]    (334.83,175.92) -- (291.87,33.29) ;
\draw [shift={(291,30.42)}, rotate = 73.23] [color={rgb, 255:red, 208; green, 2; blue, 27 }  ,draw opacity=1 ][line width=1.5]    (14.21,-4.28) .. controls (9.04,-1.82) and (4.3,-0.39) .. (0,0) .. controls (4.3,0.39) and (9.04,1.82) .. (14.21,4.28)   ;
\draw [color={rgb, 255:red, 208; green, 2; blue, 27 }  ,draw opacity=1 ][line width=1.5]    (334.83,175.92) -- (435.46,156.71) ;
\draw  [color={rgb, 255:red, 208; green, 2; blue, 27 }  ,draw opacity=1 ][pattern=_b60i0ngfc,pattern size=6pt,pattern thickness=0.75pt,pattern radius=0pt, pattern color={rgb, 255:red, 208; green, 2; blue, 27}][line width=1.5]  (128.01,84.58) -- (187.14,151.27) -- (83.83,169.92) -- cycle ;
\draw    (61.5,212) -- (83.83,169.92) ;
\draw  [draw opacity=0][fill={rgb, 255:red, 0; green, 0; blue, 0 }  ,fill opacity=1 ] (57.25,212) .. controls (57.25,209.65) and (59.15,207.75) .. (61.5,207.75) .. controls (63.85,207.75) and (65.75,209.65) .. (65.75,212) .. controls (65.75,214.35) and (63.85,216.25) .. (61.5,216.25) .. controls (59.15,216.25) and (57.25,214.35) .. (57.25,212) -- cycle ;
\draw  [draw opacity=0][fill={rgb, 255:red, 0; green, 0; blue, 0 }  ,fill opacity=1 ] (79.58,169.92) .. controls (79.58,167.57) and (81.49,165.67) .. (83.83,165.67) .. controls (86.18,165.67) and (88.08,167.57) .. (88.08,169.92) .. controls (88.08,172.26) and (86.18,174.17) .. (83.83,174.17) .. controls (81.49,174.17) and (79.58,172.26) .. (79.58,169.92) -- cycle ;
\draw  [draw opacity=0][fill={rgb, 255:red, 208; green, 2; blue, 27 }  ,fill opacity=1 ] (123.57,85.5) .. controls (123.57,83.15) and (125.47,81.25) .. (127.82,81.25) .. controls (130.17,81.25) and (132.07,83.15) .. (132.07,85.5) .. controls (132.07,87.85) and (130.17,89.75) .. (127.82,89.75) .. controls (125.47,89.75) and (123.57,87.85) .. (123.57,85.5) -- cycle ;
\draw  [draw opacity=0][fill={rgb, 255:red, 0; green, 0; blue, 0 }  ,fill opacity=1 ] (251.36,217.92) .. controls (251.36,215.57) and (249.46,213.67) .. (247.11,213.67) .. controls (244.76,213.67) and (242.86,215.57) .. (242.86,217.92) .. controls (242.86,220.26) and (244.76,222.17) .. (247.11,222.17) .. controls (249.46,222.17) and (251.36,220.26) .. (251.36,217.92) -- cycle ;
\draw  [draw opacity=0][fill={rgb, 255:red, 0; green, 0; blue, 0 }  ,fill opacity=1 ] (191.71,151.71) .. controls (191.71,149.36) and (189.81,147.46) .. (187.46,147.46) .. controls (185.12,147.46) and (183.21,149.36) .. (183.21,151.71) .. controls (183.21,154.06) and (185.12,155.96) .. (187.46,155.96) .. controls (189.81,155.96) and (191.71,154.06) .. (191.71,151.71) -- cycle ;
\draw    (247.11,217.92) -- (187.46,151.71) ;
\draw [color={rgb, 255:red, 208; green, 2; blue, 27 }  ,draw opacity=1 ][line width=1.5]    (127.82,85.5) -- (144.65,52.1) ;
\draw [shift={(146,49.42)}, rotate = 116.75] [color={rgb, 255:red, 208; green, 2; blue, 27 }  ,draw opacity=1 ][line width=1.5]    (14.21,-4.28) .. controls (9.04,-1.82) and (4.3,-0.39) .. (0,0) .. controls (4.3,0.39) and (9.04,1.82) .. (14.21,4.28)   ;
\draw [color={rgb, 255:red, 208; green, 2; blue, 27 }  ,draw opacity=1 ][line width=1.5]    (127.82,85.5) -- (88.97,40.69) ;
\draw [shift={(87,38.42)}, rotate = 49.07] [color={rgb, 255:red, 208; green, 2; blue, 27 }  ,draw opacity=1 ][line width=1.5]    (14.21,-4.28) .. controls (9.04,-1.82) and (4.3,-0.39) .. (0,0) .. controls (4.3,0.39) and (9.04,1.82) .. (14.21,4.28)   ;
\draw    (347.5,218) -- (334.83,175.92) ;
\draw  [draw opacity=0][fill={rgb, 255:red, 0; green, 0; blue, 0 }  ,fill opacity=1 ] (343.25,218) .. controls (343.25,215.65) and (345.15,213.75) .. (347.5,213.75) .. controls (349.85,213.75) and (351.75,215.65) .. (351.75,218) .. controls (351.75,220.35) and (349.85,222.25) .. (347.5,222.25) .. controls (345.15,222.25) and (343.25,220.35) .. (343.25,218) -- cycle ;
\draw  [draw opacity=0][fill={rgb, 255:red, 0; green, 0; blue, 0 }  ,fill opacity=1 ] (330.58,175.92) .. controls (330.58,173.57) and (332.49,171.67) .. (334.83,171.67) .. controls (337.18,171.67) and (339.08,173.57) .. (339.08,175.92) .. controls (339.08,178.26) and (337.18,180.17) .. (334.83,180.17) .. controls (332.49,180.17) and (330.58,178.26) .. (330.58,175.92) -- cycle ;
\draw  [draw opacity=0][fill={rgb, 255:red, 0; green, 0; blue, 0 }  ,fill opacity=1 ] (440.36,217.92) .. controls (440.36,215.57) and (438.46,213.67) .. (436.11,213.67) .. controls (433.76,213.67) and (431.86,215.57) .. (431.86,217.92) .. controls (431.86,220.26) and (433.76,222.17) .. (436.11,222.17) .. controls (438.46,222.17) and (440.36,220.26) .. (440.36,217.92) -- cycle ;
\draw  [draw opacity=0][fill={rgb, 255:red, 0; green, 0; blue, 0 }  ,fill opacity=1 ] (439.71,156.71) .. controls (439.71,154.36) and (437.81,152.46) .. (435.46,152.46) .. controls (433.12,152.46) and (431.21,154.36) .. (431.21,156.71) .. controls (431.21,159.06) and (433.12,160.96) .. (435.46,160.96) .. controls (437.81,160.96) and (439.71,159.06) .. (439.71,156.71) -- cycle ;
\draw    (436.11,217.92) -- (435.46,156.71) ;
\draw [color={rgb, 255:red, 208; green, 2; blue, 27 }  ,draw opacity=1 ][line width=1.5]    (435.46,156.71) -- (434.04,53.42) ;
\draw [shift={(434,50.42)}, rotate = 89.21] [color={rgb, 255:red, 208; green, 2; blue, 27 }  ,draw opacity=1 ][line width=1.5]    (14.21,-4.28) .. controls (9.04,-1.82) and (4.3,-0.39) .. (0,0) .. controls (4.3,0.39) and (9.04,1.82) .. (14.21,4.28)   ;

\draw (38,200.4) node [anchor=north west][inner sep=0.75pt]    {$\bfv_{1}$};
\draw (61,154.4) node [anchor=north west][inner sep=0.75pt]    {$\bfv_{2}$};
\draw (194,135.4) node [anchor=north west][inner sep=0.75pt]    {$\bfv_{4}$};
\draw (253,207.4) node [anchor=north west][inner sep=0.75pt]    {$\bfv_{5}$};
\draw (103.91,81.86) node [anchor=north west][inner sep=0.75pt]    {$A$};
\draw (150,33.4) node [anchor=north west][inner sep=0.75pt]    {$R_{1}$};
\draw (67,37.4) node [anchor=north west][inner sep=0.75pt]    {$R_{2}$};
\draw (326,208.4) node [anchor=north west][inner sep=0.75pt]    {$\bfv_{1}$};
\draw (309,159.4) node [anchor=north west][inner sep=0.75pt]    {$\bfv_{2}$};
\draw (442,140.4) node [anchor=north west][inner sep=0.75pt]    {$\bfv_{4}$};
\draw (446,207.4) node [anchor=north west][inner sep=0.75pt]    {$\bfv_{5}$};
\draw (275,48.4) node [anchor=north west][inner sep=0.75pt]    {$R_{1}$};
\draw (442,72.4) node [anchor=north west][inner sep=0.75pt]    {$R_{2}$};

\end{tikzpicture}
    \caption{The degenerated $\bfv_3$ is located in the Edges of the slashed Area. The local convex point $\bfv_3$ will be in the interior of the slash area. If two rays $R_1$ and $R_2$ have an intersection point $A$ (As Shown on the Left), then $\bfv_3$ would be categorized as corner if it close to A.}
    \label{fig:CornorEvest}
\end{figure}

    We show an example of taking care of the corners of a given data set. Suppose that we have a set of data with four corners. If we use our method straightforwardly, our smooth interpolating curve will have rounded corners. See the left of
    Figure~\ref{corners}. When the code recognizes a point as a corner, it skips the local constructive step and produces the curve with corners by replacing the local curve with a piecewise linear function, as shown on the right of Figure~\ref{corners}.

  \begin{figure}
    \centering
    \includegraphics[width=0.8\textwidth]{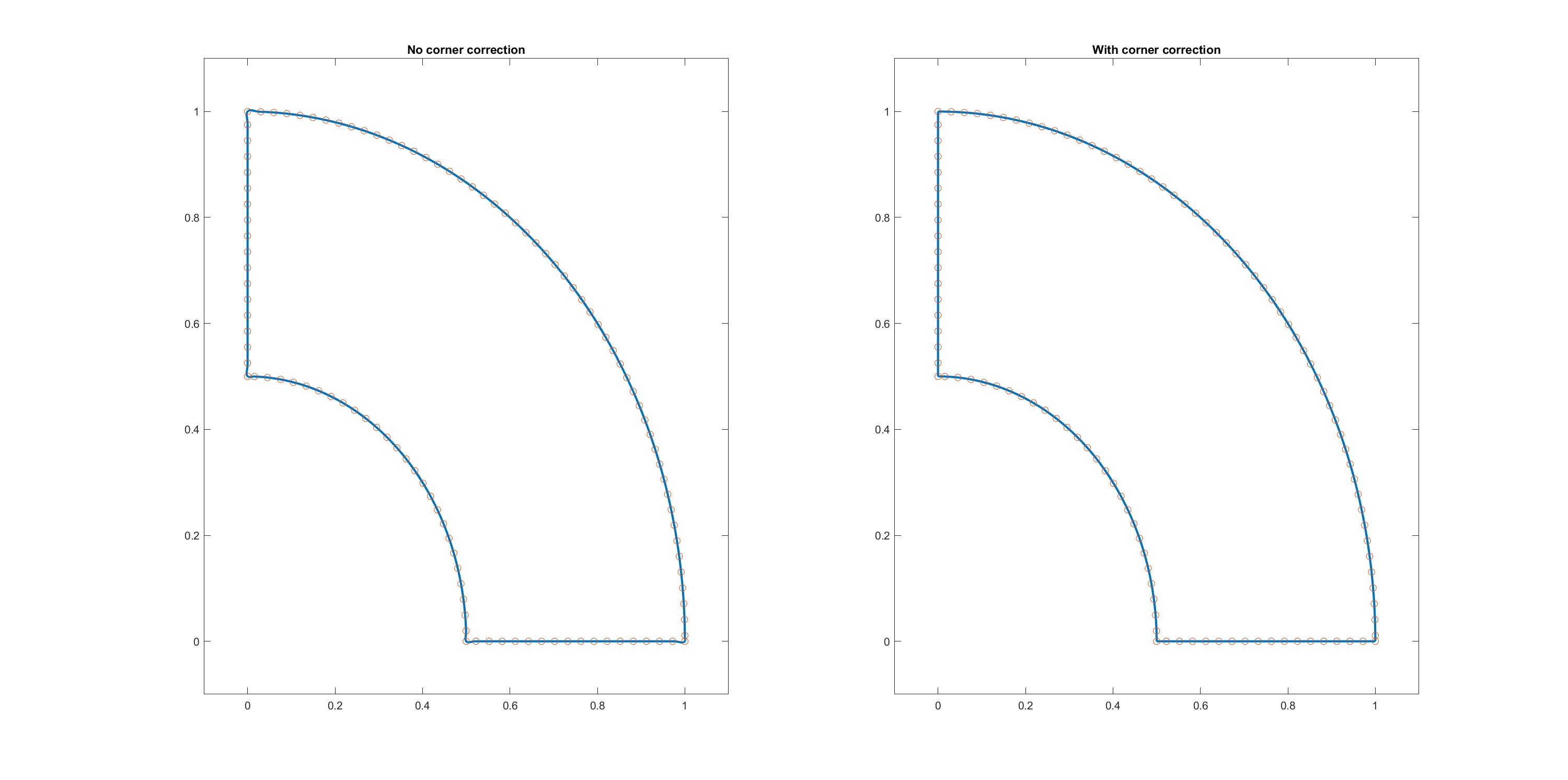}
    \includegraphics[width=0.8\textwidth]{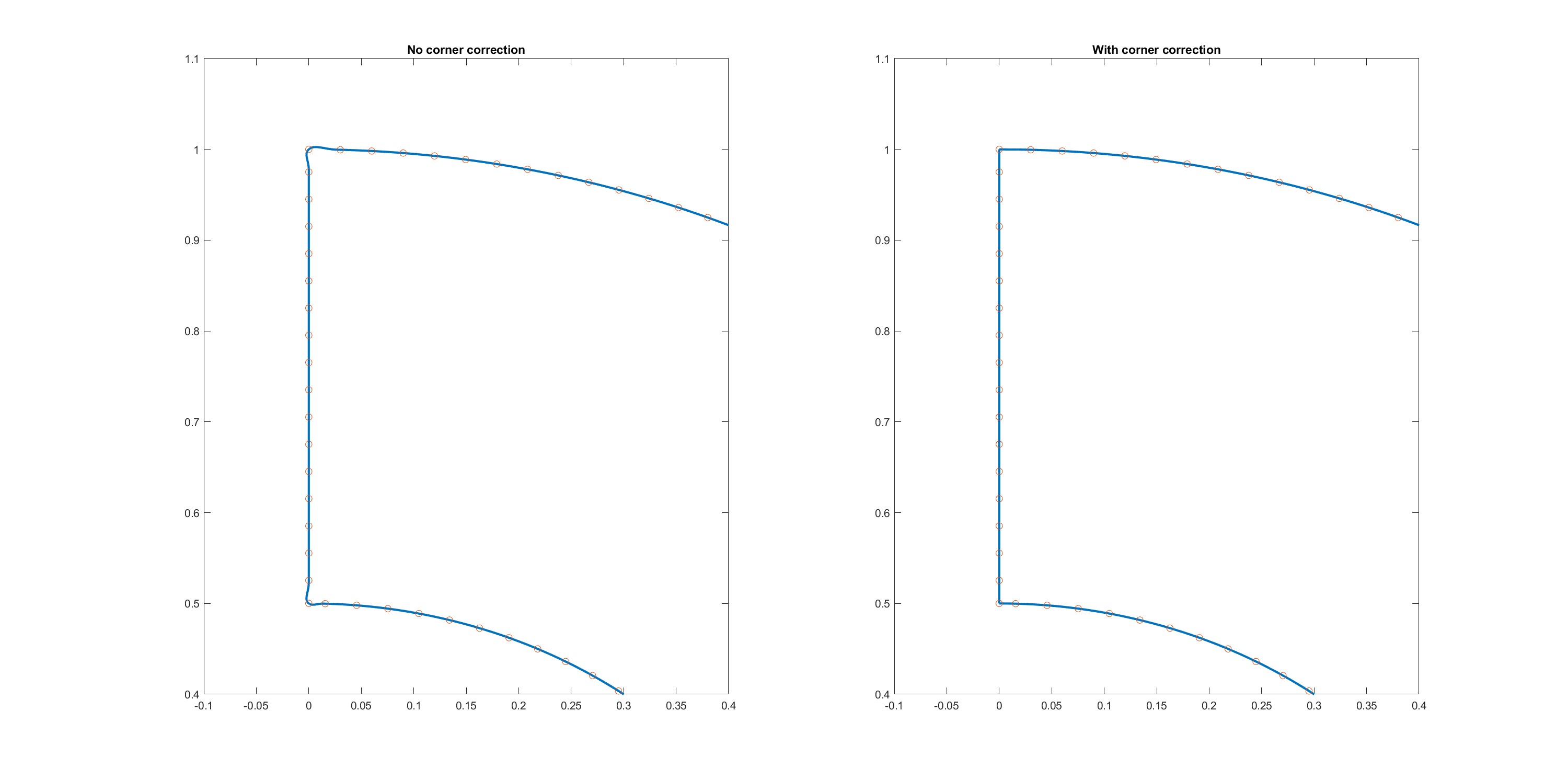}
    \caption{A smooth interpolation curve without and with corners}
    \label{corners}
\end{figure}  
    
\end{example}

\begin{example}[Convex Planar Data]\label{exp3:Convex}
We defined local convex points in Definition ~\ref{def:covexanddeg}. A natural question arises here: could we preserve the local convexity? That is, if $\bfv_i$ is a local convex point, could we have $\bfF_i$ such that $\bfF_i|_{[i-1,i+1]}$ do not change sign on its curvature? The answer is affirmative. We  can achieve this by choosing a linear local curve. Indeed, 
let $\bfv_1, \bfv_2, \cdots, \bfv_5$ be a series of  points and $\bfv_2$, $\bfv_3$, and $\bfv_4$ are local convex points. Suppose tangent line $\bfv_2-\bfv_1$, $\bfv_3-\bfv_2, \cdots , \bfv_5-\bfv_4$ are ordered clock-wisely as show in Figure ~\ref{fig:localLineseg}. Then there must be a line pass through $\bfv_3$ which one of its direction vectors is located between $\bfv_3-\bfv_2$ and $\bfv_4-\bfv_3$, shown as the blue line in Figure ~\ref{fig:localLineseg}. Because $\bfv_2$ and $\bfv_4$ are also local convex, we can find two similar lines pass $\bfv_2$ and $\bfv_4$, shown as orange lines in Figure ~\ref{fig:localLineseg}. The $P$ and $Q$ are intersections of the blue line and orange lines. The linear function from $P$ to $Q$ is defined as the local function in this example.\\

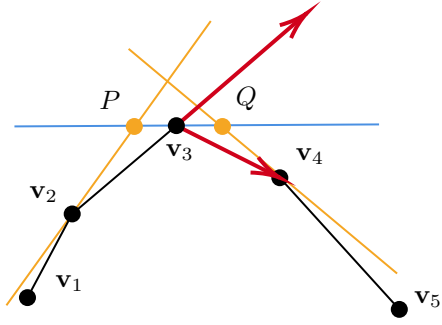
\begin{figure}
    \centering 
\tikzset{every picture/.style={line width=0.75pt}} 

\begin{tikzpicture}[x=0.75pt,y=0.75pt,yscale=-1,xscale=1]

\draw [color={rgb, 255:red, 245; green, 166; blue, 35 }  ,draw opacity=1 ]   (132,106.88) -- (266,219.88) ;
\draw [color={rgb, 255:red, 245; green, 166; blue, 35 }  ,draw opacity=1 ]   (71,235.88) -- (173,92.88) ;
\draw [color={rgb, 255:red, 74; green, 144; blue, 226 }  ,draw opacity=1 ]   (75,145.88) -- (257,144.88) ;
\draw    (81.5,232) -- (103.83,189.92) ;
\draw  [draw opacity=0][fill={rgb, 255:red, 0; green, 0; blue, 0 }  ,fill opacity=1 ] (77.25,232) .. controls (77.25,229.65) and (79.15,227.75) .. (81.5,227.75) .. controls (83.85,227.75) and (85.75,229.65) .. (85.75,232) .. controls (85.75,234.35) and (83.85,236.25) .. (81.5,236.25) .. controls (79.15,236.25) and (77.25,234.35) .. (77.25,232) -- cycle ;
\draw  [draw opacity=0][fill={rgb, 255:red, 0; green, 0; blue, 0 }  ,fill opacity=1 ] (99.58,189.92) .. controls (99.58,187.57) and (101.49,185.67) .. (103.83,185.67) .. controls (106.18,185.67) and (108.08,187.57) .. (108.08,189.92) .. controls (108.08,192.26) and (106.18,194.17) .. (103.83,194.17) .. controls (101.49,194.17) and (99.58,192.26) .. (99.58,189.92) -- cycle ;
\draw  [draw opacity=0][fill={rgb, 255:red, 0; green, 0; blue, 0 }  ,fill opacity=1 ] (271.36,237.92) .. controls (271.36,235.57) and (269.46,233.67) .. (267.11,233.67) .. controls (264.76,233.67) and (262.86,235.57) .. (262.86,237.92) .. controls (262.86,240.26) and (264.76,242.17) .. (267.11,242.17) .. controls (269.46,242.17) and (271.36,240.26) .. (271.36,237.92) -- cycle ;
\draw  [draw opacity=0][fill={rgb, 255:red, 0; green, 0; blue, 0 }  ,fill opacity=1 ] (211.71,171.71) .. controls (211.71,169.36) and (209.81,167.46) .. (207.46,167.46) .. controls (205.12,167.46) and (203.21,169.36) .. (203.21,171.71) .. controls (203.21,174.06) and (205.12,175.96) .. (207.46,175.96) .. controls (209.81,175.96) and (211.71,174.06) .. (211.71,171.71) -- cycle ;
\draw    (267.11,237.92) -- (207.46,171.71) ;
\draw [color={rgb, 255:red, 208; green, 2; blue, 27 }  ,draw opacity=1 ][line width=1.5]    (155.82,145.5) -- (204.79,170.35) ;
\draw [shift={(207.46,171.71)}, rotate = 206.91] [color={rgb, 255:red, 208; green, 2; blue, 27 }  ,draw opacity=1 ][line width=1.5]    (14.21,-4.28) .. controls (9.04,-1.82) and (4.3,-0.39) .. (0,0) .. controls (4.3,0.39) and (9.04,1.82) .. (14.21,4.28)   ;
\draw [color={rgb, 255:red, 208; green, 2; blue, 27 }  ,draw opacity=1 ][line width=1.5]    (155.82,145.5) -- (218.74,90.84) ;
\draw [shift={(221,88.88)}, rotate = 139.02] [color={rgb, 255:red, 208; green, 2; blue, 27 }  ,draw opacity=1 ][line width=1.5]    (14.21,-4.28) .. controls (9.04,-1.82) and (4.3,-0.39) .. (0,0) .. controls (4.3,0.39) and (9.04,1.82) .. (14.21,4.28)   ;
\draw  [draw opacity=0][fill={rgb, 255:red, 0; green, 0; blue, 0 }  ,fill opacity=1 ] (151.57,145.5) .. controls (151.57,143.15) and (153.47,141.25) .. (155.82,141.25) .. controls (158.17,141.25) and (160.07,143.15) .. (160.07,145.5) .. controls (160.07,147.85) and (158.17,149.75) .. (155.82,149.75) .. controls (153.47,149.75) and (151.57,147.85) .. (151.57,145.5) -- cycle ;
\draw    (103.83,189.92) -- (155.82,145.5) ;
\draw  [draw opacity=0][fill={rgb, 255:red, 245; green, 166; blue, 35 }  ,fill opacity=1 ] (130.58,145.92) .. controls (130.58,143.57) and (132.49,141.67) .. (134.83,141.67) .. controls (137.18,141.67) and (139.08,143.57) .. (139.08,145.92) .. controls (139.08,148.26) and (137.18,150.17) .. (134.83,150.17) .. controls (132.49,150.17) and (130.58,148.26) .. (130.58,145.92) -- cycle ;
\draw  [draw opacity=0][fill={rgb, 255:red, 245; green, 166; blue, 35 }  ,fill opacity=1 ] (174.58,145.92) .. controls (174.58,143.57) and (176.49,141.67) .. (178.83,141.67) .. controls (181.18,141.67) and (183.08,143.57) .. (183.08,145.92) .. controls (183.08,148.26) and (181.18,150.17) .. (178.83,150.17) .. controls (176.49,150.17) and (174.58,148.26) .. (174.58,145.92) -- cycle ;

\draw (94,219.4) node [anchor=north west][inner sep=0.75pt]    {$\bfv_{1}$};
\draw (81,174.4) node [anchor=north west][inner sep=0.75pt]    {$\bfv_{2}$};
\draw (214,155.4) node [anchor=north west][inner sep=0.75pt]    {$\bfv_{4}$};
\draw (273,227.4) node [anchor=north west][inner sep=0.75pt]    {$\bfv_{5}$};
\draw (149,153.4) node [anchor=north west][inner sep=0.75pt]    {$\bfv_{3}$};
\draw (116,126.4) node [anchor=north west][inner sep=0.75pt]    {$P$};
\draw (184,124.4) node [anchor=north west][inner sep=0.75pt]    {$Q$};

\end{tikzpicture}

    \caption{The local function is chosen to be the linear function from $P$ to $Q$.}
    \label{fig:localLineseg}
\end{figure}

We claim that if we use these local functions and the polynomial $r$-blending function, then the resulting curve is local convex as long as the data points are local convex. The proof is quite straightforward. We already claim that the  $r$-bending function is equivalent to Bernstein polynomials:
\begin{equation}
\begin{cases}
B_1(x)= \displaystyle \sum_{\tiny 0\leq i \leq r} b_{i,2r+1}(x)\\
B_2(x)= \displaystyle \sum_{\tiny r+1\leq i \leq 2r+1} b_{i,2r+1}(x).
\end{cases}
\end{equation}
Back to our example in Figure ~\ref{fig:localLineseg}. We shall show that the line segment from $\bfv_2$ to $\bfv_3$ is local convex. We know that on $[2,3]$, $\bfF_2=\bfv_2(1-t)+P \cdot t$ and $\bfF_3=P\cdot (1-t)+\bfv_3 \cdot t$. Combine with the formula of $r$ blending curve, one can write down the curve $\Gamma(t)|_{[2,3]}= \bfF_2(t) B_1 (t-2)+\bfF_3(t) B_2(t-2)$ by Bezier polynomial:
\begin{align*}
        \Gamma(t)|_{[2,3]}&= \bfv_2 \cdot b_{0,2r+2} + \displaystyle \sum_{\tiny 1\leq i \leq r} (\frac{2r+2-i}{2r+2}\bfv_2+\frac{i}{2r+2}P)b_{i,2r+1}(x)+P \cdot b_{r+1,2r+2}\\ 
&+\displaystyle \sum_{\tiny r+2 \leq i \leq 2r+1} (\frac{2r+2-i}{2r+2}P+\frac{i}{2r+2}\bfv_3)b_{i,2r+1}(x) +\bfv_3 \cdot b_{2r+2,2r+2},
\end{align*}
where $b_{i,j}$ are $i$-th Bezier base polynomial of degree $j$. As we can see, this line segment is a Bezier curve where all control points are located on the $\bfv_2$-$P$ and $P$-$\bfv_3$. The control points form a convex polygon. By variation diminishing property, the curve is a local convex curve.\\
\end{example}

\begin{example}[Sphere Preserving Curves]\label{exp4:sphere}

In this subsection, we present a few examples to show that when the given 
data points on spheres, our smooth interpolating curves will be on the
corresponding spheres. 
In Example ~\ref{sec: exp2 Arc} of Section ~\ref{Local_QRP}, we already explained the local curve of the arc. One of the advantages of using arcs as local curves is that the curve is locally a circle if the dataset itself is intended to behave like a circle, see Figure ~\ref{fig:circleExp}.\\

\begin{figure}
    \centering
    \includegraphics[width=0.8\textwidth]{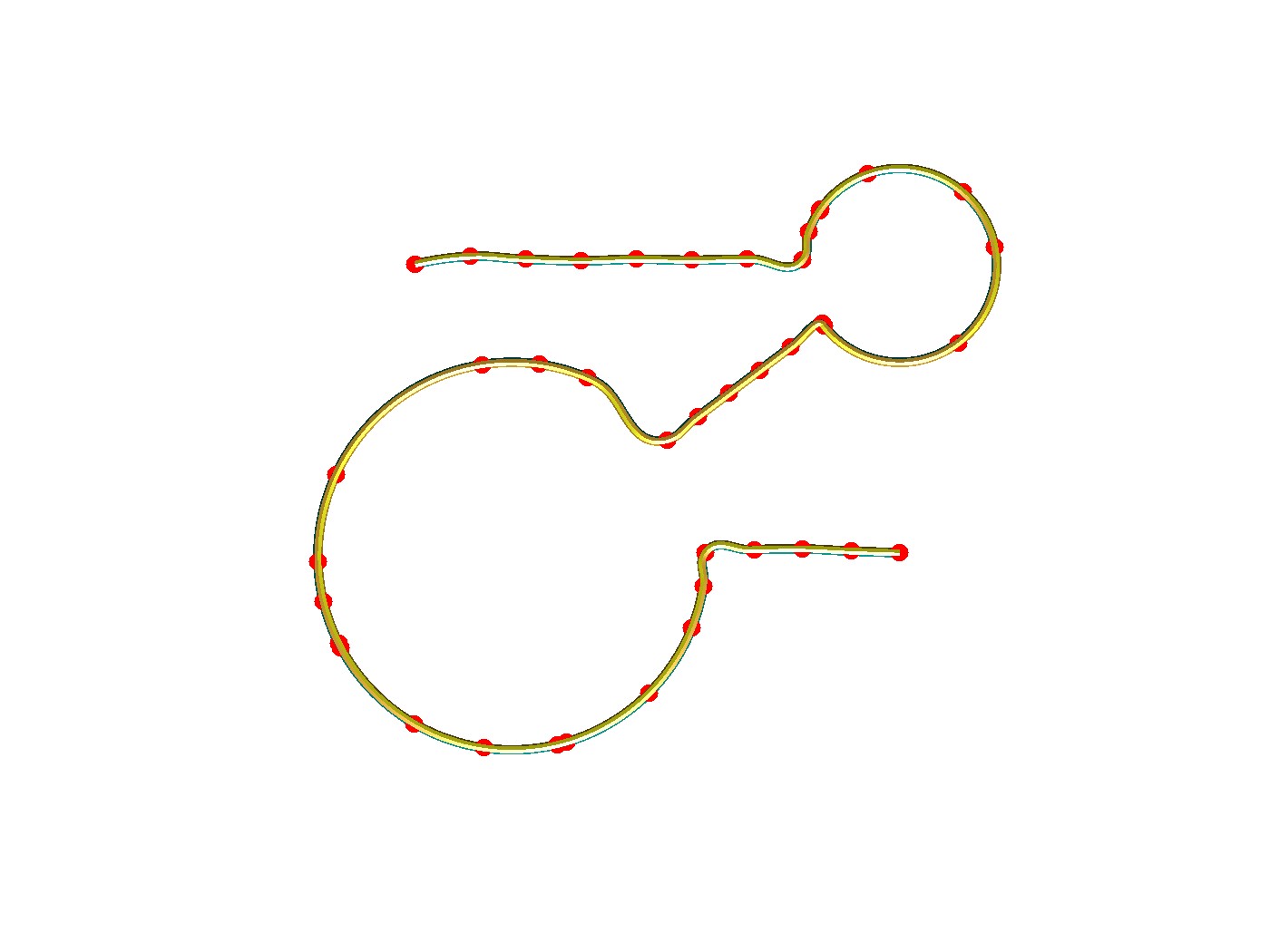}
    \caption{The curve is locally adaptive to circles when the data points are locally scattered on two circles.}
    \label{fig:circleExp}
\end{figure}

In the case of three-dimensional space, one may need to adapt a curve on spherical data. Because we use the circular local function, we can always expect our local curves to be located on the sphere when the data cloud is on the sphere. When we glue the two local curves on a sphere, all we have to do is to use a special gluing function. For instance, when we glue two arcs on a sphere from $P$ to $Q$, see Figure ~\ref{fig:Sphere_Blend}. We can  define $\Phi(s,x)$, where $\Phi(s,0)$ and $\Phi(s,1)$ are two arcs pass $P$ and $Q$ that we are going to glue. We can define $\bff_{x}(s)=\Phi(s,x)$ as a series of arcs on the sphere pass $P$ and $Q$ with uniform arc length so that $\bff_{x}(0)=P$ and $\bff_{x}(1)=Q$ for all $x \in [0,1]$. For the $x$ direction, one can easily find that the center of these arcs will be located on a circle inside the sphere. Therefore, we can parametrize the $x$ direction by the center location of these arcs. More specifically, let $C(x)$ be the center of the arc of $\bff_{x}(s)$ and we choose $|dC(x)/dx|$ to be a constant. In this way, we define the gluing function by $\Phi(s, B_1(s))$, where $B_1$ is the blending function.\\

\begin{figure}
    \centering
    \includegraphics[width=0.8\textwidth]{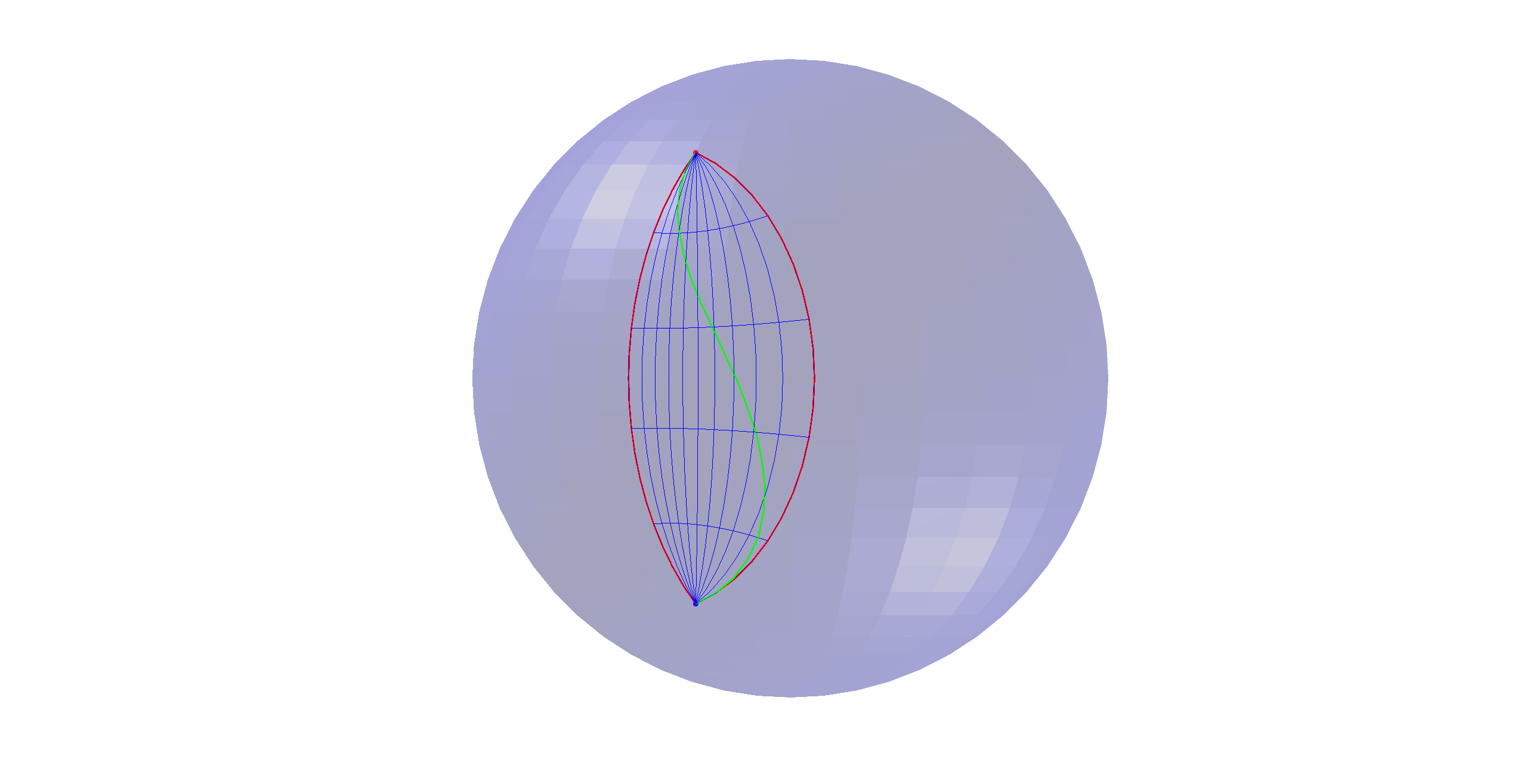}
    \caption{The grid defined by $\phi(x,t)$on the sphere. Two red curves are the arcs to be glued. And the green line is the resulting curve after gluing. }
    \label{fig:Sphere_Blend}
\end{figure}

As shown in Figure ~\ref{sphere}, two curves are presented to demonstrate that our 
interpolating curves will be on the corresponding spheres when the
given data points are on the spheres. We can see that the curve fits the sphere whenever the data points are on the sphere.\\

  \begin{figure}
    \centering
    \begin{tabular}{cc}
    \includegraphics[width=0.5\textwidth]{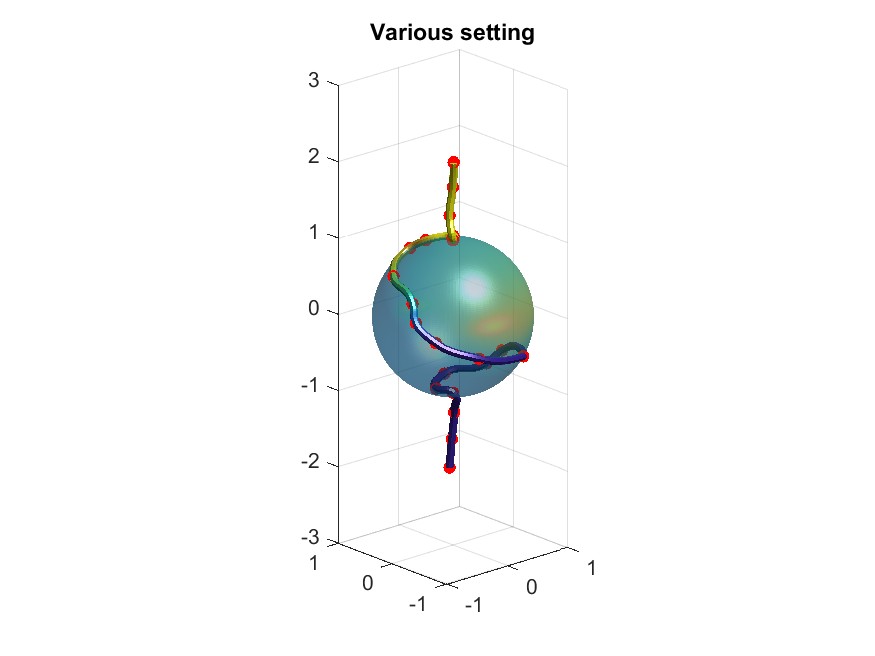} &
    \includegraphics[width=0.5\textwidth]{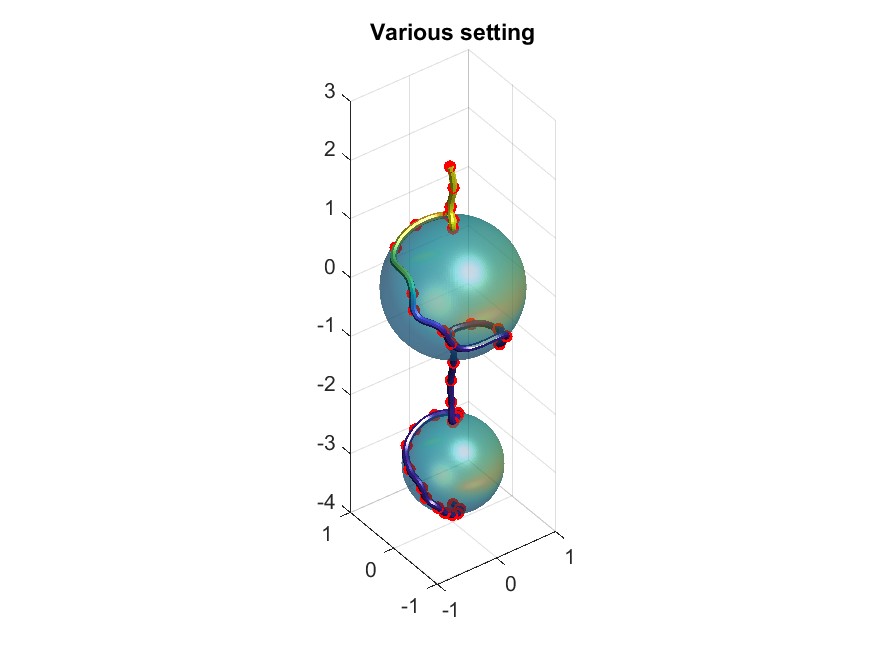}\\
    \end{tabular}
    \caption{Two examples of smooth interpolating curves which show that the curves will be on the sphere according to the given data points}
    \label{sphere}
\end{figure}

In the next example, we have a lot of points on the unit sphere. Our smooth
interpolating curve is also on the same sphere as shown in Figure~\ref{sphereD}.

  \begin{figure}
    \centering
    \includegraphics[width=1\textwidth]{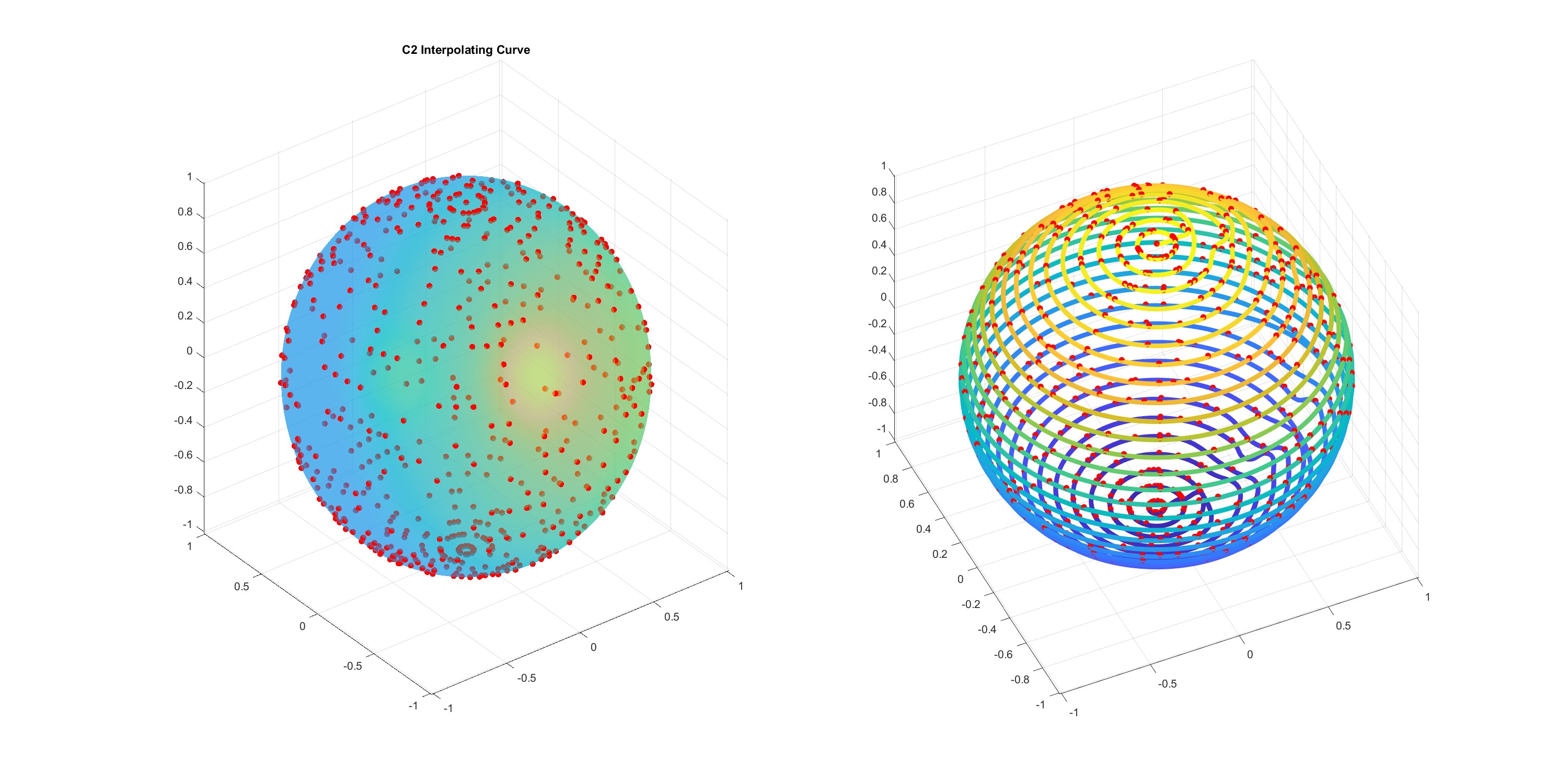}
    \caption{Our Smooth Interpolating Curve is on the Sphere (right) as the Given Data Points on the Sphere(left)}
    \label{sphereD}
\end{figure}  
\end{example}

\section{Conclusion and Future works}
In this paper, we present an organized framework for constructing smoothly interpolating curves in $\mathbb{R}^n$ with a high level of smoothness customization. While incorporating and refining existing approaches, our algorithm breaks down the smooth curve construction into four key components: the local function, blending function, distributing function, and gluing function.\\
The local function depends on the needs of shape design. One could use pure polynomials to simplify the calculation and save the storage, as shown in Example ~\ref{exp1:polyfinal}. In addition to the oriented points, some tangent vectors and normal vectors are
also given, we can modify our algorithm easily by 
using local functions with derivative interpolation  
properties. 
Some useful local functions like circular functions in Example ~\ref{sec: exp2 Arc} and linear functions to preserve the convexity in Example ~\ref{exp3:Convex} are also useful in various scenarios. Note that the local functions are not necessarily the same for all  points. For instance, on the boundary, one may apply different boundary local functions to get the desired curve. In Example ~\ref{exp2:Cornner}, we also show that by replacing the local function on the degenerated points, the curve can have the corner that we expected. In another example ~\ref{exp3:Convex}, we explained a convexity of the data points could be preserved if we use linear functions as local functions. The choice of local function is not limited to the example we have shown in this paper. 

The Blending functions are used to determine the smoothness of the final curves. Applying $r$-blending functions  can give us a $G^r$ curve in most cases. We also saw that we can use the trigonometric function or other functions to replace the polynomial blending functions if they satisfy Definition ~\ref{Def_of_rblending}.\\
The redistribution function maps the coordinate of the local function to a standard coordinate, which has one unit between each node. In this paper, we only use the piecewise linear map as our redistribution function. By Theorem \ref{thm:Main}, we may need our function $\bfF_i$ after redistribution to be positive definite and regular. For the same local function, we could make the redistributed function positive definite by choosing a suitable redistribution function. The only non positive definite $\bfF_i$ in this paper is the circular functions. However, Yuksel has proved that the curve is $G^2$ in his paper ~\cite{Yuksel2020}. So, using the piecewise linear function in this case is fine.\\
When linear gluing may be a good choice, most of the time, we may need a different way to glue the two curves together. In Example ~\ref{exp4:sphere}, we demonstrate a special way to glue two curves together so that the final curve will locate on the same sphere as its original two curves. In the example, the gluing function does not affect the final smoothness because it is $C^\infty$. If one uses other gluing functions, they may lose smoothness if the function is not smooth enough. \\

In this paper, we analyze previous works and decompose them into four parts of the function. Each part take there unique role in the construction of the final curve. We also indicated that the blending functions take an important role of smoothness. By providing  $r$-blending functions, we can now push the smoothness to any degree. We also give a necessary conditions that make the curve to be $G^r$ so when people use there preferred local function in our algorithm, they would expect a smooth curve as long as there formulate meets the condition.

\bmhead{Acknowledgments}
The authors are very grateful to the editor and referees for their helpful comments.

\section*{Declarations}


\begin{itemize}
\item \textbf{Funding.} The second author is supported by the Simons Foundation for collaboration grant \#864439. 
\item \textbf{Competing interests.} On behalf of all authors, the corresponding author states that there is no conflict of interest.
\end{itemize}

\addcontentsline{sn-bibliography}{chapter}{References} 
\bibliography{sn-bibliography}


\begin{thebibliography}{12}
\ifx \bisbn   \undefined \def \bisbn  #1{ISBN #1}\fi
\ifx \binits  \undefined \def \binits#1{#1}\fi
\ifx \bauthor  \undefined \def \bauthor#1{#1}\fi
\ifx \batitle  \undefined \def \batitle#1{#1}\fi
\ifx \bjtitle  \undefined \def \bjtitle#1{#1}\fi
\ifx \bvolume  \undefined \def \bvolume#1{\textbf{#1}}\fi
\ifx \byear  \undefined \def \byear#1{#1}\fi
\ifx \bissue  \undefined \def \bissue#1{#1}\fi
\ifx \bfpage  \undefined \def \bfpage#1{#1}\fi
\ifx \blpage  \undefined \def \blpage #1{#1}\fi
\ifx \burl  \undefined \def \burl#1{\textsf{#1}}\fi
\ifx \doiurl  \undefined \def \doiurl#1{\url{https://doi.org/#1}}\fi
\ifx \betal  \undefined \def \betal{\textit{et al.}}\fi
\ifx \binstitute  \undefined \def \binstitute#1{#1}\fi
\ifx \binstitutionaled  \undefined \def \binstitutionaled#1{#1}\fi
\ifx \bctitle  \undefined \def \bctitle#1{#1}\fi
\ifx \beditor  \undefined \def \beditor#1{#1}\fi
\ifx \bpublisher  \undefined \def \bpublisher#1{#1}\fi
\ifx \bbtitle  \undefined \def \bbtitle#1{#1}\fi
\ifx \bedition  \undefined \def \bedition#1{#1}\fi
\ifx \bseriesno  \undefined \def \bseriesno#1{#1}\fi
\ifx \blocation  \undefined \def \blocation#1{#1}\fi
\ifx \bsertitle  \undefined \def \bsertitle#1{#1}\fi
\ifx \bsnm \undefined \def \bsnm#1{#1}\fi
\ifx \bsuffix \undefined \def \bsuffix#1{#1}\fi
\ifx \bparticle \undefined \def \bparticle#1{#1}\fi
\ifx \barticle \undefined \def \barticle#1{#1}\fi
\bibcommenthead
\ifx \bconfdate \undefined \def \bconfdate #1{#1}\fi
\ifx \botherref \undefined \def \botherref #1{#1}\fi
\ifx \url \undefined \def \url#1{\textsf{#1}}\fi
\ifx \bchapter \undefined \def \bchapter#1{#1}\fi
\ifx \bbook \undefined \def \bbook#1{#1}\fi
\ifx \bcomment \undefined \def \bcomment#1{#1}\fi
\ifx \oauthor \undefined \def \oauthor#1{#1}\fi
\ifx \citeauthoryear \undefined \def \citeauthoryear#1{#1}\fi
\ifx \endbibitem  \undefined \def \endbibitem {}\fi
\ifx \bconflocation  \undefined \def \bconflocation#1{#1}\fi
\ifx \arxivurl  \undefined \def \arxivurl#1{\textsf{#1}}\fi
\csname PreBibitemsHook\endcsname

\bibitem[\protect\citeauthoryear{DeRose}{1985}]{D85}
\begin{botherref}
\oauthor{\bsnm{DeRose}, \binits{A.D.}}:
Geometric continuity: A parametrization independent measure of continuity for computer aided geometric design.
PhD thesis,
EECS Department, University of California, Berkeley
(Aug 1985).
\url{http://www2.eecs.berkeley.edu/Pubs/TechRpts/1985/6081.html}
\end{botherref}
\endbibitem

\bibitem[\protect\citeauthoryear{Gregory}{1989}]{gregory_geometric_1989}
\begin{bchapter}
\bauthor{\bsnm{Gregory}, \binits{J.A.}}:
\bctitle{Geometric {Continuity}}.
In: \beditor{\bsnm{Lyche}, \binits{T.}},
\beditor{\bsnm{Schumaker}, \binits{L.L.}} (eds.)
\bbtitle{Mathematical {Methods} in {Computer} {Aided} {Geometric} {Design}},
pp. \bfpage{353}--\blpage{371}.
\bpublisher{Academic Press}, \blocation{???}
(\byear{1989}).
\doiurl{10.1016/B978-0-12-460515-2.50028-7} .
\burl{https://www.sciencedirect.com/science/article/pii/B9780124605152500287}
Accessed 2022-07-31
\end{bchapter}
\endbibitem

\bibitem[\protect\citeauthoryear{Garrity and Warren}{1991}]{garrity_geometric_1991}
\begin{barticle}
\bauthor{\bsnm{Garrity}, \binits{T.}},
\bauthor{\bsnm{Warren}, \binits{J.}}:
\batitle{Geometric continuity}.
\bjtitle{Computer Aided Geometric Design}
\bvolume{8}(\bissue{1}),
\bfpage{51}--\blpage{65}
(\byear{1991})
\doiurl{10.1016/0167-8396(91)90049-H} .
Accessed 2022-07-31
\end{barticle}
\endbibitem

\bibitem[\protect\citeauthoryear{Farin}{1993}]{FARIN1993}
\begin{bbook}
\bauthor{\bsnm{Farin}, \binits{G.}}:
\bbtitle{Curves and Surfaces for Computer-Aided Geometric Design (Third Edition)},
\bedition{Third edition} edn.,
pp. \bfpage{229}--\blpage{249}.
\bpublisher{Academic Press},
\blocation{Boston}
(\byear{1993}).
\doiurl{10.1016/B978-0-12-249052-1.50019-2} .
\burl{https://www.sciencedirect.com/science/article/pii/B9780122490521500192}
\end{bbook}
\endbibitem

\bibitem[\protect\citeauthoryear{Les~Piegl}{1995}]{PT97}
\begin{bbook}
\bauthor{\bsnm{Les~Piegl}, \binits{W.T.}}:
\bbtitle{The NURBS Book}.
\bsertitle{Monographs in Visual Communication},
vol. \bseriesno{XIV, 646}.
\bpublisher{Springer},
\blocation{Heidelberg}
(\byear{1995})
\end{bbook}
\endbibitem

\bibitem[\protect\citeauthoryear{Wenz}{1996}]{WENZ1996}
\begin{barticle}
\bauthor{\bsnm{Wenz}, \binits{H.-J.}}:
\batitle{Interpolation of curve data by blended generalized circles}.
\bjtitle{Computer Aided Geometric Design}
\bvolume{13}(\bissue{8}),
\bfpage{673}--\blpage{680}
(\byear{1996})
\doiurl{10.1016/0167-8396(95)00054-2}
\end{barticle}
\endbibitem

\bibitem[\protect\citeauthoryear{Yuksel}{2020}]{Yuksel2020}
\begin{barticle}
\bauthor{\bsnm{Yuksel}, \binits{C.}}:
\batitle{A class of c2 interpolating splines}.
\bjtitle{ACM Transactions on Graphics}
\bvolume{39}(\bissue{5}),
\bfpage{160}--\blpage{116014}
(\byear{2020})
\doiurl{10.1145/3400301}
\end{barticle}
\endbibitem

\bibitem[\protect\citeauthoryear{Bertolazzi and Frego}{2018}]{BERTOLAZZI2018}
\begin{barticle}
\bauthor{\bsnm{Bertolazzi}, \binits{E.}},
\bauthor{\bsnm{Frego}, \binits{M.}}:
\batitle{On the g2 hermite interpolation problem with clothoids}.
\bjtitle{Journal of Computational and Applied Mathematics}
\bvolume{341},
\bfpage{99}--\blpage{116}
(\byear{2018})
\doiurl{10.1016/j.cam.2018.03.029}
\end{barticle}
\endbibitem

\bibitem[\protect\citeauthoryear{Knott}{2000}]{Knott2000}
\begin{bbook}
\bauthor{\bsnm{Knott}, \binits{G.D.}}:
\bbtitle{Function and Space Curve Interpolation},
pp. \bfpage{59}--\blpage{61}.
\bpublisher{Birkh{\"a}user Boston},
\blocation{Boston, MA}
(\byear{2000}).
\doiurl{10.1007/978-1-4612-1320-8_4} .
\burl{https://doi.org/10.1007/978-1-4612-1320-8_4}
\end{bbook}
\endbibitem

\bibitem[\protect\citeauthoryear{Lee}{1989}]{LEE1989363}
\begin{barticle}
\bauthor{\bsnm{Lee}, \binits{E.T.Y.}}:
\batitle{Choosing nodes in parametric curve interpolation}.
\bjtitle{Computer-Aided Design}
\bvolume{21}(\bissue{6}),
\bfpage{363}--\blpage{370}
(\byear{1989})
\doiurl{10.1016/0010-4485(89)90003-1}
\end{barticle}
\endbibitem

\bibitem[\protect\citeauthoryear{Carmo}{2016}]{carmo_differential_2016}
\begin{bbook}
\bauthor{\bsnm{Carmo}, \binits{M.P.d.}}:
\bbtitle{Differential {Geometry} of {Curves} and {Surfaces}: {Revised} and {Updated} {Second} {Edition}}.
\bpublisher{Courier Dover Publications}, \blocation{???}
(\byear{2016}).
\bcomment{Google-Books-ID: gg2xDQAAQBAJ}
\end{bbook}
\endbibitem

\bibitem[\protect\citeauthoryear{Hu}{2023}]{HUTSUNG-WEI2023CMfS}
\begin{botherref}
\oauthor{\bsnm{Hu}, \binits{T.-W.}}:
Computational methods for smooth mapping and interpolation using multivariate spline techniques.
PhD thesis,
University of Georgia
(2023)
\end{botherref}
\endbibitem

\end{thebibliography}


\end{document}